\documentclass[a4paper,12pt]{article}

\usepackage[T1]{fontenc}
\usepackage[utf8]{inputenc}
\usepackage[english]{babel}
\usepackage[inline,shortlabels]{enumitem}
\usepackage{amsmath}
\usepackage{amssymb}
\usepackage{amsthm}
\usepackage[noadjust]{cite}
\usepackage[a4paper,total={6in,9in}]{geometry}
\usepackage[colorlinks,citecolor=blue,linkcolor=blue]{hyperref}
\usepackage[ruled,vlined,linesnumbered,algo2e,algosection]{algorithm2e} % Before cleveref
\usepackage[capitalize,noabbrev,nameinlink]{cleveref}
\usepackage{xcolor}
\usepackage{longtable}
\usepackage{tikz,pgfplots,pgfplotstable}
\pgfplotsset{compat=1.16}

\setlist{font=\normalfont,topsep=1ex,parsep=0ex}

\usepackage{etoolbox}
\makeatletter
\patchcmd{\@algocf@start}{\addtolength{\hsize}{-1.5em}}{}{}{}

\renewenvironment{algomathdisplay}
 {\[}
 {\@endalgocfline\vspace{-\baselineskip}\]\hfill\strut\par}
\makeatother
\SetKw{KwOr}{or}

\newcommand\ccell[1]{\multicolumn{1}{c}{#1}} % Centered cell

\theoremstyle{plain}
\newtheorem{proposition}{Proposition}[section]
\newtheorem{lemma}[proposition]{Lemma}
\newtheorem{theorem}[proposition]{Theorem}

\newtheorem{corollary}[proposition]{Corollary}

\theoremstyle{definition}
\newtheorem{definition}[proposition]{Definition}
\newtheorem{example}[proposition]{Example}
\newtheorem{setting}[proposition]{Setting}

\theoremstyle{remark}
\newtheorem{remark}[proposition]{Remark}

\newcommand{\R}{\mathbb{R}}
\newcommand{\N}{\mathbb{N}}
\newcommand{\W}{\mathbb{W}}
\newcommand{\cone}{\operatorname{cone}}
\newcommand{\Span}{\operatorname{span}}
\newcommand{\rank}{\operatorname{rank}}
\newcommand{\epi}{\operatorname{epi}}
\newcommand{\trace}{\operatorname{trace}}
\newcommand{\diag}{\operatorname{diag}}

\newcommand{\shift}{r}

\setlist[enumerate]{label=(\alph*)}

\DeclareMathOperator*{\argmin}{\operatorname{argmin}}

\numberwithin{equation}{section}
\numberwithin{table}{section}    
\numberwithin{figure}{section}

\crefname{figure}{Figure}{Figures}
\crefname{table}{Table}{Tables}

\let\eqref\labelcref
\usepackage{mathtools}

\DeclarePairedDelimiter\abs{\lvert}{\rvert}
\DeclarePairedDelimiter\norm{\lVert}{\rVert}

\providecommand\given{\nonscript\;\delimsize|\nonscript\;}
\DeclarePairedDelimiterX\set[1]\{\}{#1}

\DeclarePairedDelimiterX\innerp[2]{\langle}{\rangle}{#1,#2}

\DeclarePairedDelimiter\parens()
\DeclarePairedDelimiter\bracks[]

\definecolor{todocolor}{rgb}{1.0,0.0,0.0}

\expandafter\patchcmd\csname\string\proof\endcsname{\itshape}{\bfseries}{}{}

\newcommand{\mscLink}[1]{\href{https://www.ams.org/mathscinet/msc/msc2020.html?t=#1}{#1}}

\newif\ifpreprint
\preprinttrue

\title{
	An Augmented Lagrangian Method for Optimization Problems with
	Structured Geometric Constraints%
		\thanks{This research was 
			supported by the German Research Foundation (DFG) within the priority 
			program ``Non-smooth and Complementarity-based Distributed Parameter Systems: 
			Simulation and Hierarchical Optimization'' 
			(SPP 1962) under grant numbers KA 1296/24-2 and WA 3636/4-2.
			}
	}

\author{Xiaoxi Jia\footnotemark[2]
	\hspace*{-4mm} \and \hspace*{-4mm}
	Christian Kanzow\thanks{%
		University of Würzburg,
		Institute of Mathematics,
		97074 Würzburg,
		Germany,
		email: \{xiaoxi.jia,kanzow\}@mathematik.uni-wuerzburg.de,
	}
	\hspace*{-4mm} \and \hspace*{-4mm}
	Patrick Mehlitz\footnotemark[4]
	\thanks{%
		University of Mannheim,
		School of Business Informatics and Mathematics,
		68159 Mannheim,
		Germany
	} 
	\hspace*{-4mm} \and \hspace*{-4mm}
	Gerd Wachsmuth\thanks{%
		Brandenburgische Technische Universität Cottbus-Senftenberg,
		Institute of Mathematics,
		03046 Cottbus,
		Germany,
		email: \{mehlitz,wachsmuth\}@b-tu.de.
	} 
}

\date{\today}

\begin{document}

\maketitle
{
	\small\textbf{\abstractname.}
		This paper is devoted to the theoretical and numerical investigation of an
		augmented Lagrangian method for the solution of optimization
		problems with geometric constraints. Specifically, we study
		situations where parts of the constraints are nonconvex and possibly
        complicated, but allow for a fast computation of
		projections onto this nonconvex set. Typical problem classes which satisfy this requirement are
		optimization problems with disjunctive constraints (like complementarity or
		cardinality constraints) as well as optimization problems
		over sets of matrices which have to satisfy additional rank constraints.
		The key idea behind our method is to keep these complicated
		constraints explicitly in the constraints and to penalize only
		the remaining constraints by an augmented Lagrangian function.
		The resulting subproblems are then solved with the aid
        of a problem-tailored nonmonotone projected gradient method. 
        The corresponding convergence theory allows for an inexact solution 
        of these subproblems. Nevertheless, the overall algorithm computes so-called 
        Mordukhovich-stationary points of the original problem under a mild 
        asymptotic regularity condition, which is generally weaker than 
        most of the respective available problem-tailored constraint qualifications.
		Extensive numerical experiments addressing complementarity- and cardinality-constrained
		optimization problems as well as a semidefinite reformulation of MAXCUT problems
		visualize the power of our approach.
	\par\addvspace{\baselineskip}
}

{
	\small\textbf{Keywords.}
		Asymptotic Regularity,
		Augmented Lagrangian Method,
		Cardinality Constraints,
		Complementarity Constraints,
		MAXCUT Problem,
		Mordukhovich-Stationarity,
		Nonmonotone Projected Gradient Method
	\par\addvspace{\baselineskip}
}

{
	\small\textbf{AMS subject classifications.}
	\mscLink{49J53}, \mscLink{65K10}, \mscLink{90C22}, \mscLink{90C30}, \mscLink{90C33}
	\par\addvspace{\baselineskip}
}

\section{Introduction}\label{Sec:Intro}

We consider the program
\begin{equation}\label{Eq:GenP}\tag{P}
   \min\limits_w \ f(w) \quad \text{s.t.} \quad G(w) \in C, \ w \in D ,
\end{equation}
where $\mathbb W$ and $\mathbb Y$ are Euclidean spaces, i.e., real and finite-dimensional Hilbert spaces,
$ f\colon \mathbb{W} \to \mathbb{R} $ and $ G\colon \mathbb{W} \to \mathbb{Y} $
are continuously differentiable, $ C \subset \mathbb{Y} $ is nonempty, closed, 
and convex, whereas the set $ D \subset \mathbb{W} $
is only assumed to be nonempty and closed. This setting is very general
and covers, amongst others, standard nonlinear programs, second-order cone
and, more generally, conic optimization problems 
\cite{BenTalNemirovski2001,Chen2019},
as well as several so-called disjunctive programming problems like mathematical programs with complementarity, vanishing, switching, or cardinality constraints, see 
\cite{BenkoCervinkaHoheisel,BenkoGfrerer2018,FlegelKanzowOutrata2007,Mehlitz2019b} 
for an overview and suitable references. Since $\mathbb W$ and $\mathbb Y$ are Euclidean spaces, our model also covers matrix optimization problems like
semidefinite programs or low-rank approximation problems \cite{Markovsky2012}.

The aim of this paper is to apply a (structured) augmented Lagrangian 
technique to \eqref{Eq:GenP} in order to find suitable stationary points.
The augmented Lagrangian or multiplier penalty method is a classical
approach for the solution of nonlinear programs, see \cite{Bertsekas1982}
as a standard reference. The more recent book \cite{BirginMartinez2014} presents
a slightly modified version of this classical augmented Lagrangian method,
which uses a safeguarded update of the Lagrange multipliers and has stronger
global convergence properties. In the meantime, this safeguarded augmented
Lagrangian method has also been applied
to a number of optimization problems with disjunctive constraints, see e.g.\
\cite{AndreaniHaeserSecchinSilva2019,GuoDeng2021,IzmailovSolodovUskov2012,KanzowRaharjaSchwartz2021,Ramos2019}.

Since, to the best of our knowledge, augmented Lagrangian methods have not yet
been applied to the general problem \eqref{Eq:GenP} with general nonconvex $D$
and arbitrary convex sets $ C $ in the setting of Euclidean spaces, and in order to get a
better understanding of our contributions, let us add some comments regarding
the existing results for the probably most prominent non-standard optimization
problem, namely the class of mathematical programs with complementarity
constraints (MPCCs). Due to the particular structure of the feasible set, 
the usual Karush--Kuhn--Tucker (KKT for short) conditions are typically not 
satisfied at a local minimum.
Hence, other (weaker) stationarity concepts have been proposed, like C- (abbreviating Clarke) and 
M- (for Mordukhovich) stationarity, with M-stationarity being the stronger concept. Most algorithms
(regularization, penalty, augmented Lagrangian methods etc.)
for the solution of MPCCs solve a sequence of standard nonlinear programs, and their
limit points are typically C-stationary points only. Some approaches can identify 
M-stationary points if the underlying nonlinear programs are solved exactly,
but they loose this desirable property if these programs are solved only
inexactly, see the discussion in \cite{KanzowSchwartz2015} for more details.

The authors are currently aware of only three approaches where
convergence to M-stationary points for a general (nonlinear) MPCC is shown
using inexact solutions of the corresponding subproblems,
namely \cite{AndreaniSecchinSilva2018,Ramos2019,GuoDeng2021}. All three papers 
deal with suitable modifications of the (safeguarded) augmented Lagrangian 
method. The basic idea of reference \cite{AndreaniSecchinSilva2018} is to 
solve the subproblems such that both a first- and a second-order
necessary optimality condition hold inexactly at each iteration,
i.e., satisfaction of the second-order condition is the central point here
which, obviously, causes some overhead for the subproblem solver and usually
excludes the application of this approach to large-scale problems. The 
paper \cite{Ramos2019} proves convergence to M-stationary points by
solving some complicated subproblems, but for the latter no method is specified. 
Finally, the recent
approach described in \cite{GuoDeng2021} provides an augmented Lagrangian
technique for the solution of MPCCs where the complementarity constraints 
are kept as constraints, whereas the standard constraints are penalized.
The authors present a technique which computes a suitable stationary
point of these subproblems in such a way that the entire method generates
M-stationary accumulation points for the original MPCC.
Let us also mention that \cite{HarderMehlitzWachsmuth2021} suggests to solve (a discontinuous
reformulation of) the M-stationarity system associated with an MPCC by means of
a semismooth Newton-type method. Naturally, this approach should
be robust with respect to (w.r.t.) an inexact solution of the appearing
Newton-type equations although this issue is not discussed in \cite{HarderMehlitzWachsmuth2021}.

The present paper universalizes the idea from \cite{GuoDeng2021} to the
much more general problem \eqref{Eq:GenP}. In fact, a closer look at
the corresponding proofs shows that the technique from \cite{GuoDeng2021}
can be generalized using some relatively small modifications.
This
allows us to concentrate on some additional new contributions. In
particular, we prove convergence to an M-type stationary point of
the general problem \eqref{Eq:GenP} under a very weak sequential
constraint qualification introduced recently in \cite{Mehlitz2020} for the general
setting from \eqref{Eq:GenP}. We further show that this sequential
constraint qualification holds under the conditions for which convergence
to M-stationary points of an MPCC is shown in \cite{GuoDeng2021}. 
Note that this is also the first algorithmic application of the general sequential
stationarity and regularity concepts from \cite{Mehlitz2020}.

The global convergence result for our method holds for the abstract problem
\eqref{Eq:GenP} with geometric constraints without any further assumptions
regarding the sets $ C $ and, in particular, $ D $. Conceptually, we
are therefore able to deal with a very large class of optimization problems.
On the other hand, we use a projected gradient-type method for the solution
of the resulting subproblems. Since this requires projections onto the
(usually nonconvex) set $ D $, our method can be implemented efficiently only 
if $ D $ is simple in the sense that projections onto $ D $ are easy
to compute. For this kind of ``structured'' geometric constraints (this
explains the title of this paper), the entire method is then both an efficient
tool and applicable to large-scale problems. In particular, we show that 
this is the case for MPCCs, optimization problems with cardinality constraints,
and some rank-constrained matrix optimization problems.

The paper is organized as follows. We begin with restating some basic
definitions from variational analysis in \cref{Sec:Prelims}. 
There, we also relate the general regularity concept from \cite{Mehlitz2020}
to the constraint qualification (the so-called relaxed constant positive linear dependence 
condition, RCPLD for short) used in the underlying paper \cite{GuoDeng2021}
(as well as in many other related publications in this area). We then 
present the spectral gradient method for optimization problems over nonconvex
sets in \cref{Sec:ProjGrad}. This method is used to solve the 
resulting subproblems of our augmented Lagrangian method whose details
are given in \cref{Sec:ALM}. Global convergence to M-type
stationary points is also shown in this section. Since, in our augmented 
Lagrangian approach, we penalize the seemingly easy constraints $ G(w) \in C $,
but keep the condition $ w \in D $ explicitly in the constraints, we
have to compute projections onto $ D $. \Cref{Sec:Realizations}
therefore considers a couple of situations where this can be done in a numerically
very efficient way. Extensive computational experiments for some of these 
situations are documented in \cref{Sec:Numerics}. This includes
MPCCs, cardinality-constrained (sparse) optimization problems, and 
a rank-constrained reformulation of the famous MAXCUT problem. We close
with some final remarks in \cref{Sec:Final}.

Notation. The Euclidean inner product of two vectors $x,y\in\R^n$ will be denoted by $x^\top y$.
More generally, $ \innerp{x}{y} $ is used to represent the inner product of
$x,y\in\mathbb W$ whenever $\mathbb W$ is some abstract Euclidean space.
For brevity, we exploit $x+A:=A+x:=\set{x+a \given a\in A}$ for arbitrary vectors $x\in\mathbb W$
and sets $A\subset\mathbb W$.
The sets $\cone A$ and $\Span A$ denote the smallest cone containing the set $A$ and the
smallest subspace containing $A$, respectively.
Whenever $L\colon\mathbb W\to\mathbb Y$ is a linear operator between Euclidean spaces
$\mathbb W$ and $\mathbb Y$, $L^*\colon\mathbb Y\to\mathbb W$ denotes its adjoint.
For some continuously differentiable mapping $\varphi\colon\mathbb W\to\mathbb Y$ and some
point $w\in\mathbb W$, we use $\varphi'(w)\colon\mathbb W\to\mathbb Y$ in order to denote the derivative of $\varphi$
at $w$ which is a continuous linear operator. 
In the particular case $\mathbb Y:=\R$, we set $\nabla\varphi(w):=\varphi'(w)^*1\in\mathbb W$ for brevity.

\section{Preliminaries}\label{Sec:Prelims}

We first recall some basic concepts from variational analysis in 
\cref{sub:VA}, and then introduce and discuss general stationarity and regularity
concepts for the abstract problem \eqref{Eq:GenP} in \cref{sec:asymptotic_stat_reg}.

\subsection{Fundamentals of Variational Analysis}\label{sub:VA}

In this section, we comment on the tools of variational analysis which will be exploited
in order to describe the 
geometry of the closed, convex set $C\subset\mathbb Y$ and the
closed (but not necessarily convex) set $D\subset\mathbb W$ which appear in the formulation
of \eqref{Eq:GenP}.

The Euclidean projection $P_C \colon \mathbb Y \to \mathbb Y$ onto the closed, convex set $ C $ is given by
\begin{equation*}
   P_C (y) := \argmin\limits_{z \in C} \norm{ z - y }.
\end{equation*}
Thus, the corresponding distance function $d_C \colon \mathbb Y \to \R$ can be written as
\begin{equation*}
   d_C (y) := \min_{z \in C} \norm{ z - y } = \norm{ P_C(y) - y }.
\end{equation*}
On the other hand, projections onto the potentially nonconvex set $ D $ still exist,
but are, in general, not unique. Therefore, we define the corresponding (usually set-valued)
projection operator $\Pi_D \colon \mathbb W \rightrightarrows \mathbb W$ by
\begin{equation*}
   \Pi_D (x) := \argmin\limits_{z \in D} \norm{ z - x } \ne \varnothing.
\end{equation*}
Given $ \bar w \in D $, the closed cone
\begin{equation*}
   \mathcal{N}_D^\textup{lim} (\bar w) 
   := 
   \limsup\limits_{w\to \bar w}\bracks[\big]{\cone(w-\Pi_D(w))}
\end{equation*}
is referred to as the limiting normal cone to $D$ at $\bar w$,
see \cite{Mordukhovich2018,RockafellarWets2009} for other representations and
properties of this variational tool.
Above, we used the notion of the outer (or upper) limit of a set-valued mapping at
a certain point, see e.g.\ \cite[Definition~4.1]{RockafellarWets2009}. 
For $w\notin D$, we set $\mathcal N_D^\textup{lim}(w):=\varnothing$.
Note that the limiting normal cone
depends on the inner product of $\mathbb W$
and
is stable in the sense that
\begin{equation}\label{Eq:stability}
   \limsup\limits_{w \to \bar w} \mathcal{N}_D^\textup{lim} (w) 
   =
   \mathcal{N}_D^\textup{lim} (\bar w) 
	 \qquad \forall \bar w \in \mathbb W
\end{equation}
holds. 
This stability property, which might be referred to as outer semicontinuity of the set-valued
operator $\mathcal N_D^\textup{lim}\colon\mathbb W\rightrightarrows\mathbb W$, 
will play an essential role in our subsequent analysis.
The limiting normal cone to the convex set $C$
coincides with the standard normal cone from convex analysis, 
i.e., for $\bar y\in C$, we have
\begin{equation*}
  \mathcal{N}_C^\textup{lim} (\bar y) = \mathcal{N}_C (\bar y) :=
   \set*{ \lambda \in \mathbb Y \given
   	\innerp{\lambda}{y-\bar y} \leq 0 \quad \forall y \in C 
   	}.
\end{equation*}
For points $y\notin C$, we set $\mathcal N_C(y):=\varnothing$ for formal completeness.
Note that the stability property \eqref{Eq:stability}
is also satisfied by the set-valued operator $\mathcal N_C\colon\mathbb Y\rightrightarrows\mathbb Y$.

\subsection{Stationarity and Regularity Concepts}\label{sec:asymptotic_stat_reg}

Noting that the abstract set $D$ is generally nonconvex in the exemplary
settings we have in mind, the so-called concept of Mordukhovich-stationarity, which
exploits limiting normals to $D$, is a reasonable concept of stationarity which addresses
\eqref{Eq:GenP}. 

\begin{definition}\label{Def:MStat}
Let $ \bar w \in \mathbb W $ be feasible for the optimization problem \eqref{Eq:GenP}. Then 
$ \bar w $ is called an \emph{M-stationary point} (Mordukhovich-stationary point) 
of \eqref{Eq:GenP} if there exists a multiplier $ \lambda \in \mathbb{Y} $
such that 
\begin{equation*}
    0 \in \nabla f(\bar w) + G'(\bar w)^* \lambda + \mathcal{N}_D^\textup{lim} (\bar w),
    \quad 
    \lambda \in \mathcal{N}_C ( G(\bar w) ).
\end{equation*}
\end{definition}

Note that this definition coincides with the usual KKT conditions 
of \eqref{Eq:GenP} if the set $ D $ is convex. 
An asymptotic counterpart of this definition is the following one,
see \cite{Mehlitz2020}.

\begin{definition}\label{Def:AMStat}
 Let $ \bar w \in \mathbb W $ be feasible for the optimization problem \eqref{Eq:GenP}. Then 
$ \bar w $ is called an \emph{AM-stationary point} (asymptotically M-stationary point) 
of \eqref{Eq:GenP} if there exist sequences 
$\{ w^k \},\{\varepsilon^k\}\subset\mathbb W$ and $\{\lambda^k\},\{ z^k \} \subset\mathbb Y$ such that
$w^k\to\bar w$, $\varepsilon^k\to 0$, $z^k\to 0$, as well as
\begin{equation*}
   \varepsilon^k \in \nabla f(w^k) + G'(w^k)^* \lambda^k + \mathcal N_D^\textup{lim} (w^k),
   \quad\lambda^k \in \mathcal{N}_C( G(w^k) - z^k)
   \qquad
   \forall k\in\N.
\end{equation*}
\end{definition}

The definition of an AM-stationary point is similar to the notion of an 
AKKT (asymptotic or approximate KKT) point in standard nonlinear programming, 
see \cite{BirginMartinez2014}, but 
requires some explanation: The meanings of the iterates $ w^k $ and the 
Lagrange multiplier estimates $ \lambda^k $ should be clear. The vector
$ \varepsilon^k $ measures the inexactness by which the stationary
conditions are satisfied at $ w^k $ and $ \lambda^k $. The vector $ z^k $
does not occur (at least not explicitly) in the context of standard nonlinear
programs, but is required here for the following reason: The method to
be considered in this paper generates a sequence $ \{ w^k \} $ satisfying
$ w^k \in D $, while the constraint $ G(w) \in C $ gets penalized,
hence, the condition $ G(w^k) \in C $ will typically be violated.
Consequently, the corresponding normal cone $ \mathcal{N}_C ( G(w^k) ) $
would be empty which is why we cannot expect to have $ \lambda^k \in \mathcal{N}_C ( G(w^k) ) $,
though we hope that this holds asymptotically. In order to deal with this situation,
we therefore have to introduce the sequence $ \{ z^k \} $.
Let us note that AM-stationarity corresponds to so-called AKKT stationarity for conic optimization
problems, i.e., where $C$ is a closed, convex cone and $D:=\mathbb W$, see
\cite[Section~5]{AndreaniGomezHaeserMitoRamos2020}.
The more general situation where $C$ and $D$ are closed, convex sets and the overall problem is
stated in arbitrary Banach spaces is investigated in \cite{BoergensKanzowMehlitzWachsmuth2020}.
Asymptotic notions of stationarity addressing situations where $D$ is a nonconvex set
of special type can be found, e.g., 
in \cite{AndreaniHaeserSecchinSilva2019,KanzowRaharjaSchwartz2021b,Ramos2019}.
As shown in \cite{Mehlitz2020}, the overall concept of asymptotic stationarity 
can be further generalized to feasible sets
which are given as the kernel of a set-valued mapping.
Let us mention that the theory in this section is still valid in situations where $C$ is
merely closed. In this case, one may replace the normal cone to $C$ in the sense of convex
analysis by the limiting normal cone everywhere. However, for nonconvex sets $C$,
our algorithmic approach from \cref{Sec:ALM} is not valid anymore.
Note that, for the price of a slack variable $w_\textup s \in \mathbb{Y}$,
we can transfer the given constraint system into
\[
	G(w)-w_\textup s=0,\ (w_\textup s,w)\in C\times D
\]
where the right-hand side of the nonlinear constraint is trivially convex.
In order to apply the algorithmic framework of this paper to this
reformulation, projections onto $C$ have to be computed efficiently.
Moreover, there might be a difference between the asymptotic notions of stationarity
and regularity discussed here when applied to this reformulation or the original 
formulation of the constraints.

Apart from the aforementioned difference, the motivation of AM-stationarity is similar to the 
one of AKKT-stationarity:
Suppose that the sequence $ \{ \lambda^k \} $ is bounded
and, therefore, convergent along a subsequence. Then, taking the limit 
on this subsequence in the definition of an AM-stationary point while using
the stability property \eqref{Eq:stability} of the limiting normal cone shows that 
the corresponding limit point satisfies the M-stationarity conditions from
\cref{Def:MStat}. In general, however,
the Lagrange multiplier estimates $ \{ \lambda^k \} $
in the definition of AM-stationarity
might be unbounded.
Though this boundedness can be guaranteed under suitable (relatively strong)
assumptions, the resulting convergence theory works under significantly weaker
conditions.

It is well known in optimization theory that a local minimizer of \eqref{Eq:GenP}
is M-stationary only under validity of a suitable constraint qualification. In
contrast, it has been pointed out in \cite[Theorem~4.2, Section~5.1]{Mehlitz2020}
that each local minimizer of \eqref{Eq:GenP} is AM-stationary. 
In order to infer that an AM-stationary point is already M-stationary, the presence
of so-called asymptotic regularity is necessary, 
see \cite[Definition~4.4]{Mehlitz2020}.

\begin{definition}\label{Def:AsymptoticRegularity}
	A feasible point $\bar w\in\mathbb{W}$ of \eqref{Eq:GenP} is called \emph{AM-regular}
	(asymptotically Mordukhovich-regular) whenever the condition
	\[
		\limsup\limits_{w\to \bar w,\,z\to 0}\mathcal M(w,z)
		\subset
		\mathcal M(\bar w,0)
	\]
	holds, where $\mathcal M\colon\mathbb{W}\times\mathbb{Y}\rightrightarrows\mathbb{W}$ 
	is the set-valued mapping defined via
	\[
		\mathcal M(w,z)
		:=
		G'(w)^*\mathcal N_C(G(w)-z)+\mathcal N_D^\textup{lim}(w).
	\]
\end{definition}

The concept of AM-regularity has been inspired by the notion of AKKT-regularity (sometimes referred
to as cone continuity property), which became popular as one of the weakest constraint qualifications
for standard nonlinear programs or MPCCs,
see e.g.\ \cite{AndreaniMartinezRamosSilva2018,AndreaniMartinezRamosSilva2016,Ramos2019}, and can be
generalized to a much higher level of abstractness. 
In this regard, we would like to point the reader's attention to the fact that AM-stationarity and
-regularity from \cref{Def:AMStat,Def:AsymptoticRegularity} 
are referred to as \emph{decoupled} asymptotic Mordukhovich-stationarity and -regularity
in \cite{Mehlitz2020} since these are already refinements of more general concepts. 
For the sake of a concise notation, however, we omit the term \emph{decoupled} here.

It has been shown in \cite[Section~5.1]{Mehlitz2020} that validity of AM-regularity at a
feasible point $\bar w\in\mathbb{W}$ of \eqref{Eq:GenP} is implied by
\begin{equation}\label{eq:GMFCQ}
	0\in G'(\bar w)^*\lambda+\mathcal N_D^\textup{lim}(\bar w),\quad
	\lambda\in\mathcal N_C(G(\bar w))
	\quad\Longrightarrow\quad
	\lambda = 0.
\end{equation}
The latter is known as NNAMCQ (no nonzero abnormal multiplier constraint qualification) or
GMFCQ (generalized Mangasarian--Fromovitz constraint qualification) in the literature.
Indeed, in the setting where we fix $C:=\R^{m_1}_-\times\{0\}^{m_2}$ and $D:=\mathbb{W}$, \eqref{eq:GMFCQ}
boils down to the classical Mangasarian--Fromovitz constraint qualification from standard
nonlinear programming. 
The latter choice for $C$ will be of particular interest, which is why we formalize this setting below.

\begin{setting}\label{set:standard_nonlinear_constraints}
	Given $m_1,m_2\in\N$, we set $m:=m_1+m_2$,  
	$\mathbb{Y}:=\R^{m}$, and
	$C:=\R^{m_1}_-\times\{0\}^{m_2}$. 
	No additional assumptions are postulated on the set $D$.
	We denote the component functions
	of $G$ by $G_1,\ldots,G_m\colon\mathbb{W}\to\R$. Thus, the constraint $G(w)\in C$ encodes
	the constraint system
	\[
		G_i(w) \leq 0\quad i=1,\ldots,m_1,\qquad
		G_i(w) =0\quad i=m_1+1,\ldots,m
	\]
	of standard nonlinear programming. For our analysis, we exploit the index sets
	\[
		I(\bar w):=\set{i\in\{1,\ldots,m_1\} \given G_i(\bar w)=0},
		\qquad
		J:=\{m_1+1,\ldots,m\},
	\]
	whenever $\bar w\in D$ satisfies $G(\bar w)\in C$ in the present situation.
\end{setting}

Let us emphasize that we did not make any assumptions regarding the structure of the set
$D$ in \cref{set:standard_nonlinear_constraints}.
Thus, it still covers numerous interesting
problem classes like complementarity-, vanishing-, or 
switching-constrained programs.
These so-called disjunctive programs of special type are addressed in the setting 
mentioned below which provides
a refinement of \cref{set:standard_nonlinear_constraints}.
\begin{setting}\label{set:disjunctive_programs_with_bi_disjunction}
	Let $\mathbb{X}$ be another Euclidean space, let $X\subset\mathbb X$ be the union
	of finitely many convex, polyhedral sets, 
	and let $T\subset\R^2$ be the union of two polyhedrons
	$T_1,T_2\subset\R^2$. 
	For functions $g\colon\mathbb{X}\to\R^{m_1}$, $h\colon\mathbb{X}\to\R^{m_2}$, and
	$p,q\colon\mathbb{X}\to\R^{m_3}$, we consider the constraint system given by
	\[
		\begin{aligned}
			g_i(x)&\leq 0&\quad&i=1,\ldots,m_1,\\
			h_i(x)&=0&&i=1,\ldots,m_2,\\
			\parens[\big]{ p_i(x),q_i(x) }&\in T&&i=1,\ldots,m_3,\\
			x&\in X.&&
		\end{aligned}
	\]
	Setting $\mathbb W:=\mathbb X\times\R^{m_3}\times\R^{m_3}$, 
	$\mathbb{Y}:=\R^{m_1}\times\R^{m_2}\times\R^{m_3}\times\R^{m_3}$, 
	\[
		G(x,u,v):=\parens[\big]{g(x),h(x),p(x)-u,q(x)-v},
	\]
	and
	\[
		C:=\R^{m_1}_-\times\{0\}^{m_2+2m_3},\qquad
		D:=X\times\widetilde T,
	\]
	where we used $\widetilde T:=\set*{(u,v) \given (u_i,v_i)\in T\ \forall i\in\{1,\ldots,m_3\}}$,
	we can handle this situation in the framework of this paper.
\end{setting}

Constraint regions as characterized in \cref{set:standard_nonlinear_constraints} can be tackled with
a recently introduced version of RCPLD (relaxed constant positive linear dependence constraint
qualification), see \cite[Definition~1.1]{XuYe2020}.

\begin{definition}\label{Def:RCPLD}
	Let $\bar w\in\mathbb{W}$ be a feasible point of the optimization problem \eqref{Eq:GenP}
	in \cref{set:standard_nonlinear_constraints}. Then $\bar w$ is said to satisfy RCPLD
	whenever the following conditions hold:
	\begin{enumerate}
		\item[(i)] the family $(\nabla G_i(w))_{i\in J}$ has constant rank on a neighborhood of $\bar w$,
		\item[(ii)] there exists an index set $S\subset J$ such that the family $(\nabla G_i(\bar w))_{i\in S}$ is a basis
			of the subspace $\Span\set{\nabla G_i(\bar w) \given i\in J}$, and
		\item[(iii)] for each index set $I\subset I(\bar w)$, each set of multipliers
			$\lambda_i\geq 0$ ($i\in I$) and $\lambda_i\in\R$ ($i\in S$), not all vanishing
			at the same time, and each vector $\eta\in \mathcal N^\textup{lim}_D(\bar w)$
			which satisfy
			\[
				0\in\sum_{i\in I\cup S}\lambda_i\nabla G_i(\bar w)+\eta,
			\]
			we find neighborhoods $U$ of $\bar w$ and $V$ of $\eta$ such that
			for all $w\in U$ and $\tilde\eta\in\mathcal N^\textup{lim}_D(w)\cap V$,
			the vectors from
			\[
				\begin{cases}
					(\nabla G_i(w))_{i\in I\cup S},\,\tilde \eta	& \text{if }\tilde \eta \neq 0,\\
					(\nabla G_i(w))_{i\in I\cup S}					& \text{if }\tilde \eta = 0
				\end{cases}
			\]
			are linearly dependent.
	\end{enumerate}
\end{definition}

RCPLD has been introduced for standard nonlinear programs 
(i.e., $D:=\mathbb{W}=\R^n$ in \cref{set:standard_nonlinear_constraints}) in \cite{AndreaniHaeserSchuverdtSilva2012}.
Some extensions to complementarity-constrained programs can be found in
\cite{ChieuLee2013,GuoLin2013}.
A more restrictive RCPLD-type constraint qualification which is capable of handling an abstract constraint set
can be found in \cite[Definition~1]{GuoYe2018}.
	Let us note that RCPLD from \cref{Def:RCPLD}
	does not depend on the precise choice of the index set 
	$S$ in (ii).

In case where $D$ is a set of product structure, condition (iii) in \cref{Def:RCPLD} can be slightly
weakened in order to obtain a reasonable generalization of the classical relaxed constant positive 
linear dependence constraint qualification, see \cite[Remark~1.1]{XuYe2020} for details. 
Observing that GMFCQ from \eqref{eq:GMFCQ} takes the particular form
\[
	0\in\sum\limits_{i\in I(\bar w)\cup J}\lambda_i\nabla G_i(\bar w)+\mathcal N^\textup{lim}_D(\bar w),
	\quad
	\lambda_i\geq 0 \,(i\in I)
	\quad\Longrightarrow\quad
	\lambda_i=0\,(i\in I(\bar w)\cup J)
\]
in \cref{set:standard_nonlinear_constraints}, it is obviously sufficient for RCPLD.
The subsequently stated result generalizes related observations from
\cite{AndreaniMartinezRamosSilva2016,Ramos2019}.

\begin{lemma}\label{lem:RCPLD_implies_AM_regularity}
	Let $\bar w\in\mathbb{W}$ be a feasible point for the optimization problem \eqref{Eq:GenP} in
	\cref{set:standard_nonlinear_constraints} where RCPLD holds. 
	Then $\bar w$ is AM-regular.
\end{lemma}

\begin{proof}
	Fix some $\xi\in \limsup_{w\to \bar w,\,z\to 0}\mathcal M(w,z)$.
	Then we find $\{w^k\},\{\xi^k\}\subset\mathbb W$ and $\{z^k\}\subset\R^m$ which satisfy
	$w^k\to\bar w$, $\xi^k\to\xi$, $z^k\to 0$, and
	$\xi^k\in\mathcal M(w^k,z^k)$ for all $k\in\N$. 
	Particularly, there are sequences $\{\lambda^k\}$ and
	$\{\eta^k\}$ satisfying $\lambda^k\in\mathcal N_C(G(w^k)-z^k)$, 
	$\eta^k\in\mathcal N^\textup{lim}_D(w^k)$, and $\xi^k=G'(w^k)^*\lambda^k+\eta_k$
	for each $k\in\N$.
	From $G(w^k)-z^k\to G(\bar w)$ and the special structure of $C$, we find
	$G_i(w^k)-z^k_i<0$ for all $i\in\{1,\ldots,m_1\}\setminus I(\bar w)$ and all sufficiently large
	$k\in\N$, i.e., 
	\[
		\lambda^k_i
		\begin{cases}
			=0		& i\in \{1,\ldots,m_1\}\setminus I(\bar w),\\
			\geq 0	& i\in I(\bar w)
		\end{cases}
	\]
	for sufficiently large $k\in\N$. Thus, we may assume without loss of generality that
	\[
		\xi^k=\sum\limits_{i\in I(\bar w)\cup J}\lambda^k_i\nabla G_i(w^k)+\eta^k
	\]
	holds for all $k\in\N$.
	By definition of RCPLD, $(\nabla G_i(w^k))_{i\in S}$ is a basis of the subspace
	$\Span\set{\nabla G_i(w^k) \given i\in J}$ for all sufficiently large $k\in\N$.
	Hence, there exist scalars $\mu^k_i$ ($i\in S$) such that
	\[
		\xi^k
		=
		\sum\limits_{i\in I(\bar w)}\lambda^k_i\nabla G_i(w^k)
		+
		\sum\limits_{i\in S}\mu^k_i\nabla G_i(w^k)
		+
		\eta^k
	\]
	holds for all sufficiently large $k\in\N$.
	On the other hand, \cite[Lemma~1]{AndreaniHaeserSchuverdtSilva2012} yields the existence
	of an index set $I^k\subset I(\bar w)$ and multipliers $\hat\mu^k_i>0$ ($i\in I^k$), 
	$\hat\mu^k_i\in\R$ ($i\in S$), and $\sigma_k\geq 0$
	such that
	\[
		\xi^k
		=
		\sum\limits_{i\in I^k\cup S}\hat\mu^k_i\nabla G_i(w^k)+\sigma_k\,\eta^k
	\]
	and
	\[
		\begin{aligned}
			&\sigma_k>0&	&\Longrightarrow&	&(\nabla G_i(w^k))_{i\in I^k\cup S},\,\eta^k
				\text{ linearly independent,}&\\
			&\sigma_k=0&	&\Longrightarrow&	&(\nabla G_i(w^k))_{i\in I^k\cup S}
				\text{ linearly independent.}&
		\end{aligned}
	\]
	Since there are only finitely many subsets of $I(\bar w)$, there needs to exist $I\subset I(\bar w)$
	such that $I^k=I$ holds along a whole subsequence. Along such a particular subsequence
	(without relabeling),
	we furthermore may assume $\sigma_k>0$ (otherwise, the proof will be easier) and, thus, may
	set $\hat\eta^k:=\sigma_k \eta^k\in\mathcal N^\textup{lim}_D(w^k)\setminus\{0\}$. 
	From above, we find linear independence of
	\[
		(\nabla G_i(w^k))_{i\in I\cup S},\,\hat{\eta}^k.
	\]
	Furthermore, we have
	\begin{equation}\label{eq:appropriate_representation}
		\xi^k=\sum\limits_{i\in I\cup S}\hat \mu^k_i\nabla G_i(w^k)+\hat\eta^k.
	\end{equation}
	
	Suppose that the sequence $\{((\hat\mu^k_i)_{i\in I\cup S},\hat\eta^k)\}$ is not bounded.
	Dividing \eqref{eq:appropriate_representation} by the norm of 
	$((\hat\mu^k_i)_{i\in I\cup S},\hat\eta^k)$, taking the limit $k\to\infty$, and
	respecting boundedness of $\{\xi^k\}$, continuity of $G'$, and outer semicontinuity 
	of the limiting normal cone
	yield the existence of a non-vanishing multiplier $((\hat\mu_i)_{i\in I\cup S},\hat\eta)$ 
	which satisfies $\hat\mu_i\geq 0$ ($i\in I$), $\hat\eta\in\mathcal N^\textup{lim}_D(\bar w)$,
	and
	\[
		0=\sum\limits_{i\in I\cup S}\hat\mu_i\nabla G_i(\bar w)+\hat\eta.
	\]
	Obviously, the multipliers $\hat\mu_i$ ($i\in I\cup S$) do not vanish at the same time since,
	otherwise, $\hat\eta=0$ would follow from above which yields a contradiction.
	Now, validity of RCPLD guarantees that the vectors 
	\[
		(\nabla G_i(w^k))_{i\in I\cup S},\,\hat{\eta}^k
	\]
	need to be linearly dependent for sufficiently large $k\in\N$. However, we already have shown
	above that these vectors are linearly independent, a contradiction.
	
	Thus, the sequence $\{((\hat\mu^k_i)_{i\in I\cup S},\hat\eta^k)\}$ is bounded and, therefore,
	possesses a convergent subsequence with limit
	$((\bar\mu_i)_{i\in I\cup S},\bar\eta)$. Taking the limit in \eqref{eq:appropriate_representation}
	while respecting $\xi^k\to\xi$, the continuity of $G'$, and the outer semicontinuity of the limiting
	normal cone, we come up with $\bar\mu_i\geq 0$ ($i\in I$), 
	$\bar\eta\in\mathcal N^\textup{lim}_D(\bar w)$, and
	\[
		\xi=\sum\limits_{i\in I\cup S}\bar\mu_i\nabla G_i(\bar w)+\bar\eta.
	\]
	Finally, we set $\bar\mu_i:=0$ for all $i\in\{1,\ldots,m\}\setminus(I\cup S)$. Then we
	have $(\bar\mu_i)_{i=1,\ldots,m}\in\mathcal N_C(G(\bar w))$ from $I\subset I(\bar w)$, i.e.,
	\[
		\xi\in G'(\bar w)^*\mathcal N_C(G(\bar w))+\mathcal N^\textup{lim}_D(\bar w)=\mathcal M(\bar w,0).
	\]
	This shows that $\bar w$ is AM-regular.
\end{proof}

A popular situation, where AM-regularity simplifies and, thus, becomes easier to verify, is described in the
following lemma which follows from \cite[Theorems~3.10, 5.2]{Mehlitz2020}.
\begin{lemma}\label{lem:affine_data_implies_AM_regularity}
	Let $\bar w\in\mathbb{W}$ be a feasible point for the optimization problem \eqref{Eq:GenP} 
	where $C$ is a polyhedron and $D$ is the union of finitely many polyhedrons.
	Then $\bar w$ is AM-regular if any only if 
	\[
		\limsup\limits_{w\to\bar w}\parens*{G'(w)^*\mathcal N_C(G(\bar w))+\mathcal N_D^\textup{lim}(\bar w)}
		\subset
		G'(\bar w)^*\mathcal N_C(G(\bar w))+\mathcal N_D^\textup{lim}(\bar w).
	\]
	Particularly, in case where $G$ is an affine function, $\bar w$ is AM-regular.
\end{lemma}

	Let us consider the situation where \eqref{Eq:GenP} is given as described in \cref{set:standard_nonlinear_constraints},
	and assume in addition that $D:=\mathbb W$ holds, i.e., that \eqref{Eq:GenP} is a standard nonlinear optimization
	problem with finitely many equality and inequality constraints.
	Then \cref{lem:affine_data_implies_AM_regularity} shows that AM-regularity corresponds to 
	the cone continuity property from \cite[Definition~3.1]{AndreaniMartinezRamosSilva2016},
	and the latter has been shown to be weaker than most of the established constraint qualifications
	which can be checked in terms of initial problem data.

The above lemma also helps us to find a tangible representation of AM-regularity in \cref{set:disjunctive_programs_with_bi_disjunction}.
\begin{lemma}\label{lem:AM_regularity_in_disjunctive_setting}
	Let $\bar x\in\mathbb X$ be a feasible point of the optimization problem from \cref{set:disjunctive_programs_with_bi_disjunction}.
	Furthermore, define a set-valued mapping $\widetilde{\mathcal M}\colon\mathbb X\rightrightarrows\mathbb X$ by
	\[
		\widetilde{\mathcal M}(x)
		:=
		\set*{
			\mathfrak L(x,\lambda,\rho,\mu,\nu,\xi)
			\given
			\begin{aligned}
			&0\leq\lambda\perp g(\bar x),\\
			&(\mu,\nu)\in\mathcal N^\textup{lim}_{\widetilde T}(p(\bar x),q(\bar x)),\\
			&\xi\in\mathcal N^\textup{lim}_X(\bar x)
			\end{aligned}
		}
	\]
	where $\mathfrak L\colon\mathbb X\times\R^{m_1}\times\R^{m_2}\times\R^{m_3}\times\R^{m_3}\times\mathbb X\to\mathbb X$ is the function given by
	\[
		\mathfrak L(x,\lambda,\rho,\mu,\nu,\xi)
		:=
		g'(x)^*\lambda+h'(x)^*\rho+p'(x)^*\mu+q'(x)^*\nu+\xi.
	\]
	Then the feasible point $(\bar x,p(\bar x),q(\bar x))$ of the associated problem \eqref{Eq:GenP} is AM-regular if and only if
	\begin{equation}\label{eq:AM_regularity_disjunctive}
		\limsup\limits_{x\to\bar x}\widetilde{\mathcal M}(x)
		\subset
		\widetilde{\mathcal M}(\bar x).
	\end{equation}
\end{lemma}
\begin{proof}
	First, observe that transferring the constraint region from \cref{set:disjunctive_programs_with_bi_disjunction} into the form used in
	\eqref{Eq:GenP} and keeping \cref{lem:affine_data_implies_AM_regularity} in mind shows that AM-regularity 
	of $(\bar x,p(\bar x),q(\bar x))$ is equivalent to
	\begin{equation}\label{eq:AM_regularity_disjunctive_definition}
		\limsup\limits_{x\to\bar x}\widehat{\mathcal M}(x)
		\subset
		\widehat{\mathcal M}(\bar x)
	\end{equation}
	where $\widehat{\mathcal M}\colon\mathbb X\rightrightarrows\mathbb X\times\R^{m_3}\times\R^{m_3}$ is given by
	\[
		\widehat{\mathcal M}(x)
		:=
		\set*{
			\parens[\big]{\mathfrak L(x,\lambda,\rho,\tilde\mu,\tilde\nu,\xi),-\tilde\mu+\mu,-\tilde\nu+\nu}
			\given
			\begin{aligned}
				&0\leq\lambda\perp g(\bar x),\\
				&(\mu,\nu)\in\mathcal N^\textup{lim}_{\widetilde T}(p(\bar x),q(\bar x)),\\
				&\xi\in\mathcal N^\textup{lim}_X(\bar x)
			\end{aligned}
		}.
	\]
	Observing that $\eta\in\widetilde{\mathcal M}(x)$ is equivalent to $(\eta,0,0)\in\widehat{\mathcal M}(x)$, \eqref{eq:AM_regularity_disjunctive_definition}
	obviously implies \eqref{eq:AM_regularity_disjunctive}.
	In order to show the converse relation, we assume that \eqref{eq:AM_regularity_disjunctive} holds
	and fix $(\eta,\alpha,\beta)\in\limsup_{x\to\bar x}\widehat{\mathcal M}(x)$.
	Then we find sequences $\{x^k\},\{\xi^k\},\{\eta^k\}\subset\mathbb X$, $\{\lambda^k\}\subset\R^{m_1}$, $\{\rho^k\}\subset\R^{m_2}$,
	and $\{\mu^k\},\{\tilde\mu^k\},\{\nu^k\},\{\tilde\nu^k\}\subset\R^{m_3}$ such that
	$x^k\to\bar x$, $\eta^k\to\eta$, $-\tilde\mu^k+\mu^k\to \alpha$, $-\tilde\nu^k+\nu^k\to \beta$,
	and $\eta^k=\mathfrak L(x^k,\lambda^k,\rho^k,\tilde\mu^k,\tilde\nu^k,\xi^k)$,
	$0\leq\lambda^k\perp g(\bar x)$, $(\mu^k,\nu^k)\in\mathcal N^\textup{lim}_{\widetilde T}(p(\bar x),q(\bar x))$, as well as
	$\xi^k\in\mathcal N^{\textup{lim}}_X(\bar x)$ for all $k\in\N$. 
	Setting $\alpha^k:=-\tilde \mu^k+\mu^k$ and $\beta^k:=-\tilde\nu^k+\nu^k$, we find 
	$\eta^k+p'(x^k)^*\alpha^k+q'(x^k)^*\beta^k=\mathfrak L(x^k,\lambda^k,\rho^k,\mu^k,\nu^k,\xi^k)$ for each $k\in\N$, and due
	to $\alpha^k\to\alpha$ and $\beta^k\to\beta$, validity of \eqref{eq:AM_regularity_disjunctive} yields
	$\eta+p'(\bar x)^*\alpha+q'(\bar x)^*\beta\in\widetilde{\mathcal M}(\bar x)$, i.e., 
	the existence of $\lambda\in\R^{m_1}$, $\rho\in\R^{m_2}$, $\mu,\nu\in\R^{m_3}$, and $\xi\in\mathbb X$ such that
	$\eta+p'(\bar x)^*\alpha+q'(\bar x)^*\beta=\mathfrak L(\bar x,\lambda,\rho,\mu,\nu,\xi)$,
	$0\leq\lambda\perp g(\bar x)$, $(\mu,\nu)\in\mathcal N^\textup{lim}_{\widetilde T}(p(\bar x),q(\bar x))$, and
	$\xi\in\mathcal N^\textup{lim}_X(\bar x)$.
	Thus, setting $\tilde\mu:=\mu-\alpha$ and $\tilde\nu:=\nu-\beta$, we find
	$(\eta,\alpha,\beta)\in\widehat{\mathcal M}(\bar x)$ showing \eqref{eq:AM_regularity_disjunctive_definition}.
\end{proof}

Let us specify these findings for MPCCs which can be stated
in the form \eqref{Eq:GenP} via \cref{set:disjunctive_programs_with_bi_disjunction}. Taking 
\cref{lem:affine_data_implies_AM_regularity,lem:AM_regularity_in_disjunctive_setting} into account, AM-regularity corresponds to the so-called
MPCC cone continuity property from \cite[Definition~3.9]{Ramos2019}.
The latter has been shown to be strictly weaker than MPCC-RCPLD, see \cite[Definition~4.1, Theorem~4.2, Example~4.3]{Ramos2019}
for a definition and this result.
A similar reasoning can be used in order to show that problem-tailored versions
of RCPLD associated with other classes of disjunctive programs are sufficient
for the respective AM-regularity.
This, to some extend, recovers our result from \cref{lem:RCPLD_implies_AM_regularity} although we need to admit that, exemplary,
RCPLD from \cref{Def:RCPLD} applied to MPCC in \cref{set:disjunctive_programs_with_bi_disjunction} does not correspond to
MPCC-RCPLD.

The above considerations underline that AM-regularity is a comparatively weak constraint
qualification for \eqref{Eq:GenP}. Exemplary, for standard nonlinear problems and for MPCCs, this follows from the above
comments and the considerations in \cite{AndreaniMartinezRamosSilva2016,Ramos2019}. For other types of disjunctive
programs, the situation is likely to be similar, see e.g.\ \cite[Figure~3]{LiangYe2021} for the
setting of switching-constrained optimization.
It remains a topic of future research to find further sufficient conditions for AM-regularity
which can be checked in terms of initial problem data, particularly, in situations where $C$ and $D$ are of particular structure
like in semidefinite or second-order cone programming, see e.g. \cite[Section~6]{AndreaniHaeserViana2020}.
Let us mention that the provably weakest constraint qualification which guarantees
that local minimizers of a geometrically constrained program are M-stationary is slightly weaker
than validity
of the pre-image rule for the computation of the limiting normal cone to the constraint region
of \eqref{Eq:GenP}, see \cite[Section~3]{GuoLin2013} for a discussion, but the latter cannot be checked in practice.
Due to \cite[Theorem~3.16]{Mehlitz2020}, AM-regularity indeed implies validity of this pre-image rule.

\section{A Spectral Gradient Method for Nonconvex Sets}\label{Sec:ProjGrad}

In this section, we discuss a solution method for constrained optimization problems
which applies whenever projections onto the feasible set are easy to find. 
Exemplary, our method can be used in situations where the feasible set has a 
disjunctive nonconvex structure.

To motivate the method, first consider the unconstrained optimization problem
\begin{equation*}
   \min_w \ \varphi(w)\quad \text{s.t.} \quad w \in \mathbb{R}^n
\end{equation*}
with a continuously differentiable objective function $ \varphi\colon\mathbb{R}^n \to \mathbb{R} $,
and let $ w^j $ be a current estimate for a solution of this problem. Computing
the next iterate $ w^{j+1} $ as the unique minimizer of the local quadratic model
\begin{equation*}
   \min_w \ \varphi(w^j) + \nabla \varphi (w^j)^\top (w - w^j) + \frac{\gamma_j}{2} \norm{ w - w^j }^2
\end{equation*}
for some $ \gamma_j > 0 $ leads to the explicit expression 
\begin{equation*}
   w^{j+1} := w^j - \frac{1}{\gamma_j} \nabla \varphi(w^j),
\end{equation*}
i.e., we get a steepest descent method with stepsize $ t_j := 1/\gamma_j $. Classical
approaches compute $ t_j $ using a suitable stepsize rule such that $\varphi(w^{j+1}) <\varphi(w^j) $.
On the other hand, one can view the update formula as a special instance of 
a quasi-Newton scheme
\begin{equation*}
   w^{j+1} := w^j - B_j^{-1} \nabla\varphi(w^j)
\end{equation*}
with the very simple quasi-Newton matrix $ B_j := \gamma_j I $ as an estimate
of the (not necessarily existing) Hessian $ \nabla^2\varphi(w^j) $. Then the 
corresponding quasi-Newton equation
\begin{equation*}
   B_{j+1} s^j = y^j \quad \text{with } s^j := w^{j+1} - w^j, \ 
   y^j := \nabla\varphi(w^{j+1}) - \nabla\varphi(w^j),
\end{equation*}
see \cite{DennisSchnabel1983}, 
reduces to the linear system $ \gamma_{j+1} s^j = y^j $. Solving this 
overdetermined system in a least squares sense, we then obtain the stepsize
\begin{equation*}
    \gamma_{j+1} := (s^j)^\top y^j / (s^j)^\top s^j
\end{equation*}
introduced by Barzilai and Borwein \cite{BarzilaiBorwein1988}. This stepsize 
often leads to very good numerical results, but may not yield a monotone 
decrease in the function value. A convergence proof for general nonlinear programs
is therefore difficult, even if the choice of $ \gamma_j $ is safeguarded in the 
sense that it is projected onto some box $ [ \gamma_{\min}, \gamma_{\max} ] $ for
suitable constants $ 0 < \gamma_{\min} < \gamma_{\max} $.

Raydan \cite{Raydan1997} then suggested to control this nonmonotone behavior
by combining the Barzilai--Borwein stepsize with the nonmonotone linesearch
strategy introduced by Grippo et al.\ \cite{GrippoLamparielloLucidi1986}.
This, in particular, leads to a global convergence theory for general 
unconstrained optimization problems.

This idea was then generalized by Birgin et al.~\cite{BirginMartinezRaydan2000}
to constrained optimization problems
\begin{equation*}
   \min_w \ \varphi(w) \quad \text{s.t.} \quad w \in W
\end{equation*}
with a nonempty, closed, and convex set $ W\subset\R^n $ and is called the \emph{nonmonotone spectral 
gradient method}. Here, we extend their approach to minimization problems
\begin{equation}\label{Eq:P}
   \min_w \ \varphi(w) \quad \text{s.t.} \quad w \in D
\end{equation}
with a continuously differentiable function $\varphi\colon \mathbb{W} \to \mathbb{R} $
and some nonempty, closed set $D\subset \mathbb{W} $, where $\mathbb W$ is an arbitrary
Euclidean space. Let us emphasize that neither $\varphi$ nor $D$ need to be convex in our
subsequent considerations.
A detailed description of the corresponding generalized spectral gradient
is given in \cref{Alg:NPG}.

\newcommand\Qref[1]{%
	\hyperref[eq:sub_j_i]{\text{($\textup{Q}(#1)$)}}%
}

\begin{algorithm2e}[htb]
	\SetAlgoLined
	\KwData{$ \tau > 1, \sigma \in (0,1), 0 < \gamma_{\min} \leq \gamma_{\max} < \infty,
    m \in \mathbb{N}, w^0 \in D $
	}
	\For{$j \leftarrow 0$ \KwTo $\infty$}{
		Set $ m_j := \min ( j, m ) $, $i \leftarrow 0$ and choose $ \gamma_j^0 \in [ \gamma_{\min}, \gamma_{\max} ] $\;
		\Repeat
		{
			\label{step:nonmonotone_linesearch}%
			$
			\varphi(w^{j,i}) \leq \max_{\shift = 0,1, \ldots, m_j}\varphi(w^{j-\shift}) +
			\sigma \innerp{\nabla\varphi (w^j)}{w^{j,i} - w^j}
				% \end{equation}
			$
				% holds.
		}
		{
			Set $i \leftarrow i + 1$,
			$\gamma_{j,i} := \tau^{i-1} \gamma_j^0$
			and
			compute a solution $w^{j,i}$
			of
			\begin{algomathdisplay}
				\label{eq:sub_j_i}
				\tag{$\textup{Q}(j,i)$}
				\min_{w} \ \varphi(w^j) + \innerp{ \nabla \varphi (w^j) }{ w-w^j } + \frac{\gamma_{j,i}}{2} \norm{ w - w^j }^2
				\quad 
				\text{s.t.} \quad w \in D 
				% \;
			\end{algomathdisplay}
			\If{$w^{j,i}$ satisfies a termination criterion\label{item:termination_NSG}}{
				% STOP algorithm
				\Return $w^{j,i}$\;
			}
		}
		Set $i_j := i$,
		$\gamma_j := \gamma_{j,i}$,
		and $w^{j+1} := w^{j,i}$\;
				% \KwBreak (the inner loop)
	}
	\caption{General Spectral Gradient Method}
	\label{Alg:NPG}
\end{algorithm2e}

Particular instances of this approach with nonconvex sets $ D $ can already be 
found in  \cite{BeckEldar2013,ChenGuoLuYe2017,ChenLuPong2016,GuoDeng2021}. Note that all iterates belong to the set
$ D $, that the subproblems \eqref{eq:sub_j_i} are always solvable,
and that we have to compute only one solution,
although their solutions are not necessarily unique.
We would like to emphasize that $\nabla\varphi(w^j)$ was used in the formulation of
\eqref{eq:sub_j_i} in order to underline that \cref{Alg:NPG} is a projected gradient
method. Indeed, simple calculations reveal that the global solutions of \eqref{eq:sub_j_i}
correspond to the projections of $w^j-\gamma_{j,i}^{-1}\nabla\varphi(w^j)$ onto $D$.
Note also that the acceptance criterion in \cref{step:nonmonotone_linesearch} is the 
nonmonotone Armijo rule introduced by Grippo et al.\ \cite{GrippoLamparielloLucidi1986}.
In particular, the parameter $ m_j := \min (j,m) $ controls the nonmonotonicity.
The choice $ m = 0 $ corresponds to the standard (monotone) method, whereas $ m > 0 $
typically allows larger stepsizes and often leads to faster convergence of the method.

We stress that the previous generalization of existing spectral gradient methods
plays a fundamental role in order to apply our subsequent augmented Lagrangian
technique to several interesting and difficult optimization problems, 
but the 
convergence analysis of \cref{Alg:NPG} can be carried out similar to the 
one given in \cite{GuoDeng2021} where a more specific situation is discussed.
We therefore skip the corresponding proofs
in this section, but for the reader's convenience, we present them
in \cref{sec:Appendix}.

	The goal of \cref{Alg:NPG} is the computation of a point which is
	approximately M-stationary for \eqref{Eq:P}.
We
recall that $ w $ is an M-stationary point of \eqref{Eq:P}
if 
\begin{equation*}
   0 \in \nabla\varphi(w) + \mathcal N_D^\textup{lim} (w)
\end{equation*}
holds, 
and that each locally optimal solution of \eqref{Eq:P} is M-stationary
by \cite[Theorem~6.1]{Mordukhovich2018}.
Similarly, since $ w^{j,i} $ solves the subproblem
\Qref{j,i},
% \begin{equation}\label{Eq:Subproblem_ji}
% 	\min_w \ \varphi(w^j) + \innerp{ \nabla \varphi (w^j) }{ w-w^j }  + \frac{\gamma_{j,i}}{2} \norm{ w - w^j }^2 \quad 
% 	\text{s.t.} \quad w \in D
% \end{equation}
it satisfies
the corresponding M-stationarity condition
\begin{equation}\label{eq:M_St_surrogate_ji}
	0 \in \nabla\varphi(w^j) + \gamma_{j,i} \parens[\big]{ w^{j,i} - w^j } + \mathcal N_D^\textup{lim} (w^{j,i}) .
\end{equation}
Let us point the reader's attention to the fact that strong stationarity, where the
limiting normal cone is replaced by the smaller regular normal cone in the stationarity system,
provides a more restrictive necessary optimality condition for \eqref{Eq:P} and the surrogate
\Qref{j,i},
% \eqref{Eq:Subproblem_ji},
see \cite[Definition~6.3, Theorem~6.12]{RockafellarWets2009}.
It is well known that the limiting normal cone is the outer limit of the regular normal cone.
In contrast to the limiting normal cone, the regular one is not robust in the sense of
\eqref{Eq:stability}, and since we are interested in taking limits later on, one either way
ends up with a stationarity systems in terms of limiting normals at the end. 
Thus, we will rely on the limiting normal cone and the associated concept of M-stationarity.

For the following theoretical results,
we neglect the termination criterion in \cref{item:termination_NSG}.
	This means that \cref{Alg:NPG} does not terminate
	and performs either infinitely many inner or infinitely many outer interations.
	The first result analyzes the inner loop.

\begin{proposition}\label{Prop:InnerLoop}
	Consider a fixed (outer) iteration $ j $ in \cref{Alg:NPG}.
	Then the inner loop terminates (due to \cref{step:nonmonotone_linesearch})
			or
			\begin{equation}
				\label{eq:inner_not_termination}
				\norm{
					\gamma_{j,i} \parens[\big]{ w^j - w^{j,i} }
					+
					\nabla\varphi(w^{j,i}) - \nabla\varphi(w^j)
				}
				\to 0
				\qquad\text{as}\quad
				i \to \infty.
			\end{equation}
			If the inner loop does not terminate,
			we get $w^{j,i} \to w^j$ and $w^j$ is M-stationary.
\end{proposition}
We refer to \cref{sec:Appendix} for the proof.
It remains to analyze the situation where the inner loop always terminates.
Let $ w^0 \in D $ be the starting point from \cref{Alg:NPG}, and let
\begin{equation*}
   \mathcal{S}_\varphi(w^0) := \set[\big]{ w \in D \given \varphi(w) \leq\varphi(w^0) }
\end{equation*}
denote the corresponding (feasible) sublevel set. Then the following observation holds, see
\cite{GrippoLamparielloLucidi1986,WrightNowak2009} and \cref{sec:Appendix}
for the details.

\begin{proposition}\label{Prop:Diffzero}
	We assume that the inner loop in \cref{Alg:NPG} always terminates
	(due to \cref{step:nonmonotone_linesearch})
	and we denote by
	$ \{ w^j \} $ the infinite sequence of (outer) iterates.
Assume that $ \varphi $
is bounded from below and uniformly continuous on $ \mathcal{S}_\varphi(w^0) $. Then
we have
$ \norm{ w^{j+1} - w^j } \to 0 $ as $ j \to \infty $.
\end{proposition}

The previous result allows to prove the following main convergence result for 
\cref{Alg:NPG}, see, again,
\cref{sec:Appendix}
for a complete proof.

\begin{proposition}\label{Thm:ConvNPG}
	We assume that the inner loop in \cref{Alg:NPG} always terminates
	(due to \cref{step:nonmonotone_linesearch})
	and we denote by
	$ \{ w^j \} $ the infinite sequence of (outer) iterates.
Assume that $ \varphi $
is bounded from below and uniformly continuous on $ \mathcal{S}_\varphi(w^0) $.
Suppose that $\bar w$ is an accumulation point of $\{ w^j \}$,
i.e., $w^j \to_K \bar w$ along a subsequence $K$.
Then $\bar w$ is 
an M-stationary point of the optimization problem
\eqref{Eq:P},
and we have
$\gamma_j \, \parens*{ w^{j+1} - w^j } \to_K 0$.
\end{proposition}

	From the proof of \cref{Prop:Diffzero}, it can be easily seen that the iterates 
	of \cref{Alg:NPG} belong to the sublevel set $\mathcal S_\varphi(w^0)$ although
	the associated sequence of function values does not need to be monotonically
	decreasing. Hence, whenever this sublevel set is bounded, e.g.,
	if $\varphi$ is coercive or if $D$ is bounded, the existence of an
	accumulation point as in \cref{Thm:ConvNPG} is ensured.
	Moreover, the boundedness of
	$\mathcal S_\varphi(w^0)$
	implies that
	this set is compact.
	Hence, $\varphi$ is automatically bounded from below
	and uniformly continuous on
	$\mathcal S_\varphi(w^0)$ in this situation.

	By combining
	\cref{Prop:InnerLoop,Thm:ConvNPG}
	we get the following convergence result.
	\begin{theorem}
		\label{thm:overall_convergence}
		We consider \cref{Alg:NPG} without termination in \cref{item:termination_NSG}
		and
		assume that
		Then exactly one of the following situations occurs.
		\begin{enumerate}[label=(\roman*)]
			\item
				The inner loop does not terminate in the outer iteration $j$,
				$w^{j,i} \to w^j$ as $i \to \infty$, $w^j$ is M-stationary,
				and \eqref{eq:inner_not_termination} holds.
			\item
				The inner loop always terminates.
				The infinite sequence $\{w^j\}$ of outer iterates possesses
				convergent subsequences $\{w^j\}_{K}$
				and
				every convergent subsequence satisfies
				$w^j \to_K \bar w$, $\bar w$ is M-stationary,
				and
				$\gamma_j \, \parens*{ w^{j+1} - w^j } \to_K 0$.
		\end{enumerate}
	\end{theorem}
	This result shows that the infinite sequence of (inner or outer)
	iterates of \cref{Alg:NPG}
	always converges towards M-stationary points (along subsequences).
	Note that the boundedness of
	$\mathcal S_\varphi(w^0)$
	can be replaced by the assumptions on $\varphi$ of \cref{Thm:ConvNPG},
	but then the outer iterates $\{w^j\}$
	might fail to possess accumulation points.

	In what follows,
	we show that these theoretical results
	also give rise to a reasonable and applicable termination criterion
	which can be used in \cref{item:termination_NSG}.
	To this end, we note that
	the optimality condition
	\eqref{eq:M_St_surrogate_ji} is equivalent to
	\begin{equation*}
		\gamma_{j,i} \parens[\big]{ w^j - w^{j,i} }
		+
		\nabla\varphi(w^{j,i}) - \nabla\varphi(w^j)
		\in
		\nabla\varphi(w^{j,i}) + \mathcal N_D^\textup{lim} (w^{j,i})
		.
	\end{equation*}
	This motivates
	the usage of
	\begin{equation}
		\label{eq:inner_inner_termination}
		\norm{
			\gamma_{j,i} \parens[\big]{ w^j - w^{j,i} }
			+
			\nabla\varphi(w^{j,i}) - \nabla\varphi(w^j)
		}
		\leq\varepsilon_{\textup{tol}}
	\end{equation}
	(or a similar condition),
	with $\varepsilon_{\textup{tol}} > 0$,
	as a termination criterion in \cref{item:termination_NSG}.
	Indeed, \cref{Prop:InnerLoop}
	implies that the inner loop always terminates
	if \eqref{eq:inner_inner_termination} is used.
	Moreover, the termination criterion \eqref{eq:inner_inner_termination} directly encodes
	that $w^{j,i}$
	is approximately M-stationary for \eqref{Eq:P}.
	This is very desirable since the goal of \cref{Alg:NPG}
	is the computation of approximately M-stationary points.

	Furthermore,
	we can check that condition \eqref{eq:inner_inner_termination}
	always ensures the finite termination of \cref{Alg:NPG}
	if the mild assumptions of \cref{thm:overall_convergence}
	(or the even weaker assumptions of \cref{Thm:ConvNPG})
	are satisfied.
	Indeed,
	due to $\gamma_j = \gamma_{j,i_j}$
	and $w^{j+1} = w^{j,i_j}$,
	we have $\gamma_{j,i_j} \parens[\big]{w^{j} - w^{j,i_j}} = \gamma_j \parens[\big]{w^j - w^{j+1}} \to_K 0$.
	Using
	$w^{j+1}, w^{j} \to_K \bar w$
	and the continuity of $\nabla\varphi \colon \mathbb W\to\mathbb W$
	shows
	$\nabla\varphi(w^{j,i_j}) - \nabla\varphi(w^j) = \nabla\varphi(w^{j+1}) - \nabla\varphi(w^j) \to_K 0$.
	Thus, the left-hand side of \eqref{eq:inner_inner_termination}
	with $i = i_j$
	is arbitrarily small if $j \in K$ is large enough.
	Thus,
	\cref{Alg:NPG} with the termination criterion \eqref{eq:inner_inner_termination} terminates in finitely many steps.

Let us mention that the above convergence theory differs from the one provided in
\cite{ChenGuoLuYe2017,ChenLuPong2016} since no Lipschitzianity of 
$\nabla \varphi\colon\mathbb W\to\mathbb W$ is needed.
	In the particular setting of complementarity-constrained optimization, related
	results have been obtained in \cite[Section~4]{GuoDeng2021}. 
	Our findings substantially generalize the theory from \cite{GuoDeng2021}
	to arbitrary set constraints.

\section{An Augmented Lagrangian Approach for Structured Geometric Constraints}\label{Sec:ALM}

\Cref{Sub:ALM} contains a detailed statement of our augmented Lagrangian method
applied to the general class of problems \eqref{Eq:GenP} together with several
explanations. The convergence theory is then presented in \cref{Sub:ALMConvergence}.

\subsection{Statement of the Algorithm}\label{Sub:ALM}

We now consider the optimization problem \eqref{Eq:GenP} under the given 
smoothness and convexity assumptions stated there (recall that $ D $ is not 
necessarily convex). This section presents a safeguarded augmented Lagrangian
approach for the solution of \eqref{Eq:GenP}. The method penalizes the 
constraints $ G(w) \in C $, but leaves the possibly complicated condition
$ w \in D $ explicitly in the constraints. Hence, the resulting subproblems
that have to be solved in the augmented Lagrangian framework have exactly 
the structure of the (simplified) optimization problems discussed in 
\cref{Sec:ProjGrad}.

To be specific, consider the (partially) augmented Lagrangian
\begin{equation}\label{eq:augmented_Lagrangian}
   \mathcal{L}_{\rho} (w, \lambda) :=
  f(w) + \frac{\rho}{2} d_C^2 \parens*{ G(w) + \frac{\lambda}{\rho} }
\end{equation}
of \eqref{Eq:GenP}, where $ \rho > 0 $ denotes the penalty parameter. Note that
the squared distance function of a nonempty, closed, and convex set is always continuously differentiable, 
see e.g.\ \cite[Corollary~12.30]{BauschkeCombettes2011}, 
which yields that $ \mathcal{L}_{\rho} ( \cdot, \lambda ) $
is a continuously differentiable mapping. Using the definition of the distance,
we can alternatively write this (partially) augmented Lagrangian as
\begin{equation*}
   \mathcal{L}_{\rho} (w, \lambda) =
  f(w) + \frac{\rho}{2} \norm*{ G(w) + \frac{\lambda}{\rho} -
   P_C \parens*{ G(w) + \frac{\lambda}{\rho} } }^2.
\end{equation*}
In order to control the update of the penalty parameter, we also introduce the 
auxiliary function
\begin{equation}
\label{eq:def_V}
   V_\rho(w, u) := \norm*{ G(w) -
   P_C \parens*{ G(w) + \frac{u}{\rho} } }
	 .
\end{equation}
This function $V_\rho$ can also be used to obtain a meaningful termination criterion,
see the discussion after \eqref{eq:termination} below.
The overall method is stated in \cref{Alg:ALM}.

\begin{algorithm2e}[htb]
	\SetAlgoLined
	\KwData{$ \rho_0 > 0$, $\beta > 1$, $\eta \in (0,1)$, $w^0 \in D $, nonempty and bounded set $U\subset\mathbb Y$
	}
	\For{$k \leftarrow 0$ \KwTo $\infty$}{
		\If{$w^k$ satisfies a termination criterion\label{item:termination_ALM}}{
				\Return $w^k$\;
		}
		{
		Choose $ u^k \in U$\; 
		Compute an approximately M-stationary point $w^{k+1}$ of the subproblem
    	\begin{equation*}
        	\min_w \ \mathcal{L}_{\rho_k} ( w, u^k ) \quad \text{s.t.} \quad w \in D,
    	\end{equation*}
		i.e., 
		for some suitable (sufficiently small) vector $ \varepsilon^{k+1} \in \mathbb{W} $,
		$w^{k+1}$ needs to satisfy	
		\label{item:solve:subproblem_ALM}
    	\begin{algomathdisplay}
       		\varepsilon^{k+1} \in \nabla_w \mathcal{L}_{\rho_k} (w^{k+1}, u^k) +
       		\mathcal N_D^\textup{lim} (w^{k+1}) 
    	\end{algomathdisplay}
			% \;
    	}
    	{
 		Set 
 		\label{item:adjust_multiplier_ALM}%
		$
        		\lambda^{k+1} := \rho_k \bracks*{ G(w^{k+1}) + {u^k}/{\rho_k} -
        		P_C \parens*{ G(w^{k+1}) + {u^k}/{\rho_k} } }
						$\;
    	}
			\eIf{$k=0$ \KwOr $ V_{\rho_k}( w^{k+1}, u^k) \leq \eta V_{\rho_{k-1}}( w^k, u^{k-1})$\label{item:update_penalty_ALM}}{
    		$\rho_{k+1}:=\rho_k$\label{item:update_2}\;
    	}
    	{\label{item:update_3}
    		$\rho_{k+1}:=\beta\rho_k$\label{item:update_4}\;
    	}\label{item:update_5}	
	}
	\caption{Safeguarded Augmented Lagrangian Method for Geometric Constraints}
	\label{Alg:ALM}
\end{algorithm2e}

\Cref{item:solve:subproblem_ALM} of \cref{Alg:ALM}, in general, 
contains the main computational effort since we have
to ``solve'' a constrained nonlinear program at each iteration. Due to the 
nonconvexity of this subproblem, we only require to compute
an M-stationary point of this program. In fact, we allow the computation of 
an approximately M-stationary point, with the vector $ \varepsilon^{k+1} $
measuring the degree of inexactness. The choice $ \varepsilon^{k+1} = 0 $
corresponds to an exact M-stationary point. Note that the subproblems arising
in \cref{item:solve:subproblem_ALM}
have precisely the structure of the problem investigated in 
\cref{Sec:ProjGrad}, hence, the spectral gradient method discussed there
is a canonical candidate for the solution of these subproblems (note also that 
the objective function $ \mathcal{L}_{\rho_k} ( \cdot, u^k) $ is 
once, but usually not twice continuously differentiable).

Note that \cref{Alg:ALM} is called a safeguarded augmented Lagrangian
method due to the appearance of the auxiliary sequence 
	$ \{ u^k \}\subset U $ where $U$ is a bounded set. 
In fact,
if we would replace $ u^k $ by $ \lambda^k $ in \cref{item:solve:subproblem_ALM} of \cref{Alg:ALM}
(and the corresponding
subsequent formulas), we would obtain the classical augmented Lagrangian
method. However, the safeguarded version has superior global convergence properties,
see \cite{BirginMartinez2014} for a general discussion and \cite{KanzowSteck2017} for an
explicit (counter-) example. In practice, $ u^k $ is typically chosen to be
equal to $ \lambda^k $ as long as this vector belongs to the set
$U$, otherwise $ u^k $ is taken as the projection of
$ \lambda^k $ onto this set. In situations where $\mathbb Y$ is equipped with some
(partial) order relation
$\lesssim$, a typical choice for $U$ is given by the box 
$[u_{\min},u_{\max}]:=\set{u\in\mathbb Y \given u_\textup{min}\lesssim u\lesssim u_\textup{max}}$ where 
$u_\textup{min},u_\textup{max}\in\mathbb Y$ are given bounds satisfying 
$u_\textup{min}\lesssim u_\textup{max}$.

In order to understand the update of the Lagrange multiplier estimate in
\cref{item:adjust_multiplier_ALM} of \cref{Alg:ALM}, 
recall that the augmented Lagrangian is differentiable, with its derivative
given by
\begin{equation*}
   \nabla_w \mathcal{L}_{\rho} (w, \lambda ) = \nabla f(w) +
   \rho G'(w)^* \bracks*{ G(w) + \frac{\lambda}{\rho} - P_C \parens*{
   G(w) + \frac{\lambda}{\rho} } },
\end{equation*}
see \cite[Corollary~12.30]{BauschkeCombettes2011} again.
Hence, if we denote the usual (partial) Lagrangian of \eqref{Eq:GenP} by
\begin{equation*}
   \mathcal{L} (w, \lambda) :=f(w) + \innerp{\lambda}{G(w)},
\end{equation*}
we obtain from \cref{item:adjust_multiplier_ALM} that 
\begin{equation}\label{Eq:lambdaup}
   \nabla_w \mathcal{L}_{\rho_k} (w^{k+1}, u^k) = \nabla f(w^{k+1}) + G'(w^{k+1})^*
   \lambda^{k+1} = \nabla_w \mathcal{L}(w^{k+1}, \lambda^{k+1}) .
\end{equation}
This formula is actually the motivation for the precise update used in
\cref{item:adjust_multiplier_ALM}.

The particular updating rule in 
\cref{item:update_penalty_ALM,item:update_2,item:update_3,item:update_4,item:update_5} 
of \cref{Alg:ALM} is quite common, but other formulas might
also be possible.
In particular, one can use a different norm in the definition \eqref{eq:def_V} of $V_\rho$.
Exemplary, we exploited the maximum-norm for our experiments in \cref{Sec:Numerics}
where $\mathbb W$ is a space of real vectors or matrices.
Let us emphasize that increasing the penalty parameter $ \rho_k $
based on a pure infeasibility measure does not work in \cref{Alg:ALM}.
One usually has to take
into account both the infeasibility of the current iterate (w.r.t.\
the constraint $ G(w) \in C $) and a kind of complementarity condition 
(i.e., $\lambda\in\mathcal N_C(G(w))$).

	For the discussion of a suitable termination criterion, we define
	\begin{equation*}
		z^{k} := G(w^{k}) - P_C \parens*{ G(w^{k}) + \frac{u^{k-1}}{\rho_{k-1}} }.
	\end{equation*}
	Using \eqref{Eq:lambdaup} and the update formula for $\lambda^{k}$,
	\cref{Alg:ALM} ensures
	\begin{subequations}
		\label{eq:termination}
		\begin{align}
			\varepsilon^{k}
			&\in
			\nabla f(w^{k}) + G'(w^{k})^* \lambda^{k}
			+
			\mathcal N_D^\textup{lim} (w^{k}),
			\\
			\lambda^{k} &\in \mathcal N_C(G(w^{k}) - z^{k}),
		\end{align}
	\end{subequations}
	and this corresponds to the definition of AM-stationary points,
	see \cref{Def:AMStat}.
	Thus, it is reasonable
	to require $\varepsilon^{k} \to 0$
	and to use
	\begin{equation}
		\label{eq:stopping_ALM}
		\norm{z^{k}}
		=
		V_{\rho_{k-1}}(w^{k}, u^{k-1})
		\le
		\varepsilon_{\textup{tol}}
	\end{equation}
	for some $\varepsilon_{\textup{tol}}>0$
	as a termination criterion.
In practical implementations of \cref{Alg:ALM}, a maximum number of iterations should also
be incorporated into the termination criterion.

\subsection{Convergence}\label{Sub:ALMConvergence}

	Throughout our convergence analysis, we assume implicitly that \cref{Alg:ALM}
	does not stop after finitely many iterations.

Like all penalty-type methods in the setting of nonconvex programming, 
augmented Lagrangian methods suffer from the
drawback that they generate accumulation points which are not necessarily 
feasible for the given optimization problem~\eqref{Eq:GenP}. The following
(standard) result therefore presents some conditions under which it is guaranteed that
limit points are feasible.

\begin{proposition}\label{Prop:Feas1}
Each accumulation
point $ \bar w $ of a sequence $ \{ w^k \} $ generated
by \cref{Alg:ALM} is feasible for the optimization problem~\eqref{Eq:GenP}
if one of the following conditions holds:
\begin{enumerate}
 \item \label{item:rho_bounded} $ \{ \rho_k \} $ is bounded, or 
 \item \label{item:AL_bounded} there exists some $ B \in \mathbb{R} $ such that $ \mathcal{L}_{\rho_k}
    (w^{k+1}, u^k) \leq B $ holds for all $ k \in \mathbb{N} $.
\end{enumerate}
\end{proposition}

\begin{proof}
	Let $ \bar w $ be an arbitrary accumulation point of $\{w^k\}$ and, say, $ \{ w^{k+1} \}_{K}$ a corresponding subsequence with $ w^{k+1}\to_K\bar w $. 
	
	We start with the proof under validity of condition~\ref{item:rho_bounded}.
	Since $ \{ \rho_k \} $ is bounded, \cref{item:update_penalty_ALM,item:update_2,item:update_3,item:update_4,item:update_5} 
	of \cref{Alg:ALM} imply that $ V_{\rho_k}(w^{k+1}, u^k) \to 0 $ for $ k \to \infty $. This implies 
\begin{equation*}
   d_C ( G(w^{k+1}) ) \leq \norm*{ G(w^{k+1}) - P_C \parens*{ G(w^{k+1}) +
   \frac{u^k}{\rho_k} } } = V_{\rho_k} ( w^{k+1}, u^k) \to 0.
\end{equation*}
A continuity argument
yields $ d_C ( G(\bar w) ) = 0 $. Since $ C $ is a closed set, this
implies $ G(\bar w) \in C $. Furthermore, by construction, we have $ w^{k+1} \in D $
for all $ k \in \mathbb{N} $, so that the closedness of $ D $ also yields $ \bar w \in D $.
Altogether, this shows that $ \bar w $ is feasible for the optimization problem
\eqref{Eq:GenP}.\smallskip 

Let us now proof the result in presence of~\ref{item:AL_bounded}. 
In view of~\ref{item:rho_bounded}, it suffices to consider the situation where $ \rho_k \to \infty $.
By assumption, we have
\begin{equation*}
  f(w^{k+1}) + \frac{\rho_k}{2} d_C^2 \parens*{ G(w^{k+1}) + \frac{u^k}{\rho_k} }
   \leq B \qquad \forall k \in \mathbb{N}.
\end{equation*}
Rearranging terms yields
\begin{equation}\label{Eq:Feas1}
   d_C^2 \parens*{ G(w^{k+1}) + \frac{u^k}{\rho_k} } \leq
   \frac{2 ( B - f(w^{k+1}) )}{\rho_k} \qquad \forall k \in \mathbb{N}.
\end{equation}
Taking the limit $ k \to_K \infty $
in \eqref{Eq:Feas1} and using the boundedness of $ \{ u^k \} $, we obtain
\begin{equation*}
   d_C^2 \parens[\big]{ G(\bar w) } = \lim_{k \to_K\infty} d_C^2 \parens*{ G(w^{k+1}) + \frac{u^k}{\rho_k} }
   = 0 
\end{equation*}
by a continuity argument.
Similar to part~\ref{item:rho_bounded}, this implies feasibility of $ \bar w $.
\end{proof}

The two conditions~\ref{item:rho_bounded} and~\ref{item:AL_bounded} of \cref{Prop:Feas1} are, of course,
difficult to check a priori. Nevertheless, in the situation where each iterate $ w^{k+1} $
is actually a global minimizer of the subproblem in \cref{item:solve:subproblem_ALM} of \cref{Alg:ALM}
and $ w $ denotes
any feasible point of the optimization problem \eqref{Eq:GenP}, we have
\begin{equation*}
    \mathcal{L}_{\rho_k} (w^{k+1}, u^k) 
    \leq 
    \mathcal{L}_{\rho_k} (w, u^k) 
    \leq 
    f(w)     + \frac{\norm{ u^k }^2}{2 \rho_k} 
    \leq 
    f(w) + \frac{\norm{ u^k }^2}{2 \rho_0}\leq B
\end{equation*}
for some suitable constant $ B $ due to the boundedness of the sequence $ \{ u^k \} $.
The same argument also works if $ w^{k+1} $ is only an inexact global minimizer.

The next result shows that, even in the case where a limit point is not 
necessarily feasible, it still contains some useful information in the sense
that it is at least a stationary point for the constraint violation. In general,
this is the best that one can expect.

\begin{proposition}\label{Prop:MStatInf}
Suppose that the sequence $ \{ \varepsilon^k \} $ in \cref{Alg:ALM} is 
bounded. Then each accumulation
point $ \bar w $ of a sequence $ \{ w^k \} $ generated
by \cref{Alg:ALM} is an M-stationary point of the so-called feasibility problem
\begin{equation}\label{Eq:FeasP}
   \min_w \tfrac{1}{2} d_C^2 ( G(w) ) \quad \text{s.t.} \quad w \in D.
\end{equation}
\end{proposition}

\begin{proof}
In view of \cref{Prop:Feas1}, if $ \{ \rho_k \} $ is bounded,
then each accumulation point is a global minimum of the feasibility problem
\eqref{Eq:FeasP} and, therefore, an M-stationary point of this problem.

Hence, it remains to consider the case where $ \{ \rho_k \} $ is unbounded,
i.e.,
we have $ \rho_k \to \infty $ as $ k \to \infty $. 
In view of \cref{item:solve:subproblem_ALM,item:adjust_multiplier_ALM} of \cref{Alg:ALM}, see
also \eqref{Eq:lambdaup}, we have
\begin{equation*}
   \varepsilon^{k+1} \in \nabla f(w^{k+1}) + G'(w^{k+1})^* \lambda^{k+1} +
   \mathcal N_D^\textup{lim} (w^{k+1})
\end{equation*}
with $\lambda^{k+1}$ as in \cref{item:adjust_multiplier_ALM}.
Dividing this inclusion by $ \rho_k $ and using the fact that
$ \mathcal N_D^\textup{lim} (w^{k+1}) $ is a cone, we therefore get 
\begin{equation*}
   \frac{\varepsilon^{k+1}}{\rho_k} \in 
   \frac{\nabla f(w^{k+1})}{\rho_k} + G'(w^{k+1})^* 
   \bracks*{ G(w^{k+1}) + \frac{u^k}{\rho_k} - P_C \parens*{ G(w^{k+1}) + \frac{u^k}{\rho_k} 
   } } + \mathcal N_D^\textup{lim} (w^{k+1}).
\end{equation*}
Now, let $ \bar w $ be an accumulation point and $ \{ w^{k+1} \}_K $
be a subsequence satisfying $w^{k+1}\to_K \bar w $.
Then the sequences $ \{ \varepsilon^{k+1} \}_K$, $\{ u^k \}_K $, and
$ \{ \nabla f(w^{k+1}) \}_K $ are bounded. Thus, taking the limit
$ k \to_K \infty $ yields
\begin{equation*}
   0 \in G'(\bar w)^* \bracks[\big]{ G(\bar w) - P_C ( G(\bar w) ) } 
   + 
   \mathcal N_D^\textup{lim} (\bar w)
\end{equation*}
by the outer semicontinuity of the limiting normal cone. Since 
we also have $ \bar w \in D $ and due to 
\begin{equation*}
   \nabla \parens[\big]{ \tfrac12 d_C^2 \circ  G}(\bar w) = G'(\bar w)^* \bracks[\big]{ G(\bar w) - P_C ( G(\bar w) ) },
\end{equation*}
see, once more, \cite[Corollary~12.30]{BauschkeCombettes2011},
it follows that $ \bar w $ is an M-stationary point of the feasibility problem
\eqref{Eq:FeasP}.
\end{proof}

We next investigate suitable properties of feasible limit points. The following
may be viewed as the main observation in that respect and shows that any
such accumulation point is automatically an AM-stationary point in the sense of 
\cref{Def:AMStat}.

\begin{theorem}\label{Thm:Conv}
Suppose that the sequence $ \{ \varepsilon^k \} $ in \cref{Alg:ALM} satisfies
$ \varepsilon^k \to 0 $. Then each feasible accumulation
point $ \bar w $ of a sequence $ \{ w^k \} $ generated
by \cref{Alg:ALM} is an AM-stationary point.
\end{theorem}

\begin{proof}
Let $ \{ w^{k+1} \}_{K}$ denote a subsequence such that $ w^{k+1}\to_K\bar w $.
Define 
\begin{equation*}
   s^{k+1} := P_C \parens*{ G(w^{k+1} ) + \frac{u^k}{\rho_k} } \quad \text{and} \quad 
   z^{k+1} := G ( w^{k+1} ) - s^{k+1}
\end{equation*}
for each $ k \in \mathbb{N} $. We claim that the four (sub-) sequences
$ \{ w^{k+1} \}_K, \{ z^{k+1} \}_K$, $\{ \varepsilon^{k+1} \}_K $, and 
$ \{ \lambda^{k+1} \}_K $ generated by \cref{Alg:ALM} or defined in
the above way satisfy the properties from \cref{Def:AMStat} and
therefore show that $ \bar w $ is an AM-stationary point. By construction, we have
$w^{k+1}\to_K \bar w $ and $\varepsilon^{k+1}\to_K 0 $.
Further, from \cref{item:solve:subproblem_ALM} of \cref{Alg:ALM} and \eqref{Eq:lambdaup}, we obtain
\begin{equation*}
   \varepsilon^{k+1} \in \nabla_w \mathcal{L}_{\rho_k} (w^{k+1}, u^k) +
   \mathcal N_D^\textup{lim} (w^{k+1}) = \nabla f(w^{k+1}) + G'(w^{k+1})^*
   \lambda^{k+1} + \mathcal N_D^\textup{lim} (w^{k+1}).
\end{equation*}
Since $\mathcal N_C(s^{k+1})$ is a cone,
the relation between $P_C$ and $\mathcal N_C$
together with the
definitions of $ s^{k+1}, \lambda^{k+1} $, 
and $ z^{k+1} $ yield
\begin{equation*}
   \lambda^{k+1} = \rho_k \bracks*{  G(w^{k+1} ) + \frac{u^k}{\rho_k} - s^{k+1} }
   \in \mathcal{N}_C ( s^{k+1} ) = \mathcal{N}_C ( G(w^{k+1}) - z^{k+1} ) . 
\end{equation*}
Hence, it remains to show $z^{k+1}\to_K 0 $. To this end, we consider
two cases, namely whether $ \{ \rho_k \} $ stays bounded or is unbounded.
In the bounded case, \cref{item:update_penalty_ALM,item:update_2,item:update_3,item:update_4,item:update_5} 
of \cref{Alg:ALM} imply that $ V_{\rho_k}(w^{k+1}, u^k) \to 0 $ for
$ k \to \infty $. The corresponding definitions therefore yield
\begin{equation*}
   \norm{z^{k+1}} 
   = 
   \norm{G(w^{k+1}) - s^{k+1}}
   =
   V_{\rho_k}(w^{k+1}, u^k) \to 0 \quad \text{for }
   k \to_K \infty.
\end{equation*}
On the other hand, if $ \{ \rho_k \} $ is unbounded, we have $ \rho_k \to \infty $. 
Since $ \{ u^k \} $ is bounded by construction, the 
continuity of the projection operator together with the assumed
feasibility of $ \bar w $ implies
\begin{equation*}
   s^{k+1} = P_C \parens*{ G(w^{k+1} ) + \frac{u^k}{\rho_k} } \to P_C ( G (\bar w) )
   =G(\bar w)
   \quad \text{for } k \to_K \infty .
\end{equation*}
Consequently, we obtain $ z^{k+1} = G (w^{k+1}) - s^{k+1} \to_K 0 $ also in this case.
Altogether, this implies that $ \bar w $ is AM-stationary.
\end{proof}
	We point out that the proof of \cref{Thm:Conv} even shows the convergence	
	$\norm{z^{k+1}} 
	=
	V_{\rho_k}(w^{k+1}, u^k) \to_K 0$,
	i.e., the stopping criterion \eqref{eq:stopping_ALM}
	will be satisfied after finitely many steps.

Recalling that, by definition, each AM-stationary point of \eqref{Eq:GenP} which is AM-regular
must already be M-stationary, we obtain the following corollary.

\begin{corollary}\label{cor:ALM_yields_MSstationary_point}
	Suppose that the sequence $\{\varepsilon^k\}$ in \cref{Alg:ALM} satisfies
	$\varepsilon^k\to 0$. Then each feasible and AM-regular accumulation point $\bar w$ of a
	sequence $\{w^k\}$ generated by \cref{Alg:ALM} is an M-stationary
	point. 
\end{corollary}

Keeping our discussions after \cref{lem:AM_regularity_in_disjunctive_setting} in mind,
this result generalizes \cite[Theorem~3]{GuoDeng2021} which addresses a similar MPCC-tailored
augmented Lagrangian method and exploits MPCC-RCPLD.

\section{Realizations}\label{Sec:Realizations}

Let $k$ be a fixed iteration of \cref{Alg:ALM}.
For the (approximate) solution of the ALM-subproblem in \cref{item:solve:subproblem_ALM} of \cref{Alg:ALM},
we may use \cref{Alg:NPG}.
Recall that, given an outer iteration $j$ of \cref{Alg:NPG}, we need to solve the subproblem
\begin{equation*}
   \min_w \ \mathcal L_{\rho_k}(w^j,u^k) 
	 + \innerp{ \nabla_w \mathcal L_{\rho_k}(w^j,u^k) }{ w-w^j } + \frac{\gamma_{j,i}}{2} \norm{ w - w^j }^2 \quad
   \text{s.t.} \quad w \in D
\end{equation*}
with some given $ w^j $ and $ \gamma_{j,i} > 0 $ in the inner iteration $i$ of \cref{Alg:NPG}.
As pointed out in \cref{Sec:ProjGrad}, the above problem possesses the same solutions as
\begin{equation*}
   \min_w \ 
   		\norm*{
   			w - \parens*{ w^j - \frac{1}{\gamma_{j,i}} \nabla_w\mathcal L_{\rho_k}(w^j,u^k) }
   		} ^2
   \quad \text{s.t.} \quad w \in D,
\end{equation*}
i.e., we need to be able to compute elements of the (possibly multi-valued) projection 
$ \Pi_D \parens[\big]{ w^j - \frac{1}{\gamma_{j,i}} \nabla _w\mathcal L_{\rho_k}(w^j,u^k) } $.
Boiling this requirement down to its essentials, we have to be
in position to find projections of arbitrary points onto the set $ D $ in an efficient way.
Subsequently, this will be discussed in the context of several practically relevant settings.

\subsection{The Disjunctive Programming Case}

We consider \eqref{Eq:GenP} in the special \cref{set:disjunctive_programs_with_bi_disjunction}
with $\mathbb X:=\R^n$ and $X:=[\ell,u]$ where $\ell,u\in\R^n$ satisfy $-\infty\leq\ell_i<u_i\leq\infty$
for $i=1,\ldots,n$. Recall that the set $D$ is given by
\begin{equation}\label{eq:D_bi_disjunctive}
	D
	=
	\set{(x,y,z) \in \R^n \times \R^{m_3} \times \R^{m_3}  \given
		x\in[\ell,u],\,(y_i,z_i)\in T\quad\forall i\in\{1,\ldots,m_3}
	\}
\end{equation}
in this situation.
For given $\bar w=(\bar x,\bar y,\bar z)\in\R^n\times\R^{m_3}\times\R^{m_3}$, we want to characterize
the elements of $\Pi_D(\bar w)$. Therefore, we consider the optimization problem
\begin{equation}\label{eq:disjunctive_proj}
    \min_w \ \tfrac{1}{2} \norm{ w - \bar w }^2   \quad \text{s.t.} \quad w = (x,y,z) \in D.
\end{equation}
We observe that the latter can be decomposed into the $n$ one-dimensional optimization
problems
\begin{equation*}\label{eq:interval_proj}
   \min\limits_{x_i} \ \tfrac{1}{2} ( x_i - \bar x_i )^2 \quad \text{s.t.} \quad x_i \in [\ell_i,u_i],
\end{equation*}
$i=1,\ldots,n$, possessing the respective solution $P_{[\ell_i,u_i]}(\bar x_i)$, as well as 
into $m_3$ two-dimensional optimization problems
\begin{equation}\label{eq:bi_disjunctive_proj}
   \min\limits_{y_i,z_i} \ \tfrac{1}{2} ( y_i - \bar y_i )^2 + \tfrac{1}{2} ( z_i - \bar z_i )^2\quad \text{s.t.} 
   \quad (y_i,z_i)\in T,
\end{equation}
$i=1,\ldots,m_3$.
Due to $T=T_1\cup T_2$, 
each of these problems on its own can be decomposed into
the two two-dimensional subproblems
\begin{equation}\label{eq:bi_disjunctive_proj_sur}\tag{$R(i,j)$}
   \min\limits_{y_i,z_i}\ \tfrac{1}{2} ( y_i - \bar y_i )^2 + \tfrac{1}{2} ( z_i - \bar z_i )^2\quad \text{s.t.} 
   \quad (y_i,z_i)\in T_j,
\end{equation}
$j=1,2$. In most of the popular settings from disjunctive programming, \eqref{eq:bi_disjunctive_proj_sur}
can be solved with ease. By a simple comparison of the associated objective function values,
we find the solutions of \eqref{eq:bi_disjunctive_proj}. Putting the solutions of the subproblems
together, we find the solutions of \eqref{eq:disjunctive_proj}, i.e., the elements of $\Pi_D(\bar w)$.

In the remainder of this section, we consider a particularly interesting instance of this setting
where $T$ is given by
\begin{equation}\label{eq:BoxSwitching}
	T:=\set{(s,t) \given s\in[\sigma_1,\sigma_2],\,t\in[\tau_1,\tau_2],\,st=0}.
\end{equation}
Here, $-\infty\leq\sigma_1,\tau_1\leq 0$ and $0<\sigma_2,\tau_2\leq\infty$ are given constants.
Particularly, we find the decomposition
\[
	T_1:=[\sigma_1,\sigma_2]\times\{0\},\qquad T_2:=\{0\}\times[\tau_1,\tau_2]
\]
of $T$ in this case. Due to the geometrical shape of the set $T$, one might be tempted to refer to
this setting as ``box-switching constraints''.
Note that it particularly covers
\begin{itemize}
	\item switching constraints ($\sigma_1=\tau_1:=-\infty$, $\sigma_2=\tau_2:=\infty$),
		see \cite{KanzowMehlitzSteck2019,Mehlitz2020a},
	\item complementarity constraints ($\sigma_1=\tau_1:=0$, $\sigma_2=\tau_2:=\infty$),
		see \cite{LuoPangRalph1996,OutrataKocvaraZowe1998}, and
	\item relaxed reformulated cardinality constraints ($\sigma_1:=-\infty$, $\sigma_2:=\infty$,
		$\tau_1:=0$, $\tau_2:=1$), see \cite{BurdakovKanzowSchwartz2016,CervinkaKanzowSchwartz2016}.
\end{itemize}
We refer the reader to \cref{fig:BoxSwitching} for a visualization of these types of constraints.

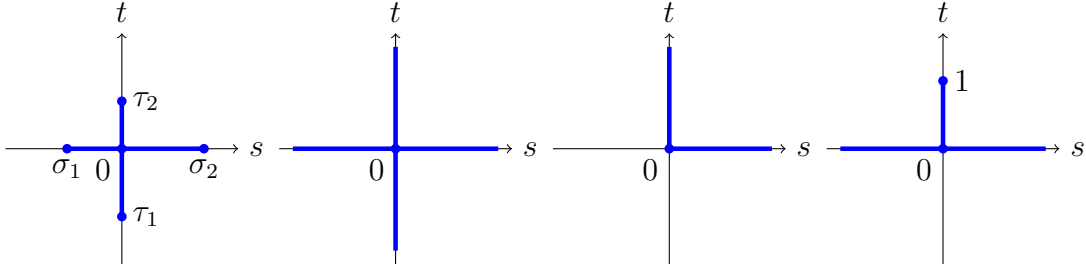
\begin{figure}[htp]\centering
 \begin{tikzpicture}[scale=0.9]
  \draw[->] (-1.7,0) -- (1.7,0) node[right] {$s$};
  \draw[->] (0,-1.7) -- (0,1.7) node[above] {$t$};  
  \draw[very thick, color = blue] (-0.8,0.01) -- (1.2,0.01);
  \draw[very thick, color = blue] (-0.8,-0.01) -- (1.2,-0.01);
  \draw[very thick, color = blue] (0.01,-1.0) -- (0.01,0.7);
  \draw[very thick, color = blue] (-0.01,-1.0) -- (-0.01,0.7);
  \fill[color=blue] (0,0) circle [radius=2pt];
  \fill[color=blue] (-0.8,0) circle [radius=2pt];
  \fill[color=blue] (1.2,0) circle [radius=2pt];
  \fill[color=blue] (0,0.7) circle [radius=2pt];
  \fill[color=blue] (0,-1) circle [radius=2pt];
  \node[below] at (-0.8,0){$\sigma_1$};
  \node[below] at (1.2,0){$\sigma_2$};
  \node[right] at (0,-1){$\tau_1$};
  \node[right] at (0,0.7){$\tau_2$};  
  \node[below left] at (0,0){$0$};
  \draw[->] (2.3,0) -- (5.7,0) node[right] {$s$};
  \draw[->] (4.0,-1.7) -- (4.0,1.7) node[above] {$t$};  
  \draw[very thick, color = blue] (2.5,0.01) -- (5.5,0.01);
  \draw[very thick, color = blue] (2.5,-0.01) -- (5.5,-0.01);
  \draw[very thick, color = blue] (4.01,-1.5) -- (4.01,1.5);
  \draw[very thick, color = blue] (3.99,-1.5) -- (3.99,1.5);
  \fill[color=blue] (4,0) circle [radius=2pt];
  \node[below left] at (4,0){$0$};
  \draw[->] (6.3,0) -- (9.7,0) node[right] {$s$};
  \draw[->] (8.0,-1.7) -- (8.0,1.7) node[above] {$t$};  
  \draw[very thick, color = blue] (7.99,0) -- (7.99,1.5);
  \draw[very thick, color = blue] (8.01,0) -- (8.01,1.5);
  \draw[very thick, color = blue] (8,-0.01) -- (9.5,-0.01);
  \draw[very thick, color = blue] (8,0.01) -- (9.5,0.01);
  \fill[color=blue] (8,0) circle [radius=2pt];
  \node[below left] at (8,0){$0$};
   \draw[->] (10.3,0) -- (13.7,0) node[right] {$s$};
  \draw[->] (12.0,-1.7) -- (12.0,1.7) node[above] {$t$};  
  \draw[very thick, color = blue] (10.5,0.01) -- (13.5,0.01);
  \draw[very thick, color = blue] (10.5,-0.01) -- (13.5,-0.01);
  \draw[very thick, color = blue] (12.01,0) -- (12.01,1);
  \draw[very thick, color = blue] (11.99,0) -- (11.99,1);
  \fill[color=blue] (12,0) circle [radius=2pt];
  \node[below left] at (12,0){$0$};
  \fill[color=blue] (12,1) circle [radius=2pt];
  \node[right] at (12,1){$1$};
 \end{tikzpicture}
\caption{Geometric illustrations of box-switching, switching, complementarity, and
relaxed reformulated cardinality constraints (from left to right), respectively.}
\label{fig:BoxSwitching}
\end{figure}

One can easily check that the solutions of 
\hyperref[eq:bi_disjunctive_proj_sur]{(R$(i,1)$)}
and
\hyperref[eq:bi_disjunctive_proj_sur]{(R$(i,2)$)}
are given by 
$(P_{[\sigma_1,\sigma_2]}(\bar y_i),0)$ and $(0,P_{[\tau_1,\tau_2]}(\bar z_i))$, respectively.
This yields the following result.

\begin{proposition}\label{Prop:ProjBoxSwitching}
Consider the set $D$ from \eqref{eq:D_bi_disjunctive} where $T$ is
given as in \eqref{eq:BoxSwitching}.
For given 
$ \bar w = ( \bar x, \bar y, \bar z ) \in\R^n\times\R^{m_3}\times\R^{m_3}$,
we have
$\hat w:= (\hat x,\hat y,\hat z)\in \Pi_D(\bar w)$ if and only if
$\hat x=P_{[\ell,u]}(\bar x)$ and
\begin{equation*}
	(\hat y_i,\hat z_i)
	\in
	\begin{cases}
		\{(P_{[\sigma_1,\sigma_2]}(\bar y_i),0)\}
			&\text{if }\phi_s(\bar y_i,\bar z_i)<\phi_t(\bar y_i,\bar z_i),\\
		\{(0,P_{[\tau_1,\tau_2]}(\bar z_i))\}
			&\text{if }\phi_s(\bar y_i,\bar z_i)>\phi_t(\bar y_i,\bar z_i),\\
		\{(P_{[\sigma_1,\sigma_2]}(\bar y_i),0),(0,P_{[\tau_1,\tau_2]}(\bar z_i))\}
			&\text{if }\phi_s(\bar y_i,\bar z_i)=\phi_t(\bar y_i,\bar z_i)
	\end{cases}
\end{equation*}
for all $i=1,\ldots,m_3$, where we used
\begin{align*}
	\phi_s(a,b):=(P_{[\sigma_1,\sigma_2]}(a)-a)^2+b^2,
	\qquad
	\phi_t(a,b):=a^2+(P_{[\tau_1,\tau_2]}(b)-b)^2.
\end{align*}
\end{proposition}

Particularly, it turns out that in order to compute the projections onto the set $D$
under consideration, one basically needs to compute $n+2m_3$ projections onto real intervals.
In the specific setting of complementarity-constrained programming, this already has been
observed in \cite[Section~4]{GuoDeng2021}.

Let us briefly mention that other popular instances of disjunctive programs like vanishing- and
or-constrained optimization problems, 
see e.g.\ \cite{AchtzigerKanzow2008,Mehlitz2020b},
where $T$ is given by
\[
	T:=\set{(s,t) \given st\leq 0,\,t\geq 0}
\quad
\text{or}
\quad
	T:=\set{(s,t) \given \min(s,t)\leq 0},
\]
respectively, can be treated in an analogous fashion.
Furthermore, an analogous procedure applies to more general situations where $T$ is the union of 
	finitely many convex, polyhedral sets.

\subsection{The Sparsity-Constrained Case}

We fix $\W:=\R^n$ and some $\kappa\in\N$ with $1\leq \kappa\leq n-1$.
Consider the set 
\begin{equation*}
    S_\kappa := \set[\big]{ w \in \R^n \given \norm{ w }_0 \leq \kappa }
\end{equation*}
with
 $ \norm{ w }_0 $ being the number of nonzero entries of the vector $ w $.
This set plays a prominent role in sparse optimization and for problems with cardinality constraints.
Since $ S_\kappa $ is nonempty and closed, projections of some vector $w\in\R^n$
(w.r.t.\ the Euclidean norm) onto this set 
exist (but may not be unique), and are known to consist of those vectors $ y\in\R^n $ such that
the nonzero entries of $ y $ are precisely the $ \kappa $ largest (in absolute value) components of 
$ w $ (which may not be unique), see e.g.\ \cite[Proposition~3.6]{BauschkeLukePhanWang2013}.

Hence, within our augmented Lagrangian framework, we may take $ D := S_\kappa $ and then get an 
explicit formula for the solutions of the corresponding subproblems arising within the
spectral gradient method. However, typical implementations of augmented Lagrangian
methods (like \texttt{ALGENCAN}, see \cite{AndreaniBirginMartinezSchuverdt2008}) 
do not penalize box constraints, i.e., they leave the box constraints
explicitly as constraints when solving the corresponding subproblems. 
Hence, let us assume that we have some
lower and upper bounds satisfying $ - \infty \leq \ell_i < u_i \leq \infty $ for all 
$ i = 1, \ldots, n $. We are then forced to compute projections onto the set
\begin{equation}\label{Eq:CsBox}
    D := S_\kappa \cap [\ell,u].
\end{equation}
It turns out that there exists an explicit formula for this projection. Before presenting
the result, let us first assume, for notational simplicity, that 
\begin{equation}\label{Eq:NullAss}
    0 \in [\ell_i,u_i] \qquad \forall i = 1, \ldots, n.
\end{equation}
We mention that this assumption is not restrictive.
Indeed, let us assume that, e.g.,
$0 \not\in [\ell_1, u_1]$.
Then the first component of $w \in D$ cannot be zero,
and this shows
\begin{equation}
	\label{eq:split_sparsity}
	D
	=
	S_\kappa \cap [\ell, u]
	=
	[\ell_1, u_1]
	\times
	\parens*{
		\hat S_{\kappa-1} \cap [\hat \ell, \hat u]
	},
\end{equation}
where
$ \hat S_{\kappa-1} := \set{ w \in \R^{n-1} \given \norm{w}_0 \le \kappa-1}$
and the vectors
$\hat\ell, \hat u \in \R^{n-1}$ are
obtained from $\ell, u$ by dropping the first component,
respectively.
For the computation of
the projection onto $S_\kappa$,
we can now exploit the product structure \eqref{eq:split_sparsity}.
Similarly, we can remove all remaining components $i=2,\ldots,n$ with $0 \not\in [\ell_i, u_i]$ from $D$.
Thus,
we can assume \eqref{Eq:NullAss}
without loss of generality.

We begin with a simple observation.
\begin{lemma}
	\label{lem:structure_of_projection}
	Let $w \in \R^n$ be arbitrary.
	Then, for each $y\in\Pi_D(w)$,
	where $D$ is the set from \eqref{Eq:CsBox}, we have 
	\begin{equation*}
		y_i \in
		\set*{
			0,
			P_{[\ell_i, u_i]}(w_i)
		}
		\qquad \forall i = 1,\ldots, n
		.
	\end{equation*}
\end{lemma}
\begin{proof}
	To the contrary,
	assume that $y_i \ne 0$ and $y_i \ne P_{[\ell_i, u_i]}(w_i)$
	hold for some index $i\in\{1,\ldots,n\}$.
	Define the vector $q \in \R^n$
	by $q_j := y_j$ for $j \ne i$
	and $q_i := P_{[\ell_i, u_i]}(w_i)$.
	Due to $y_i \ne 0$, we have $\norm{q}_0 \le \norm{y}_0\le \kappa$,
	i.e., $q \in S_\kappa$.
	Additionally, $q\in[\ell,u]$ is clear from $y\in[\ell,u]$ and $q_i=P_{[\ell_i,u_i]}(w_i)$.
	Thus, we find $q\in D$.
	Furthermore,
	$\norm{q - w} < \norm{y - w}$
	since
	$q_i = P_{[\ell_i, u_i]}(w_i) \ne y_i$.
	This contradicts the fact that $y$ is a projection
	of $w$ onto $D$.
\end{proof}

Due to the above lemma,
we only have two choices
for the value of the components
associated with projections to $D$ from \eqref{Eq:CsBox}.
Thus, for an arbitrary index set
$I \subset \set{1,\ldots,n}$
and an arbitrary vector $w\in\R^n$,
we define $p^I(w) \in \R^n$ via
\begin{equation*}
	p^I_i(w) :=
	\begin{cases}
		P_{[\ell_i,u_i]}(w_i) & \text{if } i \in I, \\
		0 & \text{otherwise}
	\end{cases}
	\qquad
	\forall i=1,\ldots,n.
\end{equation*}
It remains to characterize those index sets $I$
which ensure that $p^I(w)$ is a projection of $w$ onto $D$.
To this end, we define an auxiliary vector $d(w) \in \R^n$ via
\begin{equation*}
	d_i(w) 
	:=
	w_i^2
	-
	\parens[\big]{ P_{[\ell_i,u_i]} (w_i) - w_i }^2
	\qquad \forall i = 1, \ldots, n.
\end{equation*}
Note that this definition directly yields
\begin{equation}
	\label{eq:distance_via_d}
	\norm{p^I(w) - w}^2
	=
	\norm{w}^2 - \sum_{i \in I} d_i(w)
	.
\end{equation}
We state the following simple observation.
\begin{lemma}
	\label{Lem:Simpledx}
	Fix $ w \in \R^n $ and assume that 	\eqref{Eq:NullAss} is valid. 
	Then the following statements hold:
	\begin{enumerate}[label=(\alph*)]
		\item
			\label{item:Simpledx_1}
			$ d_i(w) \geq 0 $ for all $ i = 1, \ldots, n $,
		\item
			\label{item:Simpledx_2}
			$ d_i(w) = 0 \Longleftrightarrow P_{[\ell_i,u_i]} (w_i) = 0 $.
	\end{enumerate}
\end{lemma}

\begin{proof}
\ref{item:Simpledx_1}
Since $ 0 \in [\ell_i,u_i] $, we obtain
\begin{equation*}
   d_i(w) = ( w_i - 0 )^2 - \parens[\big]{ w_i - P_{[\ell_i,u_i]} (w_i) }^2 \geq 0
\end{equation*}
by definition of the (one-dimensional) projection.\smallskip 

\noindent
\ref{item:Simpledx_2}
If $ P_{[\ell_i,u_i]} (w_i) = 0 $ holds, we immediately obtain $ d_i(w) = 0 $. Conversely,
let $ d_i(w) = 0 $. Then 
\begin{equation*}
   0 = w_i^2 - \parens[\big]{ w_i - P_{[\ell_i,u_i]} (w_i) }^2 = P_{[\ell_i,u_i]} (w_i) \parens[\big]{
   2 w_i - P_{[\ell_i,u_i]} (w_i) }.
\end{equation*}
Hence, we find $ P_{[\ell_i,u_i]} (w_i) = 0 $ or $ P_{[\ell_i,u_i]} (w_i) = 2 w_i $. In the first
case, we are done.
In the second case,
we have $\set{0, 2 w_i} \subset [\ell_i, u_i]$.
By convexity, this gives $w_i \in [\ell_i, u_i]$.
Consequently,
$w_i = P_{[\ell_i, u_i]}(w_i) = 2 w_i$.
This implies
$P_{[\ell_i, u_i]}(w_i) = 0$.
\end{proof}

Observe that the second assertion of the above lemma implies
\begin{equation}
	\label{eq:sparsity_via_d}
	\norm{p^I(w)}_0
	=
	\abs[\big]{\set{i \in I \given P_{[\ell_i, u_i]} (w_i) \ne 0}}
	=
	\abs[\big]{\set{i \in I \given d_i(w) \ne 0}}
	\qquad
	\forall w\in\R^n.
\end{equation}
This can be used to characterize the set of projections onto the set $D$ from
\eqref{Eq:CsBox}.

\begin{proposition}\label{Prop:ProjCsBox2}
	Let $ D $ be the set from \eqref{Eq:CsBox} and assume that
	\eqref{Eq:NullAss} holds.
	Then, for each $w\in\R^n$,
	$ y \in \Pi_D (w) $ holds if and only if there exists an index set 
	$ I \subset \{ 1, \ldots, n \} $
	with $ | I | = \kappa $ such that
	\begin{equation}\label{Eq:Proj4}
		d_i(w) \geq d_j(w) \qquad \forall i \in I, \ \forall j \not\in I
	\end{equation}
	and 
	$y = p^I(w)$ hold.
	
\end{proposition}
\begin{proof}
	If $y\in\Pi_D(w)$ holds,
	then
	$y = p^J(w)$ is valid for some index set $J$, see \cref{lem:structure_of_projection}.
	Thus, it remains to check that $p^J(w)$ is a projection onto $D$
	if and only if $p^J(w) = p^I(w)$
	holds for some index set $I$
	satisfying $\abs{I} = \kappa$ and \eqref{Eq:Proj4}.

	Note that $p^J(w)$ is a projection
	if and only if
	$J$
	minimizes
	$\norm{p^I(w) - w}$
	over all $I \subset \set{1,\ldots,n}$
	satisfying
	$\norm{p^I(w)}_0 \le \kappa$.
	This can be reformulated via $d(w)$ by using
	\eqref{eq:distance_via_d} and \eqref{eq:sparsity_via_d}.
	In particular,
	$p^J(w)$ is a projection
	if and only if $J$ solves
	\begin{equation}\label{eq:problem_char_proj_Cs}
		\max\limits_I \quad \sum_{i \in I} d_i(w)
		\quad
		\text{s.t.} \quad I \subset \set{1,\ldots,n},
		\quad \abs[\big]{\set{i \in I \given d_i(w) \ne 0}} \le \kappa.
	\end{equation}
	It is clear that index sets $I$ with $\abs{I} = \kappa$
	and \eqref{Eq:Proj4}
	are solutions of this problem.
	This shows the direction $\Longleftarrow$.

	To prove the converse direction $\Longrightarrow$,
	let $p^J(w)$ be a projection.
	Thus, $J$ solves \eqref{eq:problem_char_proj_Cs}.
	We note that the solutions of this problem
	are invariant under addition and removal
	of indices $i$ with
	$d_i(w) = 0$.
	Due to \cref{Lem:Simpledx}~\ref{item:Simpledx_2},
	these operations also do not alter
	the associated $p^I(w)$.
	Thus, for each projection $p^J(w)$,
	we can add or remove indices $i$ with $d_i(w) = 0$,
	to obtain a set $I$ with $p^I(w) = p^J(w)$
	and $\abs{I} = \kappa$.
	It is also clear that \eqref{Eq:Proj4}
	holds for such a choice of $I$.
\end{proof}

Below, we comment on the result of \cref{Prop:ProjCsBox2}.
\begin{remark}\label{rem:projection_onto_CsBox}
	\begin{enumerate}
		\item[(a)] Let $y=p^I(w)$ be a projection of $w\in\R^n$ onto $D$ from \eqref{Eq:CsBox}
			such that \eqref{Eq:NullAss} holds. Observe that $y_i=0$ may also hold for
			some indices $i\in I$.
		\item[(b)] In the unconstrained case $[\ell,u] = \R^n$, we find $d_i(w) = w_i^2$ for each
			$w\in\R^n$ and all $i=1,\ldots,n$.
			Thus, \cref{Prop:ProjCsBox2} recovers the well-known characterization of
			the projection onto the set $S_\kappa$ which can be found in
			\cite[Proposition~3.6]{BauschkeLukePhanWang2013}.
	\end{enumerate}
\end{remark}

We want to close this section with some brief remarks regarding the variational geometry of
$D=S_k\cap[\ell,u]$ from \eqref{Eq:CsBox}. Observing that the sets $S_\kappa$ and $[\ell,u]$ are both polyhedral
in the sense that they can be represented as the union of finitely many polyhedrons, the
normal cone intersection rule
\[
	\mathcal N^\textup{lim}_D(w)
	=
	\mathcal N^\textup{lim}_{S_\kappa\cap[\ell,u]}(w)
	\subset
	\mathcal N^\textup{lim}_{S_\kappa}(w)+\mathcal N^\textup{lim}_{[\ell,u]}(w)
	=
	\mathcal N^\textup{lim}_{S_\kappa}(w)+\mathcal N_{[\ell,u]}(w)
\]
applies for each $w\in D$
by means of \cite[Corollary~4.2]{HenrionJouraniOutrata2002} and
\cite[Proposition~1]{Robinson1981}. While the evaluation of $\mathcal N_{[\ell,u]}(w)$
is standard, a formula for $\mathcal N^\textup{lim}_{S_\kappa}(w)$ can be found in
\cite[Theorem~3.9]{BauschkeLukePhanWang2013}.

\subsection{Low-Rank Approximation}
\subsubsection{General Low-Rank Approximations}
For natural numbers $m,n\in\N$ with $m,n\geq 2$, we fix $\mathbb W:=\R^{m\times n}$.
Equipped with the standard Frobenius inner product, $\mathbb W$ indeed is a Euclidean space.
Now, for fixed $\kappa\in\N$ satisfying $1\leq\kappa\leq\min(m,n)-1$, let us investigate
the set
\[
	D:=\set{W\in\mathbb W \given \rank W\leq\kappa}.
\]
Constraint systems involving rank constraints of type $W\in D$ can be used to model numerous
practically relevant problems in computer vision, machine learning, computer algebra, signal processing,
or model order reduction, see \cite[Section~1.3]{Markovsky2012} for an overview.
Nowadays, one of the most popular applications behind low-rank constraints is the so-called
low-rank matrix completion, particularly, the ``Netflix-problem'', 
see \cite{CandesRecht2009} for details.

Observe that the variational geometry of $D$ has been explored recently in
\cite{HosseiniLukeUschmajew2019}. Particularly, a formula for the limiting normal cone to
this set can be found in \cite[Theorem~3.1]{HosseiniLukeUschmajew2019}.
Using the singular value decomposition of a given matrix $\widetilde W\in\mathbb W$, one can
easily construct an element of $\Pi_{D}(\widetilde W)$ by means of the so-called 
Eckart--Young--Mirsky theorem, see e.g.\ \cite[Theorem~2.23]{Markovsky2012}.

\begin{proposition}\label{prop:ProjLowRank}
	For a given matrix $\widetilde W\in\mathbb W$, 
	let $\widetilde W=U\Sigma V^\top$ be its singular value
	decomposition with orthogonal matrices $U\in\R^{m\times m}$ and $V\in\R^{n\times n}$ as
	well as a diagonal matrix $\Sigma\in\R^{m\times n}$ whose diagonal entries are in
	non-increasing order. 
	Let $\widehat U\in\R^{m\times\kappa}$ and $\widehat V\in\R^{n\times\kappa}$ 
	be the matrices resulting from
	$U$ and $V$ by deleting the last $m-\kappa$ and $n-\kappa$ columns, respectively.
	Furthermore, let $\widehat\Sigma\in\R^{\kappa\times\kappa}$ be the top left $\kappa\times\kappa$ block
	of $\Sigma$. Then we have $\widehat U\widehat\Sigma\widehat{V}^\top\in \Pi_{D}(\widetilde W)$.
\end{proposition}

Note that the projection formulas from the previous sections allow 
a very efficient computation of the corresponding projections, which is
in contrast to the projection provided by \cref{prop:ProjLowRank}. Though
the formula given there is conceptually very simple, its realization 
requires to compute the singular value decomposition of the given
matrix.

\subsubsection{Symmetric Low-Rank Approximation}

Given
$n\in\N$ with $n\geq 2$, we consider
the set of symmetric matrices
$\mathbb W:=\R^{n\times n}_{\textup{sym}}$, still
equipped with the Frobenius inner product.
Now, for fixed $\kappa\in\N$ satisfying $1\leq\kappa\leq n$, let us investigate
the set
\[
	D:=\set{W\in\mathbb W \given W \succeq 0,\; \rank W\leq\kappa}.
\]
Above, the constraint $W\succeq 0$ is used to abbreviate that $W$ has to be 
positive semidefinite.
Constraint systems involving rank constraints of type $W\in D$ arise frequently
in several different mathematical models of data science, see \cite{LemonSoYe2016} for an overview, and \cref{Sub:MAXCUT}
for an application.
Note that $\kappa:=n$ covers the setting of pure semidefiniteness constraints.

Exploiting the eigenvalue decomposition of a given matrix $\widetilde W\in \W$, one can
easily construct an element of $\Pi_{D}(\widetilde W)$.

\begin{proposition}\label{prop:SymProjLowRank}
	For a given matrix $\widetilde W\in \W$, 
	we denote by $\widetilde W= \sum_{i = 1}^n \lambda_i v_i v_i^\top$ its 
	(orthonormal) eigenvalue decomposition
	with non-increasingly ordered eigenvalues
	$\lambda_1 \ge \lambda_2 \ge \ldots \ge \lambda_n$
	and associated pairwise orthonormal eigenvectors $v_1,\ldots,v_n$.
	Then we have
	$\widehat W := \sum_{i = 1}^\kappa \max(\lambda_i, 0) v_i v_i^\top \in \Pi_{D}(\widetilde W)$.
\end{proposition}

\begin{proof}
  We define the positive and negative part 
  $\widetilde W^\pm := \sum_{i = 1}^n \max(\pm\lambda_i, 0) v_i v_i^\top$.
  This yields $\widetilde W = \widetilde W^+ - \widetilde W^-$ and
  $\innerp{\widetilde W^+}{\widetilde W^-} = \trace( \widetilde W^+ \widetilde W^-) = 0$.
  Thus, for each positive semidefinite $B \in \W$, we have
  \begin{equation*}
	\norm{ \widetilde W - B }^2
	=
	\norm{ \widetilde W^+ - B }^2 + \norm{ \widetilde W^- }^2 + 2 \innerp{\widetilde W^-}{B}
	\ge
	\norm{ \widetilde W^+ - B }^2 + \norm{ \widetilde W^- }^2
	.
  \end{equation*}
  Since the singular value decomposition of $\widetilde W^+$ coincides with the eigenvalue decomposition,
  the right-hand side is minimized by $B = \widehat W$, see \cref{prop:ProjLowRank} while noting
  that we have $\widehat W=\widetilde W^+$ in case $\kappa=n$.
  Due to $\innerp{\widetilde W^-}{\widehat W} = 0$,
  $B = \widehat W$ also minimizes the left-hand side.
\end{proof}
It is clear that
the computation of the $\kappa$ largest eigenvalues of $\widetilde W\in\mathbb W$
is sufficient to compute an element from the projection
$\Pi_{D}(\widetilde W)$. This can be done particularly efficient 
for small $ \kappa $ (note that $ \kappa = 1 $ holds in our application
from \cref{Sub:MAXCUT}).

\subsection{Extension to Nonsmooth Objectives}

For some lower semicontinuous functional $q\colon\mathbb W\to\R$, we consider the optimization
problem
\begin{equation}\label{eq:non_lipschitz_optimization}
	\min_w \ f(w)+q(w) \quad \text{s.t.} \quad G(w) \in C.
\end{equation}
Particularly, we do not assume that $q$ is continuous.
Exemplary, let us mention the special cases where $q$ is the indicator function
of a closed set, counts the nonzero entries
of the argument vector (in case $\mathbb W:=\R^n$), or encodes the rank of the
argument matrix (in case $\mathbb W:=\R^{m\times n}$). In this regard, 
\eqref{eq:non_lipschitz_optimization} can be used to model real-world applications
from e.g.\ image restoration or signal processing.
Necessary optimality conditions and qualification conditions addressing
\eqref{eq:non_lipschitz_optimization} can be found in \cite{GuoYe2018}.
In \cite{ChenGuoLuYe2017}, the authors suggest to handle \eqref{eq:non_lipschitz_optimization}
numerically with the aid of an augmented Lagrangian method (without safeguarding) based on the (partially)
augmented Lagrangian function \eqref{eq:augmented_Lagrangian} and the subproblems
\[
	\min\limits_w \ \mathcal L_{\rho_k}(w,\lambda^k)+q(w)\quad\text{s.t.}\quad w\in\mathbb W
\]
which are solved with a nonmonotone proximal gradient method inspired by \cite{WrightNowak2009}.
In this regard, the solution approach to \eqref{eq:non_lipschitz_optimization} described in
\cite{ChenGuoLuYe2017} possesses some parallels to our strategy for the numerical solution
of \eqref{Eq:GenP}. The authors in \cite{ChenGuoLuYe2017} were able to prove convergence of
their method to reasonable stationary points of \eqref{eq:non_lipschitz_optimization}
under a variant of the basic qualification condition and RCPLD.
Let us mention that the authors in \cite{GuoYe2018,ChenGuoLuYe2017} only considered standard
inequality and equality constraints, but the theory in these papers can be easily extended
to the more general constraints considered in \eqref{eq:non_lipschitz_optimization} 
doing some nearby adjustments.

We note that \eqref{Eq:GenP} can be interpreted as a special instance of
\eqref{eq:non_lipschitz_optimization} where $q$ plays the role of the indicator function of the
set $D$. Then the nonmonotone proximal gradient method from \cite{ChenGuoLuYe2017} reduces
to the spectral gradient method from \cref{Sec:ProjGrad}. 
However, the authors in \cite{ChenGuoLuYe2017} did not challenge their method with discontinuous
functionals $q$ and, thus, cut away some of the more reasonable applications behind the
model \eqref{Eq:GenP}. Furthermore, we would like to mention that \eqref{eq:non_lipschitz_optimization}
can be reformulated
(by using 
the epigraph 
$\epi q:=\set{(w,\alpha) \given q(w)\leq\alpha}$ of $q$)
as
\begin{equation}\label{eq:non_Lipschitzian_problem_epi_reformulation}
	\min\limits_{w,\alpha} \ f(w)+\alpha\quad\text{s.t.}\quad G(w)\in C,\ (w,\alpha)\in\epi q
\end{equation}
which is a problem of type \eqref{Eq:GenP}.
One can easily check that \eqref{eq:non_lipschitz_optimization} and
\eqref{eq:non_Lipschitzian_problem_epi_reformulation} are equivalent in the sense that 
$\bar w\in\mathbb W$ is a local/global minimizer of \eqref{eq:non_lipschitz_optimization} if and only
if $(\bar w,q(\bar w))$ is a local/global minimizer of
\eqref{eq:non_Lipschitzian_problem_epi_reformulation}.
Problem \eqref{eq:non_Lipschitzian_problem_epi_reformulation} can be handled with
\cref{Alg:ALM} as soon as the computation of projections onto $ D := \epi q$ is possible
in an efficient way.
Our result from \cref{cor:ALM_yields_MSstationary_point} shows that 
\cref{Alg:ALM} applied to \eqref{eq:non_Lipschitzian_problem_epi_reformulation} computes M-stationary
points of \eqref{eq:non_lipschitz_optimization} under 
	AM-regularity (associated with \eqref{eq:non_Lipschitzian_problem_epi_reformulation}
	at $(\bar w,q(\bar w))$),
i.e., we are in
position to find points satisfying
\[
	0\in\nabla f(\bar w)+\partial q(\bar w)+G'(\bar w)^*\mathcal N_C(G(\bar w))
\]
under a very mild condition which enhances \cite[Theorem~3.1]{ChenGuoLuYe2017}.
Here, we used the limiting subdifferential of $q$ given by
\[
	\partial q(w)
	:=
	\set{\xi\in\mathbb W \given (\xi,-1)\in\mathcal N_{\epi q}^\textup{lim}(w,q(w))}.
\]

\section{Numerical Results}\label{Sec:Numerics}

We implemented \cref{Alg:ALM}, based on the underlying subproblem
solver \cref{Alg:NPG}, in MATLAB (R2021b) and tested it on 
three classes of difficult problems which are discussed in 
\cref{Sub:MPCC,Sub:Cardinality,Sub:MAXCUT}. All test runs use the 
following parameters:
\begin{equation*}
   \tau := 2,\, \sigma := 10^{-4},\,  \beta := 10,\,\eta := 0.8,\,
   m := 10, \, \gamma_{\min} := 10^{-10}, \, \gamma_{\max} := 10^{10}.
\end{equation*}
In iteration $k$ of \cref{Alg:ALM},
we terminate \cref{Alg:NPG}
% at the iteration $j$
if the inner iterates $w^{j,i}$
satisfy
\begin{equation*}
	\norm*{
		\gamma_{j,i} \parens[\big]{ w^j - w^{j,i} }
		+
		\nabla\varphi(w^{j,i}) - \nabla\varphi(w^j)
	}
	_{\infty} \leq \frac{10^{-4}}{\sqrt{k+1}},
\end{equation*}
where $ \norm{\,\cdot\,}_{\infty} $ stands for the maximum-norm
for both $\W$ equal to $ \mathbb{R}^n $ and equal to 
$ \mathbb{R}^{n \times n}_{\textup{sym}} $ (other Euclidean spaces do not occur
in the subsequent applications),
see \eqref{eq:inner_inner_termination}.
Similarly, we use the infinity norm in the definition \eqref{eq:def_V} of $V_\rho$.
	\cref{Alg:ALM} is terminated as soon as
	\eqref{eq:stopping_ALM}
	is satisfied with $\varepsilon_{\textup{tol}} := 10^{-4}$.
	These two termination criteria ensure
	that the final iterate $w^{k}$
	together with the multiplier $\lambda^k$
	is approximately M-stationary, see \eqref{eq:termination}.

Given an arbitrary (possibly random) starting point $ w^0 $, we note that we first project this point 
onto the set $ D $ and then use this projected point as the true starting point,
so that all iterates $ w^k $ generated by \cref{Alg:ALM} belong to $ D $. The choice of the initial penalty parameter is similar to the rule 
in \cite[p.\ 153]{BirginMartinez2014} and given by
\begin{equation*}
   \rho_0 := P_{[10^{-3}, 10^3]} \parens*{ 10 \frac{\max ( 1, f(w^0) )}{\max
   \parens[\big]{ 1, \tfrac{1}{2} d_C^2( G (w^0)) }} } .
\end{equation*}
In all our examples, the space $\mathbb Y$ is given by $\R^m$ as
in \cref{set:standard_nonlinear_constraints}.
This allows us to choose the safeguarded multiplier
estimate $u^k$ as the projection of the current 
value $ \lambda^k $ onto a given box $ [u_{\min}, u_{\max}] $, where this
box is (in componentwise fashion) 
chosen to be $ [-10^{20},10^{20}] $ for all equality constraints 
and $ [0,10^{20}] $ for all inequality constraints. In this 
way, we basically guarantee that the safeguarded augmented Lagrangian method
from \cref{Alg:ALM} coincides with the classical approach as long 
as bounded multiplier estimates $ \lambda^k $ are generated.

\subsection{MPCC Examples}\label{Sub:MPCC}

The specification of \cref{Alg:ALM} to MPCCs is essentially the 
method discussed in \cite{GuoDeng2021}, where extensive numerical results
(including comparisons with other methods) are presented. We therefore
keep this section short and consider only two particular examples 
in order to illustrate certain aspects of our method.

\begin{example}\label{Ex:Scholtes}
Here, for $w:=(y,z) \in\R^2$, we consider the two-dimensional MPCC given by
\begin{equation*}
   \min_w \tfrac{1}{2} (y-1)^2 + \tfrac{1}{2} (z-1)^2 \quad 
   \text{s.t.} \quad y+z \leq 2,\, y \geq 0,\, z \geq 0, y z = 0,
\end{equation*}
which is essentially the example from \cite{Scholtes2001} 
with an additional (inactive) inequality constraint in order to 
have at least one standard constraint, so that \cref{Alg:ALM}
does not automatically reduce to the spectral gradient method.
The problem possesses two global minimizers at $ (0,1) $ and $ (1,0) $ which are 
M-stationary (in fact, they are even strongly stationary in the 
MPCC-terminology). Moreover, it has a local maximizer at $ (0,0) $
which is a point of attraction for many MPCC solvers since it 
can be shown to be C-stationary, see e.g.\ \cite{HoheiselKanzowSchwartz2013} for 
the corresponding definitions and some convergence results to
C- and M-stationary points.
	Due to \cref{lem:affine_data_implies_AM_regularity}, each
	feasible point of the problem is AM-regular.

In view of our convergence theory, \cref{Alg:ALM} should not 
converge to the origin. To verify this statement numerically, we 
generated $ 1000 $ random starting points (uniformly distributed) from 
the box $ [-10,10]^2 $ and then applied \cref{Alg:ALM}
to the above example. As expected, the method converges for all
$ 1000 $ starting points to one of the two minima. Moreover, we can
even start our method at the origin, and the method still converges
to the point $ (1,0) $ or $ (0,1) $. The limit point itself depends 
on our choice of the projection which is not unique for iterates
$ (y^k,z^k) $ with $ y^k = z^k > 0 $.
\end{example}

The next example is used to illustrate a limitation of our approach
which is based on the fact that we exploit the spectral gradient method
as a subproblem solver. There are examples where this spectral
gradient method reduces the number of iterations even for two-dimensional
problems from more than $ 100000 $ to just a few iterations.
Nevertheless, in the end, the spectral gradient method is a 
projected gradient method, which exploits a different stepsize
selection, but which eventually reduces to a standard projected
gradient method if there are a number of consecutive iterations 
with very small progress, i.e., with almost identical function
values during the last few iterations so that the maximum term in 
the nonmonotone line search is almost identical to the current
function value used in the monotone version. This situation typically happens for 
problems which are ill-conditioned, and we illustrate this observation
by the following example.

\begin{example}\label{Ex:Control}
We consider the optimal control of a discretized obstacle problem
as investigated in \cite[Section~7.4]{HarderMehlitzWachsmuth2021}. 
Using $w:=(x,y,z)$, in our notation, the problem is given by 
\begin{eqnarray*}
   \min_w & & f(w) := \tfrac{1}{2} \norm{ x }^2 - e^\top y + \tfrac{1}{2} \norm{ y }^2 \\
   \text{s.t.} & & x \geq 0, \
   -A y - x + z = 0, \
   y \geq 0, \ z \geq 0, \ y^\top z = 0.
\end{eqnarray*}
Here, $ A $ is a tridiagonal matrix which arises from a discretization
of the negative Laplace operator in one dimension, i.e., $ a_{ii} = 2 $ for all 
$ i $ and $ a_{ij}= -1 $ for all $ i = j \pm 1 $. 
Furthermore, $e$ denotes the all-one vector of appropriate size.
We note that $\bar w:=0$ is the global minimizer as well as an
M-stationary point of this program.
	Again, \cref{lem:affine_data_implies_AM_regularity} shows that
	each feasible point is AM-regular.
Viewing the constraint $ x \geq 0 $
as a box constraint, taking a moderate discretization with 
$ A \in \mathbb{R}^{64 \times 64} $,
and using the all-one vector as a starting point, we obtain
the results from \cref{Tab:Control}.
The number of (outer) iterations is denoted by $ k $, $ j $ is the 
number of inner iterations, $ j_{\textup{cum}} $ the accumulated number of inner
iterations, $ f $-ev.\ provides the number of function evaluations
(note that, due to the stepsize rule, we might have several function
evaluations in a single inner iteration, hence, $ f $-ev.\ is always
an upper bound for $ j_{\textup{cum}} $), $ f(w^k) $ denotes the current
function value,
the 
column titled ``$V_k$'' contains $V_{\rho_{k-1}}(w^k, u^{k-1})$,
$ t_j := 1/\gamma_j$
is the stepsize, and $ \rho_k $ denotes the penalty parameter at
iteration $ k $.

The method terminates after $ 12 $ outer iterations,
which is a reasonable number, especially taking into account that the final penalty
parameter $ \rho_k $ is relatively large, so that several subproblems with different
values of $ \rho_k $ have to be solved in the intermediate steps. On the other 
hand, the number of inner iterations $ j $ (at each outer iteration $ k $) is very large.
In the final step, the method requires more than one million inner iterations. This 
is a typical behavior of gradient-type methods and indicates that the underlying
subproblems are ill-conditioned. This is also reflected by the fact that the 
stepsize $ t_j $ tends to zero.
\end{example}

\begin{table}
\centering
\begin{tabular}{rrrrrrrr}
\hline
$k$ & $j$ & $j_{\textup{cum}}$ & $f$-ev. & \ccell{$f(w^k)$}
& \ccell{$V_k$} & \ccell{$t_j$} & $\rho_k$ \\ \hline
     0 &       0 &       0 &       1 &     32.0000000 &         --- &          --- &        320 \\
   1 &    4889 &    4889 &    8561 &    -30.2322093 &    0.017885 &   0.00019214 &        320 \\
   2 &    2765 &    7654 &   13171 &    -29.5693079 &    0.010772 &   0.00019553 &        320 \\
   3 &    2959 &   10613 &   18148 &    -29.1713687 &    0.008367 &   0.00019264 &        320 \\
   4 &    2734 &   13347 &   23001 &    -28.8787629 &    0.007077 &   0.00020241 &       3200 \\
   5 &   16380 &   29727 &   51233 &    -27.6160751 &    0.003845 &   0.00001961 &       3200 \\
   6 &   16412 &   46139 &   80229 &    -26.8702076 &    0.002675 &   0.00001967 &       3200 \\
   7 &   17708 &   63847 &  111596 &    -26.4929700 &    0.002437 &   0.00003231 &      32000 \\
   8 &  128146 &  191993 &  333580 &    -25.3129057 &    0.002357 &   0.00000196 &     320000 \\
   9 &  596930 &  788923 & 1364773 &    -13.1312431 &    0.000868 &   0.00000021 &     320000 \\
  10 &  756029 & 1544952 & 2686144 &     -5.3024263 &    0.000316 &   0.00000020 &     320000 \\
  11 &  911019 & 2455971 & 4320526 &     -2.0002217 &    0.000115 &   0.00000020 &     320000 \\
  12 & 1084340 & 3540311 & 6367887 &     -0.7376656 &    0.000042 &   0.00000020 &     320000 \\

  \hline
\end{tabular}
\caption{Numerical results for \cref{Ex:Control}.}\label{Tab:Control}
\end{table}

There are two types of difficulties in \cref{Ex:Control}: there 
are challenging constraints (the complementarity constraints), and there is an ill-conditioning.
The difficult constraints are treated by \cref{Alg:ALM} successfully, but 
the ill-conditioning causes some problems when solving the resulting subproblems.
In principle, this difficulty can be circumvented by using another subproblem solver
(like a semismooth Newton method,
see \cite{HarderMehlitzWachsmuth2021}),
but then it is no longer guaranteed that we
obtain M-stationary points at the limit.

Despite the fact that the ill-conditioning causes some difficulties, we stress
again that each iteration of the spectral gradient method is extremely cheap.
Moreover, for all test problems in the subsequent sections, we put an upper
bound of $ 50000 $ inner iterations (as a safeguard), and this upper bound
was not reached in any of these examples.

\subsection{Cardinality-Constrained Problems}\label{Sub:Cardinality}

We first consider an artificial example to illustrate the 
convergence behavior of \cref{Alg:ALM} for cardinality-constrained
problems.

\begin{example}\label{Ex:Cardinality}
Consider the example
\begin{equation*}
   \min_w \ f(w):=\tfrac{1}{2} w^\top Q w + c^\top w \quad \text{s.t.} \quad e^\top w\leq 8,
   \ \norm{ w }_0 \leq 2,
\end{equation*}
where $ Q := E + I $ with $ E \in \mathbb{R}^{5 \times 5} $ being the all
one matrix, $ I \in \mathbb{R}^{5 \times 5} $ the identity matrix, and
$ c := - ( 3, 2, 3, 12, 5 )^\top \in \mathbb{R}^5 $.
	Clearly, by \cref{lem:affine_data_implies_AM_regularity}, all
	feasible points are AM-regular.
This is a minor modification of an example from \cite{BeckEldar2013}, to 
which we added an (inactive) inequality constraint for the same reason
as in \cref{Ex:Scholtes}. Taking into account that there are $ \binom{5}{2} $ possibilities
to choose two possibly nonzero components of $ w $, an elementary calculation
shows that there are exactly $ 10 $ M-stationary points $\bar w^1, \ldots, \bar w^{10} $ 
which are given in \cref{Tab:Cardinality} together with the corresponding
function values. It follows that $\bar w^6 $ is the global minimizer. The points $ \bar w^3, \bar w^8$, and $ \bar w^{10} $
have function values which are not too far away from $ f(\bar w^6) $, whereas all
other M-stationary points have significantly larger function values. We then 
took $ 1000 $ random starting points from the box $ [-10,10]^5 $ (uniformly
distributed) and applied \cref{Alg:ALM} to this example. Surprisingly,
the method converged, for all $ 1000 $ starting points, to the global minimizer
$ \bar w^6 $. We then changed the example by putting an upper bound $ w_4\leq 0 $
to the fourth component. This excludes the four most interesting points
$ \bar w^3, \bar w^6, \bar w^8, $ and $ \bar w^{10} $. Among the remaining points, the three
vectors $ \bar w^4, \bar w^7, $ and $ \bar w^9 $ have identical function values. Running 
our program again using $ 1000 $ randomly generated starting points, we obtain
convergence to $ \bar w^4 $ in $ 589 $ cases, convergence to $ \bar w^7 $ in $ 350 $
situations, whereas in $ 61 $ instances only we observe convergence to the 
non-optimal point $ \bar w^2 $.
\end{example}

\begin{table}[htp]
\centering
\begin{tabular}{lr|lr}
\hline
$\bar w^i$&$f(\bar w^i)$\hspace*{3mm}&\hspace*{4mm}$\bar w^i$&$f(\bar w^i)$\\
\hline \\[-4mm]
$ \bar w^1    := \parens[\big]{ 4/3 , 1/3  , 0   , 0    , 0    }^\top $ & $ - 2.33 $ \hspace*{3mm} & \hspace*{3mm}
$ \bar w^6    := \parens[\big]{ 0   , -8/3 , 0   , 22/3 , 0    }^\top $ & $ -41.33 $\\
$ \bar w^2    := \parens[\big]{ 1   , 0    , 1   , 0    , 0    }^\top $ & $ - 3.00 $ \hspace*{3mm} & \hspace*{3mm}
$ \bar w^7    := \parens[\big]{ 0   , -1/3 , 0   , 0    , 8/3  }^\top $ & $ -6.33 $\\
$ \bar w^3    := \parens[\big]{ -2  , 0    , 0   , 7    , 0    }^\top $ & $ -39.00 $ \hspace*{3mm} & \hspace*{3mm}
$ \bar w^8    := \parens[\big]{ 0   , 0    , -2  , 7    , 0    }^\top $ & $ -39.00 $ \\
$ \bar w^4    := \parens[\big]{ 1/3 , 0    , 0   , 0    , 7/3  }^\top $ & $ -6.33 $ \hspace*{3mm} & \hspace*{3mm}
$ \bar w^9    := \parens[\big]{ 0   , 0    , 1/3 , 0    , 7/3  }^\top $ & $ -6.33 $\\
$ \bar w^5    := \parens[\big]{ 0   , 1/3  , 4/3 , 0    , 0    }^\top $ & $ -2.33 $ \hspace*{3mm} & \hspace*{3mm}
$ \bar w^{10} := \parens[\big]{ 0   , 0    , 0   , 19/3 , -2/3 }^\top $ & $ -36.33 $\\ \hline
\end{tabular}
\caption{M-stationary points and corresponding function values for 
\cref{Ex:Cardinality}.}\label{Tab:Cardinality}
\end{table}

We next consider a class of cardinality-constrained problems of the form
\begin{equation}\label{eq:portfolio-opt}
   \min_w \ \tfrac{1}{2} w^\top Q w \ \text{ s.t. } \ 
   \mu^\top w \geq \varrho, \
   e^\top w = 1, \
   0 \leq w \leq u, \
   \norm{ w }_0 \leq \kappa .
\end{equation}
This is a classical portfolio optimization problem, where $ Q $ and  $ \mu $
denote the covariance matrix and the mean of $ n $ possible assets, respectively, while $\varrho$ is 
some lower bound for the expected return. 
Furthermore, $u$ provides an upper bound for the individual assets within the portfolio.
	The affine structure of the constraints in \eqref{eq:portfolio-opt}
	implies that all feasible points are AM-regular,
	see \cref{lem:affine_data_implies_AM_regularity}.
The data $ Q, \mu, \varrho, u $ were randomly
created by the test problem collection \cite{FrangioniGentile2007},
which is available from the webpage
\url{https://commalab.di.unipi.it/datasets/MV/}.
Here, we used all 30 test instances of dimension $ n := 200 $ and three different
values $ \kappa \in \{ 5, 10, 20 \} $ for each problem. We apply three different
methods:
\begin{enumerate}[label=(\alph*)]
\item\label{item:alm} \cref{Alg:ALM} with starting point $ w^0 := 0 $, 
\item\label{item:almboost} a boosted version of  \cref{Alg:ALM}, and
\item\label{item:cplex} a CPLEX solver \cite{CPLEX} to a reformulation of the 
	portfolio optimization problem as a mixed integer quadratic program.
\end{enumerate}
The CPLEX solver is used to (hopefully) identify the global optimum of the optimization problem
\eqref{eq:portfolio-opt}.
Note that we put a time limit of $ 0.5 $ hours for each test problem. Method \ref{item:alm} applies
our augmented Lagrangian method to \eqref{eq:portfolio-opt} using the set 
$ D := \set{ w \in [0,u]  \given \norm{ w }_0 \leq \kappa } $. 
Projections onto $ D $ are computed using the analytic formula from 
\cref{Prop:ProjCsBox2}. Finally, the boosted
version of \cref{Alg:ALM} is the following: We first delete the cardinality constraint
from the portfolio optimization problem. The resulting quadratic program is then convex and 
can therefore be solved easily.
Afterwards, we apply \cref{Alg:ALM} to a sequence of relaxations
of \eqref{eq:portfolio-opt} in which the cardinality
is recursively decreased by $10$ in each step (starting with $n - 10$)
as long as the desired value $\kappa \in \set{5, 10, 20}$ is not undercut.
For $\kappa=5$, a final call of \cref{Alg:ALM} with the correct cardinality is necessary
since, otherwise, the procedure would terminate with cardinality level $10$.
In each outer iteration,
the projection of the solution of the previous iteration onto the set $D$
is used as a starting point.

\begin{figure}[htp]
\centering
\pgfplotstableread[col sep=tab]{portfolio_alm.txt}\almtable
\pgfplotstableread[col sep=tab]{portfolio_alm_boost.txt}\almboosttable
\begin{tikzpicture}
  \pgfplotstableread[col sep=tab]{portfolio_20.txt}\datatable
  \begin{axis}[
      width=1.0\textwidth,
      height=6cm,
      ybar=0cm,
      bar width=.1cm,
      enlarge x limits={abs=.2cm},
      ylabel={objective function value},
      xtick=data,
      xticklabels from table={\datatable}{Name of problem},
      x tick label style={ rotate=45,scale=.7,anchor=east }
    ]
    \addplot [fill=red!90!black] table [x expr=\coordindex, y={kappa 20}] {\almtable};
    \addplot [fill=yellow!90!black] table [x expr=\coordindex, y={kappa 20}] {\almboosttable};
    \addplot [fill=blue!90!black] table [x expr=\coordindex, y={objective function CPLEX (0.5 h)}] {\datatable};
  \end{axis}
\end{tikzpicture}\\
\begin{tikzpicture}
  \pgfplotstableread[col sep=tab]{portfolio_10.txt}\datatable
  \begin{axis}[
      width=1.0\textwidth,
      height=6cm,
      ybar=0cm,
      bar width=.1cm,
      enlarge x limits={abs=.2cm},
      ylabel={objective function value},
      xtick=data,
      xticklabels from table={\datatable}{Name of problem},
      x tick label style={ rotate=45,scale=.7,anchor=east }
    ]
    \addplot [fill=red!90!black] table [x expr=\coordindex, y={kappa 10}] {\almtable};
    \addplot [fill=yellow!90!black] table [x expr=\coordindex, y={kappa 10}] {\almboosttable};
    \addplot [fill=blue!90!black] table [x expr=\coordindex, y={objective function CPLEX (0.5 h)}] {\datatable};
  \end{axis}
\end{tikzpicture}\\
\begin{tikzpicture}
  \pgfplotstableread[col sep=tab]{portfolio_05.txt}\datatable
  \begin{axis}[
      width=1.0\textwidth,
      height=6cm,
      ybar=0cm,
      bar width=.1cm,
      enlarge x limits={abs=.2cm},
      ylabel={objective function value},
      xtick=data,
      xticklabels from table={\datatable}{Name of problem},
      x tick label style={ rotate=45,scale=.7,anchor=east }
    ]
    \addplot [fill=red!90!black] table [x expr=\coordindex, y={kappa 5}] {\almtable};
    \addplot [fill=yellow!90!black] table [x expr=\coordindex, y={kappa 5}] {\almboosttable};
    \addplot [fill=blue!90!black] table [x expr=\coordindex, y={objective function CPLEX (0.5 h)}] {\datatable};
  \end{axis}
\end{tikzpicture}%
\caption{Optimal function values obtained by
	\cref{Alg:ALM} (red),
	\cref{Alg:ALM} with boosting technique (yellow),
	and
	CPLEX (blue),
applied to the portfolio
optimization problem \eqref{eq:portfolio-opt} with cardinality 
$\kappa = 20$, $\kappa=10$, and $ \kappa = 5 $ (top to bottom).}
\label{Fig:SparseKappa}
\end{figure}
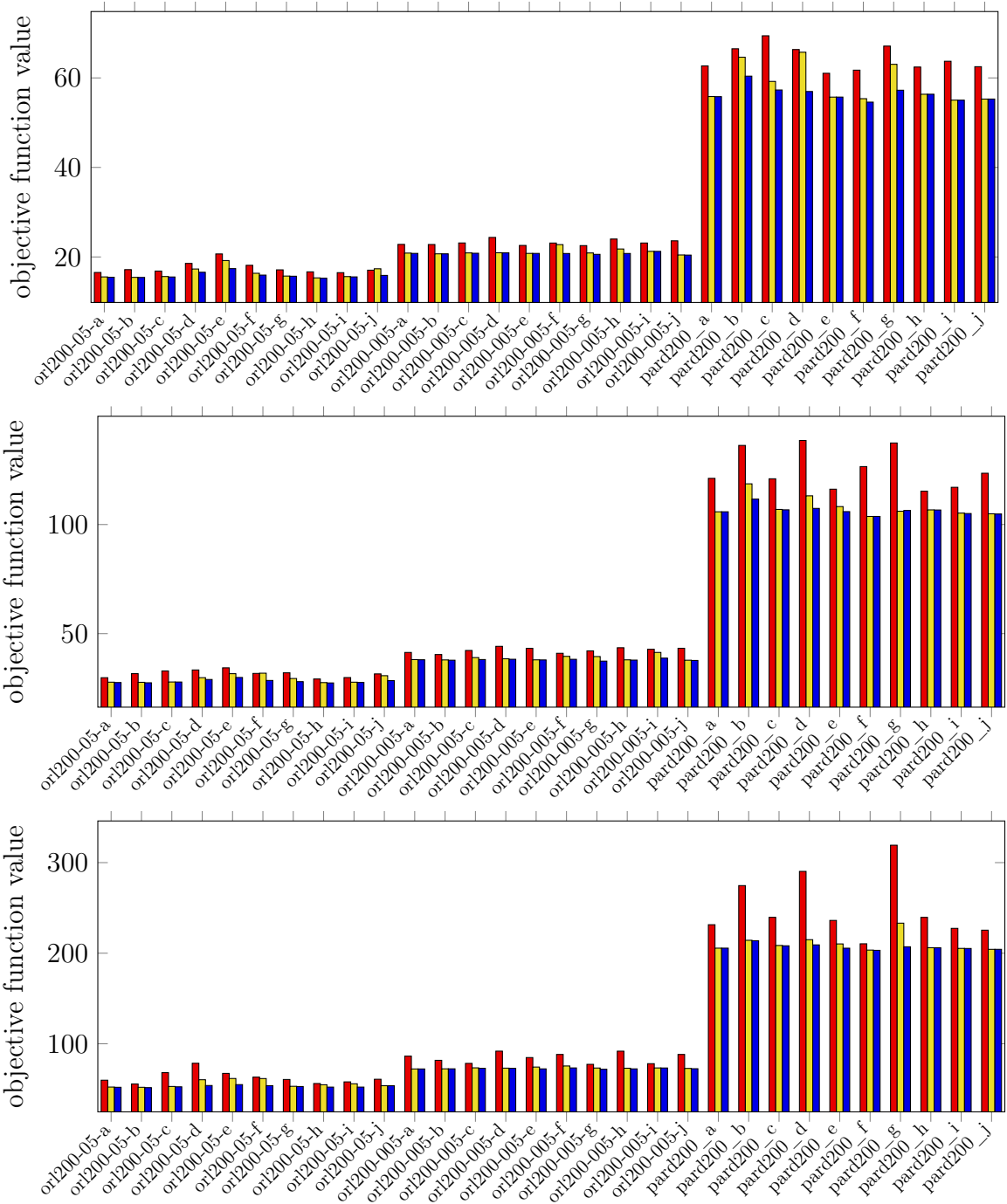

The corresponding results are summarized in 
\cref{Fig:SparseKappa} 
for the three different values $ \kappa \in\{5,10,20\}$. 
This figure compares the optimal function values obtained by the above
three methods for each of the 30 test problems. The optimal function values produced
by CPLEX are used here as a reference value in order to judge the quality of the results
obtained by the other approaches. The main observations are the following:
The optimal function value computed by CPLEX is (not surprisingly) always the best one.
On the other hand, the corresponding values computed by method \ref{item:alm} are usually not too far
away from the optimal ones. Moreover, for all test problems, the boosted version 
\ref{item:almboost} generates even better
function values which are usually very close to the ones computed by CPLEX.
Of course, if $ \kappa $ is taken smaller, the problems are getting more demanding
and are therefore more difficult to solve (in general).
Nevertheless,
also for $ \kappa = 5 $, especially the boosted algorithm still computes rather good
points. In this context, one should also note that our methods always terminate
with a (numerically) feasible point, hence, the final iterate computed by our method can
actually be used as a (good) approximation of the global minimizer.
We also would like to mention that our MATLAB implementation of \cref{Alg:ALM} typically
requires, on an Intel Core i7-8700 processor,
only a CPU time of about 0.1 seconds
for each of the test problems,
whereas the boosted version requires roughly two seconds CPU time in average.

\subsection{MAXCUT Problems}\label{Sub:MAXCUT}

This section considers the famous MAXCUT problem as an application of 
our algorithm to problems with rank constraints. To this end, let
$ G = (V, E) $ be an undirected graph with vertex set $ V = \{ 1, \ldots, n \} $
and edges $ e_{ij} $ between vertices $ i , j \in V $. We assume that we 
have a weighted graph, with $ a_{ij} = a_{ji} $ denoting the nonnegative weights of
the edge $ e_{ij} $. Since we allow zero weights, we can assume without
loss of generality that $ G $ is a complete graph. Now, given a subset
$ S \subset V $ with complement $ S^c $, the \emph{cut} defined by $ S $
is the set $ \delta (S) := \set{ e_{ij}  \given i \in S, j \in S^c } $
of all edges such that one end point belongs to $ S $ and the other one 
to $ S^c $. The corresponding weight of this cut is defined by
\begin{equation*}
   w(S) := \sum_{e_{ij} \in \delta (S)} a_{ij} .
\end{equation*}
The MAXCUT problem looks for the maximum cut, i.e., a cut with maximum
weight. This graph-theoretical problem is known to be NP-hard,
thus very difficult to solve.

Let $ A := ( a_{ij} ) $ and define $ L:= \diag(Ae) -A $. Then
it is well known, see e.g.\ \cite{GoemansWilliamson1995}, that the MAXCUT problem can be reformulated as 
\begin{equation}\label{eq:MAXCUT}
   \max_W \ \tfrac{1}{4} \trace ( L W ) \quad \text{s.t.} \quad
   \diag W = e, \ W \succeq 0, \ \rank W = 1,
\end{equation}
where the variable $W$ is chosen from the space
$\mathbb W:=\mathbb{R}^{n \times n}_{\textup{sym}}$.
Due to the linear constraint $ \diag W = e $,
it follows that this problem is equivalent to
\begin{equation}\label{eq:MAXCUT2}
   \max_W \ \tfrac{1}{4} \trace ( L W ) \quad \text{s.t.} \quad
   \diag W = e, \ W \succeq 0, \ \rank W \leq 1.
\end{equation}
Deleting the difficult rank constraint, one gets the
(convex) relaxation
\begin{equation}\label{eq:MAXCUTRelax}
   \max_W \ \tfrac{1}{4} \trace ( L W ) \quad \text{s.t.} \quad
   \diag W = e, \ W \succeq 0,
\end{equation}
which is a famous test problem for semidefinite programs.

Here, we directly deal with \eqref{eq:MAXCUT2} by taking
$ D: = \set{ W\in\mathbb W  \given W \succeq 0, \ \rank W \leq 1 } $ as the
complicated set. 
	Then GMFCQ holds at all feasible matrices of \eqref{eq:MAXCUT2}, 
	see \cref{sec:Appendix_MAXCUT_reg}.
	Particularly, AM-regularity is valid at all feasible points of
	\eqref{eq:MAXCUT2}.
Projections onto $ D $ can be calculated
via \cref{prop:SymProjLowRank}:
Let $ W \in\mathbb W$ denote an arbitrary symmetric matrix 
with maximum eigenvalue $ \lambda $ and corresponding (normalized)
eigenvector $ v $ (note that $ \lambda$ and $ v $ are not necessarily unique), 
then $ \max ( \lambda, 0 ) vv^\top $ is a projection of $ W $ onto 
$ D $.
In particular, the computation of this projection does not require the
full spectral decomposition.
However,
it is not clear
whether a projection onto the feasible set of \eqref{eq:MAXCUT2} can be computed efficiently.
Consequently,
we penalize the linear constraint $\diag W=e$ by the augmented Lagrangian
approach.

Throughout this section, we take the zero matrix as the starting point.
In order to illustrate the performance of our method, we begin with 
the simple graph from \cref{Fig:Graph}. \Cref{Alg:ALM}
applied to this example using the reformulation \eqref{eq:MAXCUT2}
(more precisely, the corresponding minimization problem) together with
the previous specifications yields the iterations shown in 
\cref{Tab:GraphResults}.
The meaning of the columns is the same as for \cref{Tab:Control}.

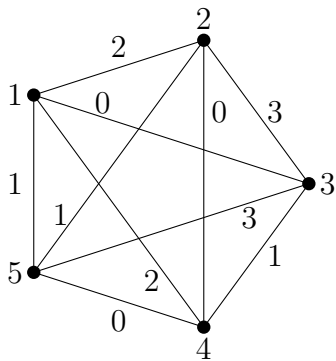
\begin{figure}[ht]
\centering
\begin{tikzpicture}[scale=1.0, transform shape] 
    \coordinate (1) at (144:2);
    \coordinate (2) at (72:2);
    \coordinate (3) at (0:2);
    \coordinate (4) at (288:2);
    \coordinate (5) at (216:2);

   \draw [-] (1) -- (2)
    node[midway, above] {$ 2 $}; 
    \draw [-] (2) -- (3)
    node[midway, right] {$ 3 $}; 
    \draw [-] (3) -- (4)
    node[midway, right] {$ 1 $}; 
    \draw [-] (4) -- (5)
    node[midway, below] {$ 0 $}; 
    \draw [-] (5) -- (1)
    node[midway, left] {$ 1 $}; 
    
    \draw [-] (2) -- (4)
    node[pos=.25, right=-.1em] {$ 0 $}; 
    \draw [-] (1) -- (3)
    node[pos=.25, above=-.2em] {$ 0 $}; 
    \draw [-] (2) -- (5)
    node[pos=.75, left=-.1em] {$ 1 $}; 
    \draw [-] (1) -- (4)
    node[pos=.75, below left=-.3em] {$ 2 $}; 
    \draw [-] (3) -- (5)
    node[pos=.25, below right=-.3em] {$ 3 $};

    \fill[black] (1) circle (2.5pt)
    node[left] {$1$};
    \fill[black] (2) circle (2.5pt)
    node[above] {$2$};
    \fill[black] (3) circle (2.5pt) 
    node[right] {$3$};
    \fill[black] (4) circle (2.5pt)
    node[below] {$4$};
    \fill[black] (5) circle (2.5pt)
    node[left] {$5$};
        
\end{tikzpicture}
\caption{Example of a complete graph for the MAXCUT problem.}\label{Fig:Graph}
\end{figure}

\begin{table}
\centering
\begin{tabular}{rrrrrrrr}
\hline
$k$ & $j$ & $j_{\textup{cum}}$ & $f$-ev. & \ccell{$f(W^k)$}
& \ccell{$V_k$} & \ccell{$t_j$} & $\rho_j$ \\ \hline
   0 &       0 &       0 &       1 &      0.0000000 &         --- &          --- &          4 \\
   1 &      11 &      11 &      16 &     19.6691638 &    0.839210 &   1.25237254 &          4 \\
   2 &       9 &      20 &      27 &     12.0395829 &    0.027365 &   0.63395340 &          4 \\
   3 &       5 &      25 &      34 &     12.0097591 &    0.006361 &   1.25001275 &          4 \\
   4 &       3 &      28 &      38 &     12.0023821 &    0.001553 &   0.62522386 &          4 \\
   5 &       3 &      31 &      42 &     12.0005415 &    0.000382 &   0.62504390 &          4 \\
   6 &       3 &      34 &      46 &     12.0001534 &    0.000097 &   0.62502107 &          4 \\

\hline
\end{tabular}
\caption{Numerical results for MAXCUT associated to the graph
from \cref{Fig:Graph}.}\label{Tab:GraphResults}
\end{table}

Note that the penalty parameter stays constant for
this example. The feasibility measure tends to zero,
and we terminate at iteration $ k = 6 $ since this measure becomes
less than $ 10^{-4} $, i.e., we stop successfully. The associated  
function value is (approximately) $ 12 $ which actually corresponds to the 
maximum cut $ S := \{ 1, 3 \} $ for the graph from \cref{Fig:Graph}, i.e., our method 
is able to solve the MAXCUT problem for this particular instance.
   
We next apply our method to two 
test problem collections that can be downloaded from 
\url{http://biqmac.aau.at/biqmaclib.html},
namely the \texttt{rudy} and the \texttt{ising} collection. The first 
class of problems consists of $ 130 $ instances, whereas the second one includes
$ 48 $ problems. The optimal function value $ f_{\textup{opt}} $ of all these examples is 
known. The details of the corresponding results obtained by our method
are given in
\ifpreprint
  \cref{sec:Appendix_B}. Let us summarize the main observations.
\else
  \cite{JiaKanzowMehlitzWachsmuth2021}. Here, we summarize the main observations.
\fi

All $ 130 + 48 $ test problems were solved successfully by our method
since the standard termination criterion was satisfied after finitely
many iterations, i.e., we stop with an iterate $ W^k $ which is
feasible (within the given tolerance). Hence, the corresponding optimal function
value $ f_{\textup{ALM}} $ is a lower bound for the optimal value $ f_{\textup{opt}} $.
For the sake of completeness, we also solved the (convex) relaxed problem
from \eqref{eq:MAXCUTRelax}, using again our augmented Lagrangian
method with $ D := \set{ W\in\mathbb W \given W \succeq 0 } $.
The corresponding function value is denoted by
$ f_{\textup{SDP}} $. Since the feasible set of \eqref{eq:MAXCUTRelax} is larger
than the one of \eqref{eq:MAXCUT2}, we have the inequalities
$ f_{\textup{ALM}} \leq f_{\textup{opt}} \leq f_{\textup{SDP}} $.
The corresponding details for the solution of the SDP-relaxation are
provided in
\ifpreprint
  \cref{sec:Appendix_B}
\else
  \cite{JiaKanzowMehlitzWachsmuth2021}
\fi
for the
\texttt{rudy} collection.

The bar charts from \cref{Fig:Rudy,Fig:Ising} summarize the results for the 
\texttt{rudy} and \texttt{ising} collections, respectively, in a very condensed way. 
They basically show that the function value $ f_{\textup{ALM}} $ 
obtained by our method is very close to the optimal value $ f_{\textup{opt}} $.
More precisely, the interpretation is as follows: For each test problem,
we take the quotient $ f_{\textup{ALM}} / f_{\textup{opt}} \in [0,1] $. If this quotient is
equal to, say, $ 0.91 $, we count this example as one where we reach 
$ 91 \% $ of the optimal function value. \Cref{Fig:Rudy} then
says that all $ 130 $ test problems were solved with at least $ 88 \% $
of the optimal function value. There are still $ 106 $ test examples
which are solved with a precision of at least $ 95 \% $. Almost one third 
of the test examples, namely $ 43 $ problems, are even solved with 
an accuracy of at least $ 99 \% $. For two examples (pm1d\_80.9,
and pw01\_100.8), we actually get the exact global maximum.

\Cref{Fig:Ising} has a similar meaning for the \texttt{ising} collection:
Though there is no example which is solved exactly, almost one half of the
problems reaches an accuracy of at least $ 99 \% $, and even in the worst
case, we obtain a precision of $ 94 \% $.

\begin{figure}[ht]
\centering
\begin{tikzpicture}
    % MAXCUT-Ergebnisse Rudy
  \pgfplotstableread{rudy.txt}\datatable

  \begin{axis}[
      width=.9\textwidth,
      height=6cm,
      xmin =  87.5,
      xmax = 100.5,
      ybar,
      ymin = 0,
      ymax = 150,
      ylabel={Number of test problems},
      xlabel={Percentage of optimal value function},
      nodes near coords,
    ]
    \addplot [fill=blue!80!black] table [x=Perc, y=Number] {\datatable};
  \end{axis}
\end{tikzpicture}%
\caption{Summary of the results from the \texttt{rudy} collection.}\label{Fig:Rudy}
\end{figure}
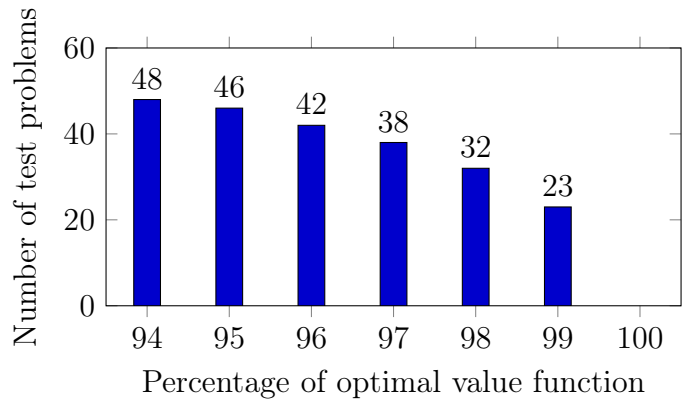

\begin{figure}[ht]
\centering
\begin{tikzpicture}
  % Read the data into a table macro
  \pgfplotstableread{ising.txt}\datatable

  \begin{axis}[
      width=.6\textwidth,
      height=5cm,
      ybar,
      xmin =  93.5,
      xmax = 100.5,
      ymin = 0,
      ymax = 60,
      ylabel={Number of test problems},
      xlabel={Percentage of optimal value function},
      nodes near coords,
    ]
    \addplot [fill=blue!80!black] table [x=Perc, y=Number] {\datatable};
  \end{axis}
\end{tikzpicture}%
\caption{Summary of the results from the \texttt{ising} collection.}\label{Fig:Ising}
\end{figure}
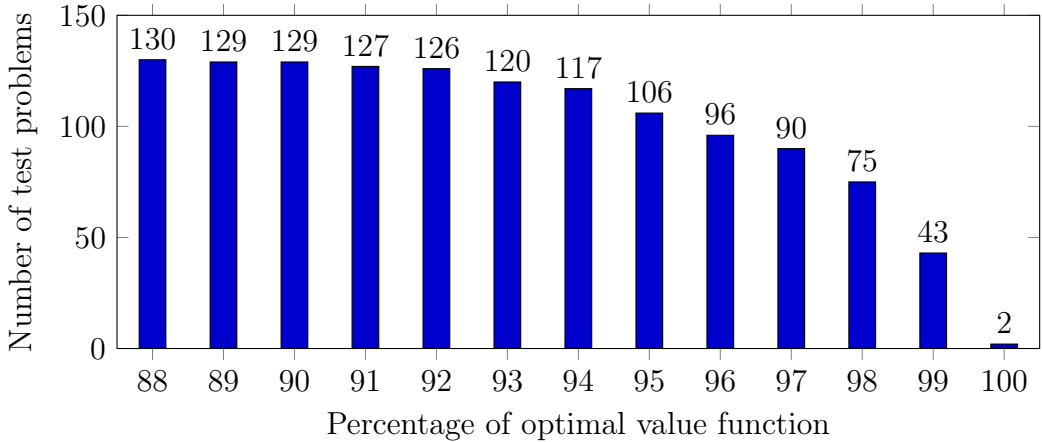

Altogether, this shows that we obtain a very good lower bound for the 
optimal function value. Moreover, since we are always feasible
(in particular, all iterates are matrices of rank one), the final matrix
can be used to create a cut through the given graph, i.e.,
the method provides a constructive way to create cuts which seem to 
be close to the optimal cuts. Note that this is in contrast to the
semidefinite relaxation \eqref{eq:MAXCUTRelax} which gives an upper
bound, but the solution associated with this upper bound is usually not feasible
for the MAXCUT problem since the rank constraint is violated (the
results in
\ifpreprint
  \cref{sec:Appendix_B}
\else
  \cite{JiaKanzowMehlitzWachsmuth2021}
\fi
show that the solutions of the relaxed programs
for the \texttt{rudy} collection are
matrices of rank between $ 4 $ and $ 7 $). In particular, these
matrices can, in general, not be used to compute a cut for the graph
and, therefore, are less constructive than the outputs of our method. Moreover, it 
is interesting to observe that $ f_{\textup{ALM}} $ is usually
much closer to $ f_{\textup{opt}} $ than $ f_{\textup{SDP}} $. In any case, 
both techniques together might be useful tools in a branch-and-bound-type
method for solving MAXCUT problems.

\section{Concluding Remarks}\label{Sec:Final}

In this paper, we demonstrated how M-stationary points of optimization problems 
with structured geometric constraints can be computed with the aid of an augmented Lagrangian
method. The fundamental idea was to keep the complicated constraints out of the augmented 
Lagrangian function and to treat them directly in the associated subproblems which are solved by means
of a nonmonotone projected gradient method. This way, the handling of challenging variational
structures is encapsulated within the efficient computation of projections.
This also puts a natural limit for the applicability.  
In contrast to several other approaches from the
literature, the convergence guarantees for our method, which are valid in the presence of a comparatively
weak asymptotic constraint qualification, remain true if the appearing subproblems
are solved inexactly. Extensive numerical experiments visualized the quantitative qualities
of this approach.

Despite our observations in \cref{Ex:Control}, it might be interesting to think about extensions of
these ideas to infinite-dimensional situations. In \cite{BoergensKanzowMehlitzWachsmuth2020}, an augmented
Lagrangian method for the numerical solution of \eqref{Eq:GenP} in the context of Banach spaces has
been considered where the set $D$ was assumed to be convex, and the subproblems in the resulting algorithm 
are of the same type as in our paper. Furthermore, convergence of the method to KKT points
was shown under validity 
of a problem-tailored version of asymptotic regularity. As soon as $D$ becomes nonconvex, one has
to face some uncomfortable properties of the appearing limiting normal cone which turns out to
be comparatively large since weak-$\ast$-convergence is used for its definition as a set limit in the
dual space, see \cite{HarderWachsmuth2018,MehlitzWachsmuth2018}. That it why the associated M-stationarity
conditions are, in general, too weak in order to yield a reasonable stationarity condition.
However, this issue might be surpassed by investigating the smaller strong limiting normal cone
which is based on strong convergence in the dual space but possesses very limited calculus. 
It remains open whether reasonable asymptotic regularity conditions w.r.t.\
this variational 
object can be formulated. Furthermore, in order to exploit the smallness of the strong limiting
normal cone in the resulting algorithm, one has to make sure (amongst others) that the (primal) 
sequence $\{w^j\}$ possesses strong accumulation points while the (dual) measures 
of inexactness $\{\varepsilon^j\}$ need to be strongly convergent as well. This might be restrictive.
Furthermore, it has to be clarified how the subproblems can be solved to approximate strong
M-stationarity.

%%%%%% Bibliography
%\bibliographystyle{plainnat}

\appendix
\section{Proofs}\label{sec:Appendix}

In this appendix, we provide the proofs which were left out in 
\cref{Sec:ProjGrad}.
\vspace{0.3cm}

\noindent
\begin{proof}[Proof of \cref{Prop:InnerLoop}]
	Recall that $w^{j,i}$ is a solution of \eqref{eq:sub_j_i}
	with $\gamma_{j,i} = \tau^{i-1} \gamma_j^0$.
	Since $ w^j \in D $, the optimality of $ w^{j,i}$ for 
\eqref{eq:sub_j_i} yields
\begin{equation}\label{Eq:CSI}
	\innerp{\nabla\varphi(w^j) }{ w^{j,i} - w^j } + \frac{\gamma_{j,i}}{2} \norm[\big]{ w^{j,i} - 
   w^j }^2 \leq 0
	 \qquad\forall i \in \N.
\end{equation}
The Cauchy--Schwarz inequality therefore gives 
\begin{equation*}
	\frac{\gamma_{j,i}}{2} \norm[\big]{ w^{j,i} - w^j } \leq \norm[\big]{ \nabla\varphi(w^j) } 
	\qquad \forall i \in\N.
\end{equation*}
This implies that $ w^{j,i} \to w^j $ for $ i \to \infty $.
Now, we distinguish two cases.
First, we consider that
\begin{equation}\label{Eq:Claim}
	\limsup_{i \to \infty} \gamma_{j,i} \norm[\big]{ w^{j,i} - w^j } > 0 .
\end{equation}
Hence, there exist a sequence $i_l \to \infty$ and a constant $ \rho > 0 $ such that
\begin{equation*}
	\gamma_{j,i_l} \norm[\big]{ w^{j,i_l} - w^j } \geq \rho \qquad \forall l\in\N
	.
\end{equation*}
Consequently, we obtain from \eqref{Eq:CSI} that 
\begin{equation*}
	\frac{\rho}{2} \norm[\big]{ w^{j,i_l} - w^j } \leq \frac{\gamma_{j,i_l}}{2}
	\norm[\big]{ w^{j,i_l} - w^j }^2 \leq - \innerp{\nabla\varphi(w^j)}{w^{j,i_l} - w^j } .
\end{equation*}
Together with a Taylor expansion, we therefore get
\begin{align*}
	\varphi \parens[\big]{ w^{j,i_l}} - \max_{\shift = 0, 1, \ldots, m_j}\varphi(w^{j-\shift})
	& \leq  \varphi \parens[\big]{ w^{j,i_l}} -\varphi(w^j) \\
	& = \innerp{\nabla\varphi(w^j)}{ w^{j,i_l} - w^j } + o \parens[\big]{ \norm[\big]{ w^{j,i_l} - w^j }
   } \\
	 & \leq \sigma \innerp{\nabla\varphi(w^j)}{ w^{j,i_l} - w^j }
\end{align*}
for all $l$ sufficiently large,
i.e., the inner loop terminates.

In the second case, \eqref{Eq:Claim} is not satisfied,
i.e.,
$\gamma_{j,i} \norm[\big]{ w^{j,i} - w^j } \to 0$.
By continuity of $\nabla\varphi$, this yields \eqref{eq:inner_not_termination}.
Together with $w^{j,i} \to w^j$ and by using the continuity of $\nabla\varphi$
as well as \eqref{Eq:stability}
we can pass to the limit $i \to \infty$ in \eqref{eq:M_St_surrogate_ji}
and obtain that $w^j$ is M-stationary.
\end{proof}

\vspace{0.3cm}
\noindent
\begin{proof}[Proof of \cref{Prop:Diffzero}]
Let $ l(j)\in\{j-m_j,\ldots,j\} $ be an index such that
\begin{equation*}
   \varphi(w^{l(j)}) = \max_{\shift= 0, 1, \ldots, m_j}\varphi(w^{j-\shift}) \qquad \forall j \in \mathbb{N}.
\end{equation*}
Then the nonmonotone Armijo rule from
\cref{step:nonmonotone_linesearch} in \cref{Alg:NPG}
can be rewritten as
\begin{equation}\label{Eq:NA1}
  \varphi(w^{j+1}) \leq \varphi(w^{l(j)}) + \sigma \innerp{\nabla\varphi(w^j)}{w^{j+1} - w^j }.
\end{equation}
Since $ w^{j+1} $ solves
\begin{equation}\label{Eq:Subproblem}
	\min_w \ \varphi(w^j) + \innerp{ \nabla \varphi (w^j) }{ w-w^j }  + \frac{\gamma_j}{2} \norm{ w - w^j }^2 \quad 
	\text{s.t.} \quad w \in D,
\end{equation}
we have
\begin{equation*}
   \innerp{\nabla\varphi(w^j) }{ w^{j+1} - w^j } + \frac{\gamma_j}{2} \norm{ w^{j+1} - w^j }^2 \leq 0,
\end{equation*}
i.e.,
\begin{equation*}
   \innerp{\nabla\varphi(w^j)}{ w^{j+1} - w^j } \leq - \frac{\gamma_j}{2} \norm{ w^{j+1} - w^j }^2 .
\end{equation*}
Hence, \eqref{Eq:NA1} implies
\begin{equation}\label{Eq:10}
  \varphi(w^{j+1}) \leq \varphi(w^{l(j)}) - \gamma_j \frac{\sigma}{2} \norm{ w^{j+1} - w^j }^2 .
\end{equation}
We first note that the sequence $ \{\varphi(w^{l(j)} ) \}_j $ is monotonically decreasing.
Using $ m_{j+1} \leq m_j + 1 $, this follows from
\begin{align*}
   \varphi ( w^{l(j+1)} ) & = \max_{\shift=0,1, \ldots, m_{j+1}}\varphi(w^{j+1-\shift}) \\
   & \leq \max_{\shift=0,1, \ldots, m_j +1}\varphi(w^{j+1-\shift}) \\
   & =  \max \parens*{ \max_{\shift = 0, 1, \ldots, m_j}\varphi(w^{j-\shift}),\varphi(w^{j+1}) } \\
   & =  \max \parens[\big]{\varphi(w^{l(j)}),\varphi(w^{j+1}) } \\
   & = \varphi(w^{l(j)}),
\end{align*}
where the last equality follows from \eqref{Eq:10}. Since $ \varphi $ is bounded from 
below, this implies 
\begin{equation}\label{Eq:Limitfk}
   \lim_{j \to \infty}\varphi( w^{l(j)} ) = \varphi^*
\end{equation}
for some finite $\varphi^*\in\R$.
Applying \eqref{Eq:10} with $ j $ replaced by $ l(j)-1 $ and rearranging terms yields
\begin{equation*}
   \varphi(w^{l(j)}) -\varphi(w^{l(l(j)-1)}) \leq - \gamma_{l(j)-1} \frac{\sigma}{2}
   \norm[\big]{ w^{l(j)} - w^{l(j)-1} }^2 \leq 0.
\end{equation*}
Taking the limit $ j \to \infty $ and using \eqref{Eq:Limitfk} therefore implies
\begin{equation*}
   \lim_{j \to \infty} \gamma_{l(j)-1} \norm[\big]{ w^{l(j)} - w^{l(j)-1} }^2 = 0 .
\end{equation*}
Since $ \gamma_j \geq \gamma_{\min} > 0 $ for all $ j \in \mathbb{N} $, we get
\begin{equation}\label{Eq:Ind1}
   \lim_{j \to \infty} d^{l(j)-1} = 0,
\end{equation}
where, for simplicity, we set $ d^j := w^{j+1} - w^j $ for all $ j \in
\mathbb{N} $. Using \eqref{Eq:Limitfk} and \eqref{Eq:Ind1}, we then obtain
\begin{equation}\label{Eq:Ind2}
   \varphi^* = \lim_{j \to \infty}\varphi(w^{l(j)}) =
   \lim_{j \to \infty} \varphi \parens[\big]{ w^{l(j)-1} + d^{l(j)-1} } =
   \lim_{j \to \infty} \varphi (w^{l(j)-1}),
\end{equation}
where the last equality takes into account the uniform continuity of $ \varphi $.
We will now prove, by induction, that 
\begin{equation}\label{Eq:Indj}
   \lim_{j \to \infty} d^{l(j)-\shift} = 0 \quad \text{and} \quad 
   \lim_{j \to \infty}\varphi(w^{l(j)-\shift}) = \varphi^* \quad
   \forall \shift \in\N.
\end{equation}
We already know from \eqref{Eq:Ind1} and \eqref{Eq:Ind2} that \eqref{Eq:Indj} holds for 
$ \shift = 1 $. Suppose that \eqref{Eq:Indj} holds for some $ \shift \geq 1 $. We need to show that
it holds for $ \shift+1 $. Using \eqref{Eq:10} with $ j $ replaced by
$ l(j)-\shift-1 $, we have
\begin{equation*}
  \varphi(w^{l(j)-\shift}) \leq\varphi(w^{l(l(j)-\shift-1)}) - \gamma_{l(j)-\shift-1} \frac{\sigma}{2}
   \norm[\big]{ d^{l(j)-\shift-1} }^2 
\end{equation*}
(here we assume implicitly that $ j $ is large enough such that no negative indices
$ l(j)-\shift-1 $ occur). Rearranging this expression and using $ \gamma_j \geq
\gamma_{\min} $ for all $ j $ yields
\begin{equation*}
   \norm[\big]{ d^{l(j)-\shift-1} }^2 \leq \frac{2}{\gamma_{\min} \sigma}
   \parens[\big]{\varphi(w^{l(l(j)-\shift-1)}) -\varphi(w^{l(j)-\shift}) } .
\end{equation*}
Taking the limit $ j \to \infty $ while using \eqref{Eq:Limitfk} as well as the induction
hypothesis, it follows that 
\begin{equation}\label{Eq:dstar}
   \lim_{j \to \infty} d^{l(j)-\shift-1} = 0,
\end{equation}
which proves the induction step for the first limit in \eqref{Eq:Indj}.
The second limit follows from 
\begin{equation*}
   \lim_{j \to \infty} \varphi\parens[\big]{ w^{l(j)-(\shift+1)} } =
   \lim_{j \to \infty} \varphi\parens[\big]{ w^{l(j)-(\shift+1)} + d^{l(j)-(\shift+1)} } =
   \lim_{j \to \infty} \varphi\parens[\big]{ w^{l(j)-\shift} } =
   \varphi^*,
\end{equation*}
where the first equation follows from \eqref{Eq:dstar} together with
the uniform continuity of $ \varphi $, whereas the final equation is the induction
hypothesis.

In the final step of our proof, we now show that $ \lim_{j \to \infty} d^j = 0 $. 
Suppose that this is not true. Then there is a (suitably shifted, for notational
simplicity) subsequence $ \{ d^{j-m-1} \}_K $
and a constant $ \rho > 0 $ such that
\begin{equation}\label{Eq:Contrad}
   \norm{ d^{j-m-1} } \geq \rho \qquad \forall j \in K.
\end{equation}
Now, for each $ j \in K $, the corresponding index $ l(j) $ is one of the indices
$ j - m, j - m + 1, \ldots, j $. Hence, we can write $ j - m - 1 = l(j) - \shift_j $
for some index $ \shift_j \in \{ 1, 2, \ldots, m+1 \} $. Since there are only finitely
many possible indices $ \shift_j $, we may assume without loss of generality that
$ \shift_j = \shift $ holds for some fixed index $ \shift $. Then \eqref{Eq:Indj} implies
\begin{equation*}
   \lim_{j \to_K \infty} d^{j-m-1} = \lim_{j \to_K \infty} d^{l(j) - \shift} = 0.
\end{equation*}
This contradicts \eqref{Eq:Contrad} and therefore completes the proof.
\end{proof}

\vspace{0.3cm}
\noindent
\begin{proof}[Proof of \cref{Thm:ConvNPG}]
Let $ \bar w $ be an arbitrary accumulation point, and let $ \{ w^j \}_K$ be a 
subsequence such that $w^j\to_K\bar w$. 

We start by showing $\gamma_j \, \parens*{w^{j+1} - w^j} \to_K 0$.
In the case that 
$ \{ \gamma_j \}_K $ is bounded, this follows from 
\cref{Prop:Diffzero}.
In the case that
$ \{ \gamma_j \}_K $ is unbounded, we find a subsequence $K' \subset K$
with $ \gamma_j \to_{K'}
\infty $ and $\gamma_j > \gamma_{\max}$ for all $j \in K'$.
Then $ \hat \gamma_j := \gamma_j / \tau = \tau^{i_j - 1} \gamma_j^0 = \gamma_{j,i_j - 1}$ also converges to infinity.
Due to $\gamma_j > \gamma_{\max}$,
we have $i_j > 0$.
Therefore,
$\hat w^{j+1} := w^{j,i_j - 1}$
(which solves \Qref{j,i_j-1})
violates the nonmonotone Armijo-type condition 
from
\cref{step:nonmonotone_linesearch} in \cref{Alg:NPG}, i.e., we have
\begin{equation}\label{Eq:NonArmijo}
	\varphi \parens[\big]{ \hat w^{j+1} } > \max_{\shift = 0, 1, \ldots, m_j}
	\varphi ( w^{j-\shift} ) + \sigma \innerp{\nabla\varphi(w^j)}{ \hat w^{j+1} - w^j }
\end{equation}
for all $ j \in K' $ sufficiently large. We now argue similar to the proof of \cref{Prop:InnerLoop}
(except that $ j $ is not fixed now). Since $ \hat w^{j+1} $ solves the subproblem
\Qref{j, i_j - 1}, we obtain
\begin{equation}\label{Eq:Optxk}
	\innerp{\nabla\varphi ( w^j ) }{\hat w^{j+1} - w^j } +
	\frac{\hat \gamma_j}{2} \norm[\big]{ \hat w^{j+1} - w^j }^2 \leq 0,
\end{equation}
which implies that
\begin{equation*}
	\frac{\hat \gamma_j}{2} \norm[\big]{ \hat w^{j+1} - w^j } \leq \norm[\big]{
   \nabla \varphi ( w^j ) }.
\end{equation*}
Since $ w^j \to_{K'} \bar w $, this yields $ \hat w^{j+1} - w^j \to_{K'} 0 $. Hence, we also
get $ \hat w^{j+1} \to_{K'} \bar w $.
For each $j\in {K'}$, the mean value theorem yields the existence of $\xi^j$ on the line segment
between $\hat w^{j+1}$ and $w^j$ such that
\begin{align*}
	\varphi \parens[\big]{ \hat w^{j+1} } -\varphi(w^j)
	=
	\innerp{\nabla\varphi(\xi^j)}{ \hat w^{j+1} - w^j }.
\end{align*}
Due to $ \hat w^{j+1}, w^j \to_{K'} \bar w $, we find $\nabla\varphi(\xi^j)-\nabla\varphi(w^j)\to_{K'} 0$.
Using \eqref{Eq:NonArmijo},
we get
\begin{align*}
	\sigma \innerp{\nabla\varphi(w^j) }{ \hat w^{j+1} - w^j }
	&<
	\varphi \parens[\big]{ \hat w^{j+1} }
	-
	\max_{\shift = 0, 1, \ldots, m_j} \varphi ( w^{j-\shift} )
	\\
	&\le
	\varphi \parens[\big]{ \hat w^{j+1} }
	-
	\varphi ( w^{j} )
	\\
	&\le
	\innerp{\nabla\varphi(w^j)}{ \hat w^{j+1} - w^j }
	+
	\norm{ \nabla\varphi(\xi^j) - \nabla\varphi(w^j) } \, \norm{ \hat w^{j+1} - w^j }
	.
\end{align*}
Together with \eqref{Eq:Optxk},
we achieve
\begin{equation*}
	\frac{\hat \gamma_j}{2} \norm[\big]{ \hat w^{j+1} - w^j }^2 \leq
	-\innerp{\nabla\varphi ( w^j ) }{\hat w^{j+1} - w^j }
	 \le
	 \frac{\norm{ \nabla\varphi(\xi^j) - \nabla\varphi(w^j) }}{1-\sigma}
	 \norm{ \hat w^{j+1} - w^j }
	.
\end{equation*}
Thus,
$
\hat \gamma_j \norm[\big]{ \hat w^{j+1} - w^j }
	\to_{K'}
	0$.
	Using the optimality of $\hat w^{j+1}$ and $w^{j+1}$ for
	\Qref{j, i_j - 1} and \Qref{j, i_j}, respectively,
	we find
	\begin{equation*}
		\gamma_j \norm{ w^{j+1} - w^j }
		=
		\tau \hat\gamma_j \norm{ w^{j+1} - w^j }
		\le
		\tau \hat\gamma_j \norm{ \hat w^{j+1} - w^j }
		\to_{K'}
		0.
	\end{equation*}
	Now, one can use a standard subsequence-subsequence argument
	to conclude that
	$ \gamma_j \norm{ w^{j+1} - w^j } \to_K 0$
	holds along the entire subsequence $K$.

It remains to verify M-stationarity of $\bar w$.
Since $ w^{j+1} $ solves the subproblem \eqref{Eq:Subproblem}, the corresponding optimality 
condition yields
\begin{equation*}
   0 \in \nabla\varphi(w^j) + \gamma_j \parens[\big]{ w^{j+1} - w^j } + \mathcal N_D^\textup{lim} (w^{j+1}).
\end{equation*}
Due to \cref{Prop:Diffzero}, we also have $ w^{j+1} \to_K \bar w $.
Hence, taking the limit $ j \to_K \infty $
and exploiting once again the upper semicontinuity of 
the limiting normal cone, we obtain 
\begin{equation*}
   0 \in \nabla\varphi(\bar w) + \mathcal N_D^\textup{lim} (\bar w),
\end{equation*}
i.e., $ \bar w $ is an M-stationary point of \eqref{Eq:P}.
\end{proof}

\section{Constraint Regularity of the MAXCUT Problem}\label{sec:Appendix_MAXCUT_reg}

We show that feasible points of \eqref{eq:MAXCUT2} with
$D:=\set{ W\in\R^{n\times n}_\textup{sym} \given W \succeq 0, \ \rank W \leq 1 }$
generally satisfy GMFCQ.
Note that we use $G\colon\R^{n\times n}_\textup{sym}\to\R^n$ given by
$G(W):=\diag W$, $W\in\R^{n\times n}_\textup{sym}$, and $C:=\{e\}$ here
in order to model the feasible set of \eqref{eq:MAXCUT2} in the form
given in \eqref{Eq:GenP}.

Let us fix a feasible matrix $W\in\R^{n\times n}_\textup{sym}$ of \eqref{eq:MAXCUT2}.
Then we find a vector $u\in\{\pm n^{-1/2}\}^n$ such that $W=nuu^\top$, i.e.,
$W$ possesses the non-zero eigenvalue $n$ and the associated eigenvector $u$.

First, we will show that
\begin{equation}\label{eq:MAXCUT_normal_cone_D}
	\mathcal N_D^\textup{lim}(W)
	\subset
	\{Y\in\R^{n\times n}_\textup{sym}\,|\,Yu=0\}.
\end{equation}
For $Y\in\mathcal N_D^\textup{lim}(W)$, 
we find sequences $\{W^k\},\{Y^k\}\subset\R^{n\times n}_\textup{sym}$ and 
$\{\alpha_k\}\subset[0,\infty)$ such that $W^k\to W$, $Y^k\to Y$, and
$Y^k\in\alpha_k(W^k-\Pi_D(W^k))$ for all $k\in\N$.
For each $k\in\N$, let $W^k = \sum_{i=1}^n\mu_i^ku_i^k(u_i^k)^\top$ be an
(orthonormal) eigenvalue decomposition with non-increasingly ordered
eigenvalues $\mu_1^k\geq\mu_2^k\geq\ldots\geq\mu_n^k$ and associated
pairwise orthonormal eigenvectors $u_1^k,\ldots,u_n^k$.
Due to $W^k\to W$, we find $\mu_1^k\to n$, $\mu_2^k,\ldots,\mu_n^k\to 0$,
and (along a subsequence without relabeling) $u_1^k\to\pm u$.
For sufficiently large $k\in\N$,
this implies $\Pi_D(W^k) = \{\mu_1^ku_1^k(u_1^k)^\top\}$,
see \cite[Proposition~3.4]{Luke2013} and \cref{prop:SymProjLowRank}.
Hence, for any such $k\in\N$, we find $Y^k=\alpha_k\sum_{i=2}^n\mu_i^ku_i^k(u_i^k)^\top$.
Particularly, this gives $Y^ku_1^k=0$ for large enough $k\in\N$,
so that $Yu=0$ follows by taking the limit $k\to\infty$.

Second, suppose that there are a vector $\lambda\in\mathcal N_C(G(W))=\R^n$ 
and a matrix $Y\in\mathcal N_D^\textup{lim}(W)$ such that $\diag\lambda+Y=0$.
In order to prove validity of GMFCQ, $\lambda=0$ has to be shown.
From \eqref{eq:MAXCUT_normal_cone_D}, we find $\lambda\bullet u=-Yu=0$ where $\bullet$ represents
the entrywise product operation. Observing that the components of $u$ 
are all different from zero, $\lambda=0$ follows.

\ifpreprint

\section{Complete Numerical Results for MAXCUT Problems}\label{sec:Appendix_B}

\Cref{Table:MAXCUT,Table:MAXCUT2} present the details of \cref{Alg:ALM} applied to the MAXCUT problem as
described in \cref{Sub:MAXCUT}. The columns have the following meanings: \\[1mm]
\begin{itemize}[leftmargin=4em]
	\item[probl.:] name of the test problem,
	\item[vert.:] number of vertices of the underlying graph, 
   	\item[edges:] number of edges with positive weights of the underlying graph, 
   	\item[$ k $:] number of (outer) iterations of \cref{Alg:ALM}, 
   	\item[$ j $:] number of inner iterations of the spectral gradient method (accumulated), 
    \item[$f$-ev.:] number of function evaluations (accumulated), 
   	\item[$ f_{\textup{ALM}} $:] function value at the final iterate generated by \cref{Alg:ALM}, 
   	\item[feas.:] feasibility measure at the final iterate generated by \cref{Alg:ALM}, 
   	\item[$ \rho $:] penalty parameter at the final iterate generated by \cref{Alg:ALM}, 
   	\item[$f_{\textup{opt}}$:] optimal function value taken from the 
   		report \cite{Wiegele2007}, 
    \item[$f_{\textup{SDP}}$:] function value obtained by the SDP relaxation, and 
  	\item[rk:] rank of the final matrix obtained by solving the SDP relaxation.
\end{itemize}

Recall that the SDP relaxation was also solved by \cref{Alg:ALM} with the semidefiniteness
constraint taken as the complicated constraint, whereas the remaining linear equality constraint
was penalized by the augmented Lagrangian approach.

\begin{scriptsize}
	%% Put the caption in \normalsize
	\def\tablename{\normalsize Table}

\begin{longtable}[c]{rrr|rrrrrr|r|rr}
\caption{Numerical results for MAXCUT problems, \texttt{rudy} collection.}\label{Table:MAXCUT}\\ \hline
     probl. & vert. & edges &  $ k $ & $ j $ & $f$-ev. &  $f_{\textup{ALM}}$\hspace{0.5mm} 
     & feas.\hspace{0.5mm} & $ \rho $ & $f_{\textup{opt}}$ & $f_{\textup{SDP}}$ & rk \\ \hline
\endfirsthead
\caption{Numerical results for MAXCUT problems, \texttt{rudy} collection (continued).} \\ \hline
      probl. & vert. & edges &  $ k $ & $ j $ & $f$-ev. &  $f_{\textup{ALM}}$\hspace{0.5mm} & feas. & $ \rho $ & $f_{\textup{opt}}$ & $f_{\textup{SDP}}$ & rk \\ \hline
\endhead
\hline
\endfoot
\hline
\endlastfoot
  g05\_100.0 &    100 &    2475 &   14 &   310 &   349 & 1420 & 2.8e-5 &   20.0 &     1430 & 1463.52 & 6 \\
  g05\_100.1 &    100 &    2475 &   32 &   396 &   448 & 1420 & 9.4e-5 &    2.0 &     1425 & 1464.05 & 6 \\
  g05\_100.2 &    100 &    2475 &   21 &   339 &   378 & 1431 & 9.9e-5 &    2.0 &     1432 & 1461.65 & 5 \\
  g05\_100.3 &    100 &    2475 &   28 &   500 &   572 & 1411 & 1.0e-4 &    2.0 &     1424 & 1456.68 & 7 \\
  g05\_100.4 &    100 &    2475 &   24 &   475 &   550 & 1430 & 9.7e-5 &    2.0 &     1440 & 1468.80 & 6 \\
  g05\_100.5 &    100 &    2475 &   36 &   409 &   507 & 1415 & 9.8e-5 &    2.0 &     1436 & 1464.66 & 6 \\
  g05\_100.6 &    100 &    2475 &   26 &   476 &   535 & 1429 & 8.8e-5 &    2.0 &     1434 & 1463.17 & 6 \\
  g05\_100.7 &    100 &    2475 &   12 &   292 &   328 & 1428 & 5.5e-5 &   20.0 &     1431 & 1464.27 & 5 \\
  g05\_100.8 &    100 &    2475 &   30 &   325 &   361 & 1425 & 8.3e-5 &    2.0 &     1432 & 1464.75 & 5 \\
  g05\_100.9 &    100 &    2475 &   35 &   344 &   391 & 1415 & 8.2e-5 &    2.0 &     1430 & 1462.39 & 5 \\
   g05\_60.0 &     60 &     885 &   20 &   286 &   322 & 530 & 7.6e-5 &    3.3 &      536 & 550.05 & 5 \\
   g05\_60.1 &     60 &     885 &   18 &   290 &   323 & 524 & 9.5e-5 &    3.3 &      532 & 543.11 & 5 \\
   g05\_60.2 &     60 &     885 &   19 &   283 &   313 & 524 & 6.3e-5 &    3.3 &      529 & 543.18 & 4 \\
   g05\_60.3 &     60 &     885 &   18 &   253 &   302 & 523 & 6.2e-5 &    3.3 &      538 & 548.65 & 4 \\
   g05\_60.4 &     60 &     885 &   18 &   326 &   413 & 526 & 9.6e-5 &    3.3 &      527 & 541.39 & 5 \\
   g05\_60.5 &     60 &     885 &   18 &   219 &   252 & 523 & 6.5e-5 &    3.3 &      533 & 542.59 & 6 \\
   g05\_60.6 &     60 &     885 &   18 &   278 &   311 & 520 & 9.7e-5 &    3.3 &      531 & 544.72 & 5 \\
   g05\_60.7 &     60 &     885 &   20 &   277 &   306 & 530 & 6.8e-5 &    3.3 &      535 & 550.42 & 5 \\
   g05\_60.8 &     60 &     885 &   16 &   338 &   381 & 520 & 6.1e-5 &    3.3 &      530 & 543.98 & 5 \\
   g05\_60.9 &     60 &     885 &   20 &   216 &   245 & 529 & 6.4e-5 &    3.3 &      533 & 549.89 & 5 \\
   g05\_80.0 &     80 &    1580 &   17 &   241 &   263 & 918 & 6.9e-5 &    2.5 &      929 & 950.92 & 5 \\
   g05\_80.1 &     80 &    1580 &   22 &   380 &   465 & 929 & 8.3e-5 &    2.5 &      941 & 957.25 & 4 \\
   g05\_80.2 &     80 &    1580 &   24 &   307 &   341 & 923 & 9.2e-5 &    2.5 &      934 & 955.55 & 5 \\
   g05\_80.3 &     80 &    1580 &   28 &   324 &   362 & 906 & 7.9e-5 &    2.5 &      923 & 947.59 & 5 \\
   g05\_80.4 &     80 &    1580 &   26 &   303 &   346 & 923 & 8.2e-5 &    2.5 &      932 & 955.32 & 5 \\
   g05\_80.5 &     80 &    1580 &   22 &   299 &   343 & 915 & 7.5e-5 &    2.5 &      926 & 947.51 & 6 \\
   g05\_80.6 &     80 &    1580 &   27 &   257 &   293 & 920 & 7.8e-5 &    2.5 &      929 & 948.68 & 5 \\
   g05\_80.7 &     80 &    1580 &   22 &   289 &   321 & 915 & 9.0e-5 &    2.5 &      929 & 949.86 & 5 \\
   g05\_80.8 &     80 &    1580 &   27 &   436 &   504 & 918 & 7.6e-5 &    2.5 &      925 & 946.67 & 5 \\
   g05\_80.9 &     80 &    1580 &   25 &   266 &   296 & 918 & 8.7e-5 &    2.5 &      923 & 943.66 & 6 \\
 pm1d\_100.0 &    100 &    4901 &   13 &   373 &   475 & 338 & 2.5e-5 &   20.0 &      340 & 405.39 & 6 \\
 pm1d\_100.1 &    100 &    4901 &   13 &   517 &   612 & 310 & 3.2e-5 &   20.0 &      324 & 396.09 & 6 \\
 pm1d\_100.2 &    100 &    4901 &   12 &   489 &   571 & 370 & 3.0e-5 &   20.0 &      389 & 453.98 & 7 \\
 pm1d\_100.3 &    100 &    4901 &   15 &   419 &   492 & 396 & 2.8e-5 &   20.0 &      400 & 459.03 & 5 \\
 pm1d\_100.4 &    100 &    4901 &   14 &   450 &   526 & 349 & 6.0e-5 &   20.0 &      363 & 430.32 & 6 \\
 pm1d\_100.5 &    100 &    4901 &   16 &   345 &   383 & 440 & 2.6e-5 &   20.0 &      441 & 510.72 & 5 \\
 pm1d\_100.6 &    100 &    4901 &   14 &   511 &   623 & 360 & 4.0e-5 &   20.0 &      367 & 431.92 & 5 \\
 pm1d\_100.7 &    100 &    4901 &   12 &   246 &   278 & 348 & 3.8e-5 &   20.0 &      361 & 421.52 & 5 \\
 pm1d\_100.8 &    100 &    4901 &   14 &   414 &   494 & 365 & 2.3e-5 &   20.0 &      385 & 438.03 & 6 \\
 pm1d\_100.9 &    100 &    4901 &   13 &   355 &   408 & 404 & 7.8e-5 &   20.0 &      405 & 470.66 & 5 \\
  pm1d\_80.0 &     80 &    3128 &   16 &   395 &   445 & 214 & 2.0e-5 &   25.0 &      227 & 269.97 & 5 \\
  pm1d\_80.1 &     80 &    3128 &   25 &   431 &   497 & 239 & 7.0e-5 &    2.5 &      245 & 292.60 & 5 \\
  pm1d\_80.2 &     80 &    3128 &   25 &   320 &   354 & 270 & 9.4e-5 &    2.5 &      284 & 325.99 & 6 \\
  pm1d\_80.3 &     80 &    3128 &   13 &   405 &   475 & 283 & 4.4e-5 &   25.0 &      291 & 331.31 & 5 \\
  pm1d\_80.4 &     80 &    3128 &   11 &   307 &   355 & 248 & 2.0e-5 &   25.0 &      251 & 290.75 & 5 \\
  pm1d\_80.5 &     80 &    3128 &   28 &   342 &   381 & 233 & 8.9e-5 &    2.5 &      242 & 290.50 & 5 \\
  pm1d\_80.6 &     80 &    3128 &   15 &   376 &   430 & 192 & 5.5e-5 &   25.0 &      205 & 253.56 & 5 \\
  pm1d\_80.7 &     80 &    3128 &   11 &   684 &   895 & 244 & 5.0e-5 &   25.0 &      249 & 292.52 & 4 \\
  pm1d\_80.8 &     80 &    3128 &   27 &   343 &   373 & 288 & 8.5e-5 &    2.5 &      293 & 329.98 & 5 \\
  pm1d\_80.9 &     80 &    3128 &   24 &   310 &   346 & 258 & 9.0e-5 &    2.5 &      258 & 294.31 & 5 \\
 pm1s\_100.0 &    100 &     495 &   16 &   252 &   289 & 118 & 2.0e-6 &   20.0 &      127 & 143.23 & 6 \\
 pm1s\_100.1 &    100 &     495 &   17 &   329 &   376 & 119 & 1.9e-6 &   20.0 &      126 & 144.61 & 5 \\
 pm1s\_100.2 &    100 &     495 &   17 &   281 &   312 & 116 & 9.5e-5 &   20.0 &      125 & 140.23 & 5 \\
 pm1s\_100.3 &    100 &     495 &   17 &   357 &   429 & 100 & 8.8e-5 &    2.0 &      111 & 130.09 & 6 \\
 pm1s\_100.4 &    100 &     495 &   16 &   298 &   369 & 127 & 9.0e-5 &    2.0 &      128 & 145.61 & 6 \\
 pm1s\_100.5 &    100 &     495 &   17 &   335 &   367 & 117 & 9.0e-5 &    2.0 &      128 & 144.66 & 5 \\
 pm1s\_100.6 &    100 &     495 &   16 &   252 &   274 & 118 & 9.0e-5 &   20.0 &      122 & 139.91 & 6 \\
 pm1s\_100.7 &    100 &     495 &   17 &   422 &   491 & 105 & 6.6e-5 &    2.0 &      112 & 126.80 & 5 \\
 pm1s\_100.8 &    100 &     495 &   16 &   288 &   329 & 118 & 7.9e-5 &    2.0 &      120 & 135.86 & 6 \\
 pm1s\_100.9 &    100 &     495 &   17 &   248 &   273 & 125 & 7.0e-7 &   20.0 &      127 & 143.52 & 5 \\
  pm1s\_80.0 &     80 &     316 &   14 &   291 &   342 & 73 & 4.6e-5 &    2.5 &       79 & 90.29 & 4 \\
  pm1s\_80.1 &     80 &     316 &   14 &   255 &   287 & 81 & 6.2e-5 &    2.5 &       85 & 96.18 & 4 \\
  pm1s\_80.2 &     80 &     316 &   14 &   416 &   492 & 80 & 7.6e-5 &    2.5 &       82 & 94.02 & 6 \\
  pm1s\_80.3 &     80 &     316 &   13 &   231 &   265 & 77 & 4.8e-5 &    2.5 &       81 & 92.14 & 5 \\
  pm1s\_80.4 &     80 &     316 &   15 &   298 &   374 & 62 & 5.4e-5 &    2.5 &       70 & 82.06 & 5 \\
  pm1s\_80.5 &     80 &     316 &   13 &   219 &   243 & 86 & 7.9e-5 &    2.5 &       87 & 98.69 & 4 \\
  pm1s\_80.6 &     80 &     316 &   13 &   246 &   297 & 70 & 6.1e-5 &    2.5 &       73 & 85.69 & 5 \\
  pm1s\_80.7 &     80 &     316 &   15 &   346 &   411 & 81 & 4.4e-5 &    2.5 &       83 & 95.45 & 5 \\
  pm1s\_80.8 &     80 &     316 &   13 &   341 &   394 & 79 & 7.5e-5 &    2.5 &       81 & 95.47 & 5 \\
  pm1s\_80.9 &     80 &     316 &   15 &   198 &   230 & 66 & 4.3e-5 &    2.5 &       70 & 82.00 & 5 \\
 pw01\_100.0 &    100 &     495 &   16 &   352 &   421 & 1963 & 9.7e-5 &   20.0 &     2019 & 2125.43 & 5 \\
 pw01\_100.1 &    100 &     495 &   16 &   438 &   538 & 2025 & 5.9e-5 &   20.0 &     2060 & 2161.61 & 5 \\
 pw01\_100.2 &    100 &     495 &   17 &   365 &   432 & 2009 & 3.4e-5 &   20.0 &     2032 & 2135.62 & 5 \\
 pw01\_100.3 &    100 &     495 &   16 &   357 &   443 & 2053 & 6.7e-5 &   20.0 &     2067 & 2167.93 & 5 \\
 pw01\_100.4 &    100 &     495 &   15 &   361 &   432 & 1990 & 4.5e-5 &   20.0 &     2039 & 2116.66 & 5 \\
 pw01\_100.5 &    100 &     495 &   16 &   371 &   420 & 2068 & 8.9e-5 &   20.0 &     2108 & 2195.59 & 5 \\
 pw01\_100.6 &    100 &     495 &   17 &   387 &   442 & 2010 & 5.8e-5 &   20.0 &     2032 & 2135.28 & 5 \\
 pw01\_100.7 &    100 &     495 &   17 &   260 &   304 & 2068 & 6.0e-5 &   20.0 &     2074 & 2182.48 & 5 \\
 pw01\_100.8 &    100 &     495 &   15 &   253 &   293 & 2022 & 7.0e-5 &   20.0 &     2022 & 2102.02 & 6 \\
 pw01\_100.9 &    100 &     495 &   17 &   612 &   748 & 1986 & 5.5e-5 &   20.0 &     2005 & 2114.21 & 5 \\
 pw05\_100.0 &    100 &    2475 &   24 &   334 &   405 & 8118 & 8.9e-5 &   20.0 &     8190 & 8427.71 & 6 \\
 pw05\_100.1 &    100 &    2475 &   12 &   299 &   360 & 7954 & 5.3e-5 &  200.0 &     8045 & 8260.33 & 5 \\
 pw05\_100.2 &    100 &    2475 &   27 &   403 &   519 & 7915 & 7.6e-5 &   20.0 &     8039 & 8271.30 & 6 \\
 pw05\_100.3 &    100 &    2475 &   29 &   446 &   576 & 8002 & 6.7e-5 &   20.0 &     8139 & 8320.32 & 6 \\
 pw05\_100.4 &    100 &    2475 &   23 &   472 &   600 & 8024 & 7.3e-5 &   20.0 &     8125 & 8350.81 & 6 \\
 pw05\_100.5 &    100 &    2475 &   24 &   939 &  1193 & 8149 & 6.7e-5 &   20.0 &     8169 & 8373.47 & 5 \\
 pw05\_100.6 &    100 &    2475 &   20 &   349 &   424 & 8133 & 6.4e-5 &   20.0 &     8217 & 8467.12 & 6 \\
 pw05\_100.7 &    100 &    2475 &   21 &   365 &   445 & 8186 & 7.3e-5 &   20.0 &     8249 & 8487.57 & 5 \\
 pw05\_100.8 &    100 &    2475 &   23 &   371 &   466 & 8051 & 7.8e-5 &   20.0 &     8199 & 8382.98 & 5 \\
 pw05\_100.9 &    100 &    2475 &   22 &   289 &   358 & 8076 & 9.4e-5 &   20.0 &     8099 & 8304.87 & 5 \\
 pw09\_100.0 &    100 &    4455 &   30 &   398 &   504 & 13497 & 7.9e-5 &   20.0 &    13585 & 13805.97 & 6 \\
 pw09\_100.1 &    100 &    4455 &   31 &   491 &   603 & 13357 & 9.8e-5 &   20.0 &    13417 & 13643.51 & 6 \\
 pw09\_100.2 &    100 &    4455 &   38 &   485 &   665 & 13324 & 8.7e-5 &   20.0 &    13461 & 13645.66 & 6 \\
 pw09\_100.3 &    100 &    4455 &   32 &   432 &   543 & 13554 & 9.8e-5 &   20.0 &    13656 & 13842.17 & 6 \\
 pw09\_100.4 &    100 &    4455 &   14 &   350 &   440 & 13480 & 3.1e-5 &  200.0 &    13514 & 13712.77 & 5 \\
 pw09\_100.5 &    100 &    4455 &   34 &   386 &   505 & 13487 & 9.1e-5 &   20.0 &    13574 & 13790.21 & 5 \\
 pw09\_100.6 &    100 &    4455 &   31 &   353 &   448 & 13587 & 8.4e-5 &   20.0 &    13640 & 13835.51 & 5 \\
 pw09\_100.7 &    100 &    4455 &   36 &   385 &   486 & 13451 & 9.3e-5 &   20.0 &    13501 & 13712.94 & 6 \\
 pw09\_100.8 &    100 &    4455 &   35 &   401 &   516 & 13516 & 8.3e-5 &   20.0 &    13593 & 13804.64 & 6 \\
 pw09\_100.9 &    100 &    4455 &   24 &   358 &   438 & 13523 & 8.8e-5 &   20.0 &    13658 & 13864.02 & 5 \\
  w01\_100.0 &    100 &     495 &   17 &   335 &   389 & 624 & 4.2e-5 &   20.0 &      651 & 740.89 & 5 \\
  w01\_100.1 &    100 &     495 &   17 &   467 &   560 & 671 & 5.4e-5 &   20.0 &      719 & 811.83 & 5 \\
  w01\_100.2 &    100 &     495 &   17 &   345 &   406 & 642 & 4.0e-5 &   20.0 &      676 & 781.40 & 6 \\
  w01\_100.3 &    100 &     495 &   18 &   365 &   449 & 793 & 6.8e-5 &   20.0 &      813 & 910.42 & 5 \\
  w01\_100.4 &    100 &     495 &   15 &   222 &   259 & 618 & 5.3e-5 &   20.0 &      668 & 747.01 & 5 \\
  w01\_100.5 &    100 &     495 &   16 &   349 &   405 & 610 & 7.9e-5 &   20.0 &      643 & 737.14 & 5 \\
  w01\_100.6 &    100 &     495 &   17 &   373 &   435 & 627 & 4.2e-5 &   20.0 &      654 & 740.11 & 5 \\
  w01\_100.7 &    100 &     495 &   17 &   451 &   531 & 667 & 5.4e-5 &   20.0 &      725 & 828.69 & 4 \\
  w01\_100.8 &    100 &     495 &   16 &   331 &   395 & 713 & 7.2e-5 &   20.0 &      721 & 792.74 & 4 \\
  w01\_100.9 &    100 &     495 &   16 &   302 &   369 & 721 & 7.9e-5 &   20.0 &      729 & 816.08 & 5 \\
  w05\_100.0 &    100 &    2475 &   23 &   424 &   537 & 1612 & 5.7e-5 &   20.0 &     1646 & 1918.05 & 7 \\
  w05\_100.1 &    100 &    2475 &   23 &   502 &   605 & 1524 & 6.0e-5 &   20.0 &     1606 & 1857.11 & 5 \\
  w05\_100.2 &    100 &    2475 &   22 &   385 &   476 & 1815 & 5.9e-5 &   20.0 &     1902 & 2182.08 & 5 \\
  w05\_100.3 &    100 &    2475 &   21 &   362 &   429 & 1617 & 7.1e-5 &   20.0 &     1627 & 1893.10 & 5 \\
  w05\_100.4 &    100 &    2475 &   22 &   534 &   654 & 1512 & 5.7e-5 &   20.0 &     1546 & 1838.11 & 5 \\
  w05\_100.5 &    100 &    2475 &   22 &   391 &   492 & 1491 & 5.7e-5 &   20.0 &     1581 & 1871.73 & 5 \\
  w05\_100.6 &    100 &    2475 &   22 &   476 &   595 & 1367 & 8.7e-5 &   20.0 &     1479 & 1747.94 & 6 \\
  w05\_100.7 &    100 &    2475 &   22 &   402 &   501 & 1896 & 8.4e-5 &   20.0 &     1987 & 2248.94 & 5 \\
  w05\_100.8 &    100 &    2475 &   21 &   386 &   468 & 1263 & 6.9e-5 &   20.0 &     1311 & 1598.22 & 5 \\
  w05\_100.9 &    100 &    2475 &   20 &   304 &   382 & 1747 & 7.9e-5 &   20.0 &     1752 & 2017.39 & 4 \\
  w09\_100.0 &    100 &    4455 &   27 &   402 &   525 & 2011 & 6.6e-5 &   20.0 &     2121 & 2500.30 & 5 \\
  w09\_100.1 &    100 &    4455 &   25 &   392 &   512 & 2085 & 8.5e-5 &   20.0 &     2096 & 2511.47 & 5 \\
  w09\_100.2 &    100 &    4455 &   26 &   439 &   547 & 2675 & 9.5e-5 &   20.0 &     2738 & 3130.01 & 6 \\
  w09\_100.3 &    100 &    4455 &   13 &   372 &   464 & 1958 & 1.8e-5 &  200.0 &     1990 & 2333.06 & 6 \\
  w09\_100.4 &    100 &    4455 &   26 &   477 &   606 & 1921 & 8.4e-5 &   20.0 &     2033 & 2424.99 & 6 \\
  w09\_100.5 &    100 &    4455 &   25 &   429 &   558 & 2357 & 9.5e-5 &   20.0 &     2433 & 2733.65 & 4 \\
  w09\_100.6 &    100 &    4455 &   26 &   385 &   503 & 2172 & 6.6e-5 &   20.0 &     2220 & 2552.12 & 6 \\
  w09\_100.7 &    100 &    4455 &   28 &   636 &   823 & 2122 & 7.0e-5 &   20.0 &     2252 & 2639.74 & 6 \\
  w09\_100.8 &    100 &    4455 &   15 &   447 &   575 & 1665 & 7.8e-5 &  200.0 &     1843 & 2213.13 & 6 \\
  w09\_100.9 &    100 &    4455 &   24 &   367 &   446 & 2041 & 1.0e-4 &   20.0 &     2043 & 2409.78 & 6 \\

\end{longtable}

\end{scriptsize}

\begin{scriptsize}
	%% Put the caption in \normalsize
	\def\tablename{\normalsize Table}

\begin{longtable}[c]{rrr|rrrrrr|r}
\caption{Numerical results for MAXCUT problems, \texttt{ising} collection.}\label{Table:MAXCUT2}\\ \hline
 probl. & vert. & edges &  $ k $ & $ j $ & $f$-ev. &  $f_{\textup{ALM}}$\hspace{0.5mm} & feas. & $ \rho $ & $f_{\textup{opt}}$ \\ \hline
\endfirsthead
\caption{Numerical results for MAXCUT problems, \texttt{ising} collection 
(continued).} \\
 \hline
 probl. & vert. & edges &  $ k $ & $ j $ & $f$-ev. &  $f_{\textup{ALM}}$\hspace{0.5mm} & feas. & $ \rho $ & $f_{\textup{opt}}$ \\ \hline
\endhead
\hline
\endfoot
\hline
\endlastfoot
ising2.5-100\_5555 &    100 &    4950 &   40 &   687 &  1619 &  2406360.19 & 9.8e-5 & 20000.0 &  2460049 \\
ising2.5-100\_6666 &    100 &    4950 &   17 &   472 &   794 &  1978286.63 & 3.3e-5 & 200000.0 &  2031217 \\
ising2.5-100\_7777 &    100 &    4950 &   36 &   590 &  1370 &  3333168.11 & 8.0e-5 & 20000.0 &  3363230 \\
ising2.5-150\_5555 &    150 &   11175 &   18 &   739 &  1198 &  4315079.10 & 9.2e-5 & 133333.3 &  4363532 \\
ising2.5-150\_6666 &    150 &   11175 &   25 &   959 &  1648 &  4041057.48 & 7.2e-5 & 133333.3 &  4057153 \\
ising2.5-150\_7777 &    150 &   11175 &   31 &   734 &  1389 &  4224911.34 & 5.5e-5 & 133333.3 &  4243269 \\
ising2.5-200\_5555 &    200 &   19900 &   24 &   651 &  1148 &  6267758.47 & 4.5e-5 & 100000.0 &  6294701 \\
ising2.5-200\_6666 &    200 &   19900 &   19 &   630 &  1005 &  6752676.03 & 6.9e-5 & 100000.0 &  6795365 \\
ising2.5-200\_7777 &    200 &   19900 &   19 &   586 &   948 &  5506984.49 & 9.4e-5 & 100000.0 &  5568272 \\
ising2.5-250\_5555 &    250 &   31125 &   22 &   692 &  1154 &  7864741.24 & 8.1e-5 & 80000.0 &  7919449 \\
ising2.5-250\_6666 &    250 &   31125 &   19 &   896 &  1437 &  6852662.07 & 8.8e-5 & 80000.0 &  6925717 \\
ising2.5-250\_7777 &    250 &   31125 &   22 &   624 &  1030 &  6462343.01 & 6.9e-5 & 80000.0 &  6596797 \\
ising2.5-300\_5555 &    300 &   44850 &   22 &  1076 &  1705 &  8523955.32 & 5.9e-5 & 66666.7 &  8579363 \\
ising2.5-300\_6666 &    300 &   44850 &   21 &  1304 &  2075 &  9058514.43 & 6.7e-5 & 66666.7 &  9102033 \\
ising2.5-300\_7777 &    300 &   44850 &   25 &   887 &  1420 &  8168651.28 & 6.0e-5 & 66666.7 &  8323804 \\
ising3.0-100\_5555 &    100 &    4950 &   38 &   711 &  1690 &  2431408.96 & 9.8e-5 & 20000.0 &  2448189 \\
ising3.0-100\_6666 &    100 &    4950 &   35 &   510 &  1238 &  1975552.41 & 9.2e-5 & 20000.0 &  1984099 \\
ising3.0-100\_7777 &    100 &    4950 &   36 &   694 &  1505 &  3327994.14 & 8.0e-5 & 20000.0 &  3335814 \\
ising3.0-150\_5555 &    150 &   11175 &   15 &   456 &   761 &  4246560.88 & 8.2e-5 & 133333.3 &  4279261 \\
ising3.0-150\_6666 &    150 &   11175 &   26 &   647 &  1189 &  3935446.60 & 6.7e-5 & 133333.3 &  3949317 \\
ising3.0-150\_7777 &    150 &   11175 &   31 &   840 &  1582 &  4205069.88 & 5.3e-5 & 133333.3 &  4211158 \\
ising3.0-200\_5555 &    200 &   19900 &   20 &   592 &  1014 &  6209779.63 & 8.1e-5 & 100000.0 &  6215531 \\
ising3.0-200\_6666 &    200 &   19900 &   18 &   599 &   973 &  6697966.13 & 7.0e-5 & 100000.0 &  6756263 \\
ising3.0-200\_7777 &    200 &   19900 &   18 &   721 &  1120 &  5529450.26 & 8.1e-5 & 100000.0 &  5560824 \\
ising3.0-250\_5555 &    250 &   31125 &   22 &   753 &  1247 &  7790361.34 & 5.3e-5 & 80000.0 &  7823791 \\
ising3.0-250\_6666 &    250 &   31125 &   20 &  1003 &  1642 &  6879016.15 & 7.2e-5 & 80000.0 &  6903351 \\
ising3.0-250\_7777 &    250 &   31125 &   20 &  1670 &  2935 &  6287504.49 & 6.3e-5 & 80000.0 &  6418276 \\
ising3.0-300\_5555 &    300 &   44850 &   20 &   942 &  1598 &  8426148.11 & 5.6e-5 & 66666.7 &  8493173 \\
ising3.0-300\_6666 &    300 &   44850 &   23 &   928 &  1515 &  8907934.89 & 4.9e-5 & 66666.7 &  8915110 \\
ising3.0-300\_7777 &    300 &   44850 &   23 &   816 &  1375 &  8169591.67 & 5.4e-5 & 66666.7 &  8242904 \\
 t2g10\_5555 &    100 &     200 &   18 &   411 &   775 &  5778570.15 & 5.4e-5 & 200000.0 &  6049461 \\
 t2g10\_6666 &    100 &     200 &   18 &   463 &   820 &  5503220.56 & 6.9e-5 & 200000.0 &  5757868 \\
 t2g10\_7777 &    100 &     200 &   19 &   492 &   882 &  6261175.46 & 4.1e-5 & 200000.0 &  6509837 \\
 t2g15\_5555 &    225 &     450 &   25 &   601 &  1108 & 14446186.29 & 7.6e-5 & 88888.9 & 15051133 \\
 t2g15\_6666 &    225 &     450 &   27 &   605 &  1191 & 15454604.55 & 6.0e-5 & 88888.9 & 15763716 \\
 t2g15\_7777 &    225 &     450 &   27 &  1351 &  2153 & 14798901.82 & 7.9e-5 & 88888.9 & 15269399 \\
 t2g20\_5555 &    400 &     800 &   22 &   919 &  1408 & 24487271.71 & 4.6e-5 & 500000.0 & 24838942 \\
 t2g20\_6666 &    400 &     800 &   19 &  1131 &  1782 & 28725534.39 & 4.2e-5 & 500000.0 & 29290570 \\
 t2g20\_7777 &    400 &     800 &   19 &  1026 &  1515 & 27294253.39 & 7.4e-5 & 500000.0 & 28349398 \\
  t3g5\_5555 &    125 &     375 &   22 &   496 &   951 & 10843693.83 & 9.4e-5 & 160000.0 & 10933215 \\
  t3g5\_6666 &    125 &     375 &   22 &   724 &  1187 & 11358698.15 & 8.4e-5 & 160000.0 & 11582216 \\
  t3g5\_7777 &    125 &     375 &   25 &   668 &  1225 & 11196295.63 & 9.1e-5 & 160000.0 & 11552046 \\
  t3g6\_5555 &    216 &     648 &   30 &  1092 &  1819 & 17046996.57 & 7.2e-5 & 92592.6 & 17434469 \\
  t3g6\_6666 &    216 &     648 &   26 &   671 &  1252 & 20014468.29 & 8.9e-5 & 92592.6 & 20217380 \\
  t3g6\_7777 &    216 &     648 &   32 &  1403 &  2282 & 18487004.21 & 8.3e-5 & 92592.6 & 19475011 \\
  t3g7\_5555 &    343 &    1029 &   20 &  1201 &  1665 & 26773833.26 & 9.4e-5 & 583090.4 & 28302918 \\
  t3g7\_6666 &    343 &    1029 &   19 &   912 &  1342 & 32934614.99 & 8.0e-5 & 583090.4 & 33611981 \\
  t3g7\_7777 &    343 &    1029 &   21 &   871 &  1332 & 27953083.17 & 1.9e-5 & 583090.4 & 29118445 \\

     \hline
\end{longtable}

\end{scriptsize}
\fi
\end{document}